\documentclass[11pt]{article}

\usepackage[utf8]{inputenc}
\usepackage[T1]{fontenc}
\usepackage[english]{babel}

\usepackage{fullpage}

\usepackage{amsmath,amsthm,amsfonts,amssymb,upgreek,eufrak,dsfont}
\usepackage{mathtools}
\mathtoolsset{showonlyrefs}
\usepackage{mathabx}

\usepackage{graphicx}
\usepackage{caption}
\usepackage{float}
\usepackage{tikz}
\usepackage{tikz-cd}
\usepackage{tkz-tab}
\usepackage{ytableau}

\usepackage{needspace}
\usepackage{url}
\usepackage{cite}
\usepackage{braket}
\usepackage{physics}

\usepackage[
pdftex,
colorlinks=true,
linkcolor=blue,
citecolor=blue,
bookmarks=true,
bookmarksopen=true,
pagebackref=true
]{hyperref}

\newtheorem{theo}{Theorem}[section]
\newtheorem{prop}{Proposition}[section]
\newtheorem{lem}{Lemma}[section]
\newtheorem{coro}{Corollary}[section]
\newtheorem{defin}{Definition}[section]
\newtheorem{rem}{Remark}[section]
\newtheorem{ass}{Assumption}[section]

\newcommand{\bgamma}{\boldsymbol{\gamma}}
\newcommand{\bc}{\boldsymbol{c}}

\newcommand{\bCW}{\boldsymbol{\mathcal{W}}}
\newcommand{\bX}{\boldsymbol{X}}
\newcommand{\bx}{\boldsymbol{x}}
\newcommand{\bTheta}{\boldsymbol{\Theta}}
\newcommand{\btheta}{\boldsymbol{\theta}}
\newcommand{\bxi}{\boldsymbol{\xi}}
\newcommand{\bXi}{\boldsymbol{\Xi}}
\newcommand{\bP}{\boldsymbol{P}}
\newcommand{\br}{\boldsymbol{r}}
\newcommand{\bCC}{\boldsymbol{\mathcal{C}}}

\newcommand{\sss}[1]{\scriptscriptstyle{#1}}

\DeclareRobustCommand\longtwoheadrightarrow
{\relbar\joinrel\twoheadrightarrow}

\usepackage[draft]{fixme}
\fxusetheme{color}
\FXRegisterAuthor{yo}{ayo}{Yo}

\usepackage{todonotes}

\usepackage[blocks]{authblk}

\author{Yoann Offret}
\affil{\small Institut de Mathématiques de Bourgogne (IMB) - UMR CNRS 5584\\
	Université Bourgogne Europe, 21000 Dijon, France}

\date{}

\title{\bf \centering \Large From maximal entropy exclusion process to unitary Dyson Brownian motion and free unitary hydrodynamics
}

\author{Yoann Offret} 

\usepackage{tocloft}

\begin{document}

\maketitle

\vskip-40pt

\noindent
{\small {\bf Abstract.} We investigate the \emph{Maximal Entropy Simple Symmetric Exclusion Process} (MESSEP) on a discrete ring with $L$ sites and $N$ indistinguishable particles. Its eigenfunctions are Schur polynomials evaluated at the $L$-th roots of unity, yielding an explicit spectral decomposition. The analysis relies on this eigenstructure and on the link between Schur polynomials and irreducible characters of the symmetric group, which forms the core algebraic tool for the scaling limits.

In the low-density regime, where $N$ is fixed and $L$ tends to infinity, the rescaled dynamics converge to the \emph{Unitary Dyson Brownian Motion} (UDBM). The electrostatic repulsion then emerges as an entropic force, providing a canonical microscopic derivation of the UDBM.

In the hydrodynamic regime, where $N$ is equivalent to  $\alpha L$ with  $\alpha\in(0,1)$, the empirical measure converges to a density solving a nonlinear, nonlocal transport equation. Its moment generating function satisfies a complex Burgers-type equation. As $\alpha$ tends to $0$, this equation coincides with that governing the spectral distribution of the \emph{Free Unitary Brownian Motion} (FUBM), thereby bridging discrete entropic exclusion dynamics and free unitary hydrodynamics.

Overall, the MESSEP provides a unified canonical discrete framework connecting unitary Dyson motion and free unitary Brownian motion through nonlinear hydrodynamic limits, with Schur and character theory as the central algebraic structure.

\vspace{20pt}

\noindent
\rule{\linewidth}{1.5pt}

\vspace{5pt}


\vspace{5pt}


\noindent
\rule{\linewidth}{1.5pt}
}

\newpage  

{\small
\tableofcontents
}

\newpage  

\section{Introduction}

\setcounter{equation}{0}

The principle of maximum entropy selects probability distributions that encode the imposed constraints and nothing more. Born in statistical physics, it explains the emergence of equilibrium and bridges microscopic configurations with macroscopic laws. Entropy quantifies the combinatorial complexity of states compatible with macroscopic observables. 

In probability theory, it provides a canonical rule for modeling randomness under structural constraints. Across disciplines, entropy maximization serves not merely as a variational principle, but as a guideline for  natural models of complex systems.

The present work originates in Section~6 of \cite{DudaPhd}, where possible many-body extensions of the \emph{Maximal Entropy Random Walks} (MERWs) were discussed informally. In the setting relevant here -- a symmetric irreducible graph $G=(V,E)$ -- the MERW is the unique Markov chain $(X_n)_{n\geq 0}$ on $V$ compatible with the edge structure $E$ which maximizes the asymptotic entropy rate
\begin{equation*}
h=\lim_{n\to\infty}\frac{H(X_0,\cdots,X_n)}{n}=-\sum_{x\in V}\mu(x)\sum_{y\in V} P(x,y)\ln(P(x,y)),
\end{equation*}
where $H(Y)$ denotes the Shannon entropy of a discrete random variable $Y$, $P$ is the transition kernel of the Markov chain, and $\mu$ its invariant probability measure. It turns out that  $P$ admits a Doob $h$-transform structure built from the Perron--Frobenius eigenelements of the adjacency matrix $A$ of $G$, which also determine the entropy rate $h$ and the invariant probability measure $\mu$.  Denoting by $\rho$ the spectral radius of $A$ and by $\psi$ the associated positive eigenvector:
\begin{equation}\label{eq:MERW}
P(x,y)=A(x,y)\frac{\psi(y)}{\rho\,\psi(x)},\quad \mu(x)=\psi^2(x) \quad\text{and}\quad h=\ln(\rho).
\end{equation}
The positive function $\psi$, solution of $A\psi=\rho\psi$, can be viewed as the ground state of the discrete Schr\"odinger equation $H\psi=-\Delta\psi+V\psi$, where $V(x)=-\mathrm{deg}(x)$ and $\Delta$ is the graph Laplacian. 

We also refer  \cite{BurdaReview} for a broader account of MERWs from the viewpoint of quantum mechanics and statistical physics.

\subsection{Overview of the main results and  ideas}

We consider $N$ indistinguishable particles on the ring $\mathbb T_L=\{0,\ldots,L-1\}$, identified with $\mathbb Z_L=\mathbb Z/L\mathbb Z$ or with $\mathbb U_L\subset\mathbb C$, the $L$-th roots of unity. Particles jump to nearest empty sites, one at a time. This defines an undirected graph on the configuration space $\mathcal C_{L,N}$, identified either with the subsets of $\mathbb T_L$ of size $N$, or with strictly increasing $N$-tuples in $\mathbb T_L$ (see Figure~\ref{ring}).

\begin{figure}[!h]
	\centering
	\includegraphics[width=4.5cm]{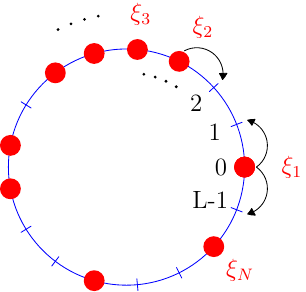}
	\captionsetup{width=11.5cm}
	\caption{\small  A configuration $\xi \in \mathcal C_{L,N}$ with $\xi_1 < \cdots < \xi_N$ on the ring $\mathbb{T}_L$, identified with $\mathbb U_L$. Particle positions are ordered counterclockwise. Arrows indicate possible particle transitions.}
	\label{ring}
\end{figure}

We study the associated MERW $(\Xi(n))_{n\geq 0}$, called the \emph{Maximum Entropy Simple Symmetric Exclusion Process} (MESSEP) on the circle (Definition~\ref{def:MERW}).

\bigskip

\noindent
{\textbf 1.} Section~\ref{sec:1} introduces the MESSEP kernel $P$ and tagged versions of this process (Definition~\ref{defmessep}), which are used to establish the scaling limits of Sections~\ref{sec:low} and~\ref{sec:hydro}. The MESSEP coincides with independent simple symmetric random walks on the circle conditioned  never to collide (Proposition~\ref{noncollinding}). We describe the spectrum of the graph $\mathcal C_{L,N}$ and of the MESSEP kernel  $P$, proving that $P$ admits an $\mathbb L^2(\mu)$-eigenbasis of Schur polynomials $s_\lambda$ evaluated at $N$-tuples of $L$-th roots of unity, and we describe the correspondence between partitions $\lambda$ and configurations $\xi\in\mathcal C_{L,N}$ (Proposition~\ref{prop:schureigen}, Remark \ref{rem:corres}). Spectral gap estimates (Proposition~\ref{gapestimate}) are obtained in the low-density limit ($N$ fixed, $L\to\infty$) and in the hydrodynamic limit ($N\sim \alpha L$, $\alpha\in(0,1)$). Proofs are deferred to Section~\ref{sec:proofsec1}.

\bigskip

\noindent
{\bf 2.} Section~\ref{sec:low} states our main result in the low-density regime (Theorem~\ref{scalingthm1}): a functional scaling limit of the MESSEP toward the \emph{Unitary Dyson Brownian Motion} (UDBM). For convenience, we give below a simplified version. 
\begin{theo}\label{thm:0}
	Let $(\Xi(t))_{t\geq 0}$ be the continuous-time linear interpolation of the MESSEP. Under standard assumptions, the following functional convergence holds:
	\begin{equation}\label{eq:thmconv}
	\left(\left\{e^{2i\pi\, {\Xi_1(L^2 t)}/{L}},\cdots, e^{2i\pi\, {\Xi_N(L^2 t)}/{L}}\right\}\right)_{t\geq 0}
	\xRightarrow[L\to\infty]{}
	\left(\left\{e^{i\boldsymbol{X}_1(t)},\cdots, e^{i\boldsymbol{X}_N(t)}\right\}\right)_{t\geq 0},
	\end{equation}
where $(\boldsymbol{X}(t))_{t\geq 0}$ is the solution of the system of stochastic differential equations
\begin{equation}\label{Dyson10}
\forall\,1\leq i\leq N,\quad
d\boldsymbol{X}_i(t)
= {\frac{2\pi}{\sqrt N}}\,dB_i(t)
+ \frac{2\pi^2}{N} \sum_{\substack{1 \leq j \leq N \\ j \neq i}}
\cot\!\left( \frac{\boldsymbol{X}_i(t) - \boldsymbol{X}_j(t)}{2} \right) dt.
\end{equation}

\end{theo}

As a matter of fact, this convergence holds at the level of a suitable lift of the MESSEP (introduced in Section~\ref{sec:1}) toward the UDBM, and \eqref{eq:thmconv} follows by projection. Moreover, the orthonormal Schur eigenstructure of the MESSEP yields a Hilbert--Schur expansion of the Markov semigroup $(\bP_t)_{t\geq 0}$ associated with $(\bXi(t))_{t\geq 0}$, where $(\bXi(t))_{t\geq 0}$ is the projection of the UDBM defined in \eqref{Dyson10}  onto the circle, modulo permutations of the particles
 (see Proposition~\ref{orthonormal} and Theorem~\ref{spectraldecompodyson}). We also refer to Proposition~\ref{prop:schureigen}, where the Schur representation of the eigenfunctions slightly differs from the discrete case in Proposition~\ref{Schur}. As a consequence, the semigroup is infinitely smoothing on $\mathbb L^2(\boldsymbol \mu)$ and converges exponentially fast in $\mathbb L^2(\boldsymbol \mu)$ to its invariant measure $\boldsymbol \mu$ (see Corollary~\ref{coro:convergenceL2}). Proofs and the full framework are deferred to Section~\ref{sec:sec3proof}.

Tightness for the MESSEP and its extended dynamics is established through classical martingale arguments, combined with the fact that Schur polynomials form a basis of symmetric polynomials. Viewing the particles as roots of a polynomial on the unit circle and expressing elementary symmetric polynomials through Schur eigenfunctions, we control fluctuations via Rouché’s theorem. The limit is then identified from its marginals using the $\mathbb{L}^2$-spectral decomposition of the UDBM semigroup, inherited from that of the MESSEP.  Arguments that are typically analytic for interacting diffusions thus follow transparently here from the Schur eigenstructure.

\bigskip

\noindent
{\bf 3.} Section~\ref{sec:hydro} states the main result of this paper (Theorem~\ref{thm:hydro}) on the asymptotic behaviour of the empirical measures of the rescaled continuous-time MESSEP,
\begin{equation}\label{eq:occupmeasure0}
\upsilon_{L,N}(t,dx)=\frac{1}{N}\sum_{k=1}^N\delta_{{2\pi\,\Xi_k(L^2 t)}/{L}}(dx).
\end{equation}

\medskip

\noindent
{\it A nonlinear and nonlocal conservation law for the density.} We give below a streamlined formulation, split into two statements, together with concise versions of the main related results.

\begin{theo}\label{thm0}
	Assume $N\sim \alpha L$ with $\alpha\in(0,1)$ and that $\upsilon_{L,N}(0,dx)$ converges in probability (see Definition~\ref{def:cvgprob}) to a density $f_0$ on $\mathbb R/2\pi\mathbb Z$. Then, for all $t\geq 0$, $\upsilon_{L,N}(t,dx)$ converges in the same sense to a density $f(t,x)$ solving, in the weak sense specified later,
	\begin{equation}\label{PDE0}
	\frac{\partial f}{\partial t}+\frac{1}{\alpha\sin(\pi\alpha)} \frac{\partial \left(\sin(2\pi^2 \alpha \,f)\sinh(2\pi^2\alpha\, \mathcal Hf)\right)}{\partial x}
	=0,
	\end{equation}
	where $\mathcal H$ denotes the spatial Hilbert transform on the circle.
\end{theo}

We develop a regularity theory for the Partial Differential Equation (PDE)  \eqref{PDE0}. We refer to Theorem~\ref{prop:regularityPDE}, as well as to the discussion below. The main features are:

\begin{itemize}
	\item Singularities originate from extremal (saturated) regions of the initial profile $f_0$, namely where $f_0$ vanishes (macroscopic voids) or reaches its maximum value $1/(2\pi\alpha)$ (particle congestion, see Remark~\ref{rem:packed}), and can propagate only through extremal regions of $f(t,x)$.

	\item There exists $t^\ast>0$ such that $f(t,x)$ stays away from its extremal values for $t> t^\ast$, yielding an analytic strong solution on $(t^\ast,\infty)\times \mathbb R/2\pi\mathbb Z$.
	\item For any initial density profile, the solution converges exponentially fast to the constant equilibrium profile $1/(2\pi)$.
\end{itemize}

Notably, even an entire analytic initial profile may develop derivative discontinuities (see Section~\ref{sec:periodicdens}). The proof relies on the method of moments and the Schur eigenstructure of the MESSEP kernel. To compute $\mathfrak m_n(t)$, the asymptotic expectation of the $n$th complex moment of \eqref{eq:occupmeasure0}, we use the Frobenius character formula \eqref{eq:frobenius}, relating power-sums $p_\pi$, Schur polynomials $s_\lambda$, and irreducible characters $\chi_\pi^\lambda$. For $p_n$, hook partitions $(n-k,1^k)$ (see Figure \ref{fig:confighook}) and the associated eigenfunctions $s_{\{n|k\}}$ play a central role. The argument relies on a precise Taylor expansion of the MESSEP eigenvalues for hook partitions (Lemma~\ref{lem:eigenhookdevasym}) and new character identities (Lemma~\ref{lem:characteridentities}), yielding the combinatorial formula \eqref{eq:momentasymptotic}.

To prove convergence in probability, we show that the variance of the $n$th moment vanishes asymptotically. Evaluating at $L$th roots of unity yields $|p_n|^2=p_n p_{L-n}$, whose decomposition in the Schur eigenbasis involves infinitely many terms. To overcome this difficulty, the key step is to rewrite this in terms of  products $s_{\{n_1|k_1\}}\overline{s_{\{n_2|k_2\}}}$, allowing us to reuse the Character identities in Lemma~\ref{lem:characteridentities}. The main additional inputs are the Schur decompositions in  Lemma~\ref{lem:variancehook} and the Taylor expansion of the eigenvalues for double-hook partitions $\{n|k,l\}$ (see Lemma~\ref{lem:eigendoublehookdevasym} and Figure~\ref{fig:doublehook}). The central references are \cite{Stembridge,Koike}. This yields the required decomposition and shows that the variance vanishes asymptotically.

\bigskip
 
\noindent
{\it A generalized Burgers equation for the moment generating function.} The moment generating function $g(t,z)=\sum_{n=1}^\infty \mathfrak m_n(t) z^n$ satisfies a generalized inviscid Burgers equation, a weak Fourier counterpart of \eqref{PDE0}. We denote by $\mathbb D$ the open unit disk and by $g_0(z)$ the generating function of $f_0(x)$.

\begin{theo}
	The  generating function $g(t,z)$ is the unique solution on $[0,\infty)\times\mathbb D$ of 
	\begin{equation}\label{eq:derivative0gbis00}
	\frac{\partial g(t,z)}{\partial t}+2\pi^2z\frac{\sin(\pi\alpha(1+2g(t,z)))}{\sin(\pi\alpha)} \frac{\partial g(t,z)}{\partial z}=0.
	\end{equation}
\end{theo}

Formally, the PDE \eqref{PDE0} is obtained from \eqref{eq:derivative0gbis00} by writing 
\begin{equation}\label{eq:linkgfh}
1+2g(t,e^{ix})=2\pi (f(t,x)+i\,\mathcal H f(t,x)).
\end{equation}
This representation recovers the density $f(t,x)$ via the method of characteristics, clarifies its regularity, and provides a geometric interpretation (see Section~\ref{sec:PDEid}).

The guiding principle is as follows. The characteristic flow $\Phi_t$ associated with \eqref{eq:derivative0gbis00},
\begin{equation}\label{eq:flow0}
\Phi_t(w)=w \exp\left(2\pi^2 t \frac{\sin(\pi\alpha(1+2g_0(w)))}{\sin(\pi\alpha)}\right),
\end{equation}
defines a biholomorphism from an open set $V_t\subset\mathbb D$ onto $\mathbb D$, with $(V_t)_{t\geq 0}$ decreasing. Denoting by $w(t,\cdot)$ its inverse, \eqref{eq:linkgfh} and the identity $g(t,z)=g_0(w(t,z))$ formally yield
\begin{equation}\label{eq:represent}
f(t,x)=\frac{1+\operatorname{Re} g_0(w(t,e^{ix}))}{2\pi},
\end{equation}
provided $w(t,\cdot)$ extends continuously to $\partial V_t$. 

Such an extension follows from the regularity of $\partial V_t$ via classical results in complex analysis, such as the Carathéodory--Torhorst theorem, which require $\partial V_t$ to be locally connected, a Jordan curve, or sufficiently smooth. In our setting, $\partial V_t\cap\mathbb D$ is always analytic, which implies that $f(t,x)$ is analytic on the corresponding domain. Moreover, for any $w_0=e^{ix_0}\in \partial V_t\cap\partial\mathbb D$, one has $f_0(x_0)=0$ or $f_0(x_0)=1/(2\pi\alpha)$ when $f_0$ is continuous. Setting $x_t=\arg(\Phi_t(w_0))$, the function $f(t,x_t)$ attains the same extremal value. The regularity of $f(t,x)$ at $x_t$ is thus governed by that of $\partial V_t$ at $w_0$. We refer to Figure~\ref{fig:boundary} for an illustration of these principles.

\begin{figure}[!h]
	\centering
	\includegraphics[scale=0.55]{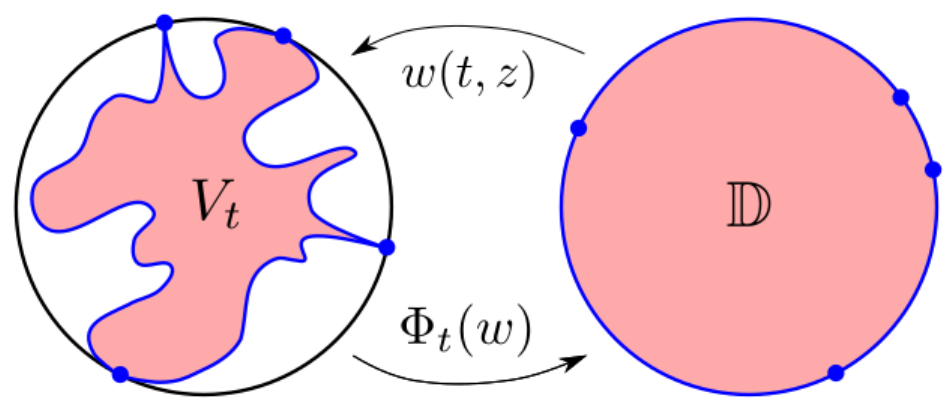}
	\captionsetup{width=14.7cm}
	\caption{ \small 
		Correspondence between $V_t$ and $\mathbb D$ under the conformal maps $\Phi_t$ and $w(t,\cdot)$. Here, boundary points match bijectively  since $\partial V_t$ is a Jordan curve. Consequently, $f(t,x)$ is continuous on the unit circle. The density has four extremal points, two singular and two smooth.
	}
	\label{fig:boundary}
\end{figure}

 \bigskip

\noindent
{\it The case of a maximally packed step initial profile.}  Finally, we analyze the maximally packed initial configuration
\begin{equation}\label{eq:packed}
f_0(x)=\frac{1}{2\pi\alpha}\,\mathds{1}_{[-\pi\alpha,\pi\alpha]}(x).
\end{equation}
The resulting evolution exhibits striking regularity, saturation, and support properties, detailed in Theorem~\ref{thm:pack} and illustrated in Figure~\ref{densites} with $\alpha<1/2$. Since Dirac measures are excluded, this initial profile plays the role of a Dirac mass $\delta_0$.  This configuration also hints at connections with Plancherel-type growth processes, discussed below.

In Figure~\ref{densites}, the time $t_\ast$ marks the disappearance of macroscopic congestion and $t^\ast$ that of macroscopic voids. For $\alpha>1/2$, the order is reversed, while for $\alpha=1/2$ the two times coincide. In all cases, $t^\ast$ is the first time the density becomes analytic. Singularities are either of square-root or cubic type, the latter occurring precisely at the critical times when a void or saturated region disappears. Sharp moment estimates follow from the stationary phase method, and the geometry of the associated open sets $(V_t)_{t\geq 0}$ is described explicitly in Section~\ref{sec:packedproof} (see Figure~\ref{fig:levelsetstep}).

\begin{figure}[!h]
	\centering
	\includegraphics[width=7cm]{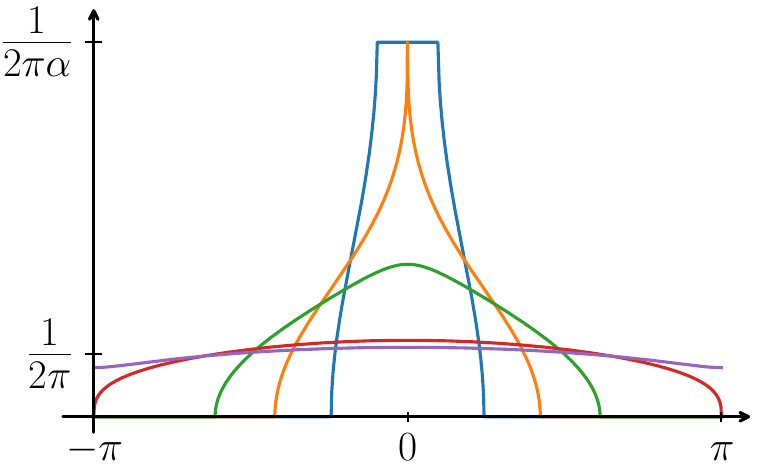}
	\captionsetup{width=12cm}
	\caption{\small  Evolution of $f(t,x)$ from the step initial profile \eqref{eq:packed} with $\alpha<1/2$, at times $t_1<t_\ast<t_2<t^\ast<t_3$. The density regularizes, converges to the uniform equilibrium $1/(2\pi)$, whereas the saturated plateau and support evolve.}
	
	\label{densites}
\end{figure}

\subsection{Related results and perspectives}

In this section, we take a step back from our results and place them in perspective by comparing them with related contributions from the literature.

\bigskip

\noindent
{\bf 1.} Independently and in parallel with the present work, results closely related to those of Sections~\ref{sec:1} and~\ref{sec:low} were obtained in \cite{GLT2023} using a different yet complementary algebraic approach, in contrast to the probabilistic and analytical methods developed here. The graph studied in \cite{GLT2023} is closely related to $\mathcal C_{L,N}$, but only positive (counterclockwise) jumps are permitted. In our framework, this model naturally corresponds to what may be termed the \emph{Maximum Entropy Totally Asymmetric Simple Exclusion Process} (METASEP).

They exploit that their graph is positively multiplicative, allowing  to  interpreted iterates of the METASEP kernel as products in a commutative algebra. This algebra, identified with the quantum cohomology of the Grassmannian, can be realized as generated by Schur functions specialized at $N$-tuples of $L$-th roots of unity.
This framework yields generalized Berry--Esseen estimates and a local limit theorem for fixed-time marginals.

Since the METASEP and MESSEP share the same eigenfunctions, the results of \cite{GLT2023} imply convergence in distribution of suitably centered fixed-time marginals to those of the UDBM. Our approach further yields functional convergence of the rescaled METASEP to the same limit. Although \cite{GLT2023} appears to extend to the symmetric setting (see their Example~2.5), it is not clear whether the symmetric graph $\mathcal C_{L,N}$ is positively multiplicative, and if so, with respect to which commutative algebra and basis. In particular, the classical Pieri rule no longer directly captures the symmetric adjacency structure.

Besides, our work also resolves a question raised in \cite[p.~14]{GLT2023} on the hydrodynamic limit. While that paper conjectures convergence of the marginals to the FUBM, the PDE \eqref{PDE0} shows that this does not occur: discretization introduces additional nonlinear terms in the limit. Nevertheless (see Point~{\bf 3.\@} below), the spectral distribution of the FUBM is recovered in the regime $\alpha\to 0$. It is therefore natural to analyze the hydrodynamic limit of the METASEP. For the usual Symmetric Simple Exclusion Process (SSEP) and the Totally/Weakly Asymmetric Simple Exclusion Process (TASEP/WASEP) on the circle, the hydrodynamic limits differ: heat diffusion in the symmetric case and Burgers-type equations in the asymmetric one (see paragraph {\bf 4.} below). It is thus not clear whether, after suitable centering, the limiting density for the METASEP satisfies the same PDE as in the MESSEP setting.

Finally, we point out that a connection between Hopf algebra Markov chains and MERW has recently been established in \cite{Marcolli25}. It would be interesting to further investigate possible similarities between this approach and that of \cite{GLT2023}.

\bigskip

\noindent
{\bf 2.} The UDBM \eqref{Dyson10}, introduced by Dyson \cite{Dyson}, arises in the study of Coulomb gases and random matrix theory (see, for instance, \cite{Serfaty,Forest}). It describes, up to deterministic time scaling, the evolution of the eigenvalues of the Unitary Brownian Motion $(U_t)_{t\ge0}$. 

The latter is a diffusion on the unitary group $U(N)$ solving the SDE
$dU_t = i\,U_t \circ dH_t$,
in the Stratonovich sense, where $(H_t)_{t\ge0}$ is a Hermitian Brownian motion with independent entries: real Brownian motions on the diagonal and complex Brownian motions above the diagonal. 

Its infinitesimal generator is (up to a multiplicative factor) the Laplace--Beltrami operator on $U(N)$. The irreducible characters of $U(N)$, that is  the Schur polynomials evaluated on the unit circle, are its eigenfunctions, and its radial part coincides with the generator of the UDBM (see \cite[Section~2.4.3]{BorodimPetrov}). In particular, the UDBM inherits the same symmetric eigenfunctions, with eigenvalues given by \eqref{energyspectrum}.

The UDBM can also be viewed as the Doob $h$-transform of the Calogero--Sutherland Hamiltonian 
\begin{equation}\label{eq:calo}
\mathcal H_{\mathrm{Suth}} 
= -\tfrac12 \sum_{i=1}^N \partial_{x_i}^2 
+ \sum_{i<j} \frac{1}{4 \sin^2\!\big(\tfrac{x_i-x_j}{2}\big)},
\end{equation}
via the ground-state transform with the Vandermonde eigenfunction $\Psi$ defined in \eqref{eq:Psi}. Conjugation by $\Psi$, up to an additive scalar, yields the generator of the UDBM. It is well known (see \cite{ForestCalo}) that the symmetric eigenfunctions of $\mathcal H_{\mathrm{Suth}}$ are the Schur polynomials on the unit circle. In light of the discussion below \eqref{eq:MERW}, the Hamiltonian \eqref{eq:calo} can thus be interpreted as a continuous analogue of the discrete one associated with  $\mathcal C_{L,N}$.

The UDBM is the law of $N$ Brownian motions on the circle conditioned never to collide (see  \cite{Werner,WangNCB}). This is in agreement with the interpretation of the MESSEP in Proposition~\ref{noncollinding}. Equivalently, it describes Brownian particles interacting via Coulomb repulsion (see \cite{CepaLepingle01}). In our setting, this repulsion emerges as an entropic force in the scaling limit of the discrete MESSEP, in a form even more explicit than in \cite{Coulomb}.

The novelty of our approach is that the UDBM arises from a simple, canonical discrete model  yielding a new microscopic derivation of this diffusion. This contrasts with earlier Markovian approximations built in more abstract frameworks in \cite[Section~4.1]{BorodimPetrov}, \cite{Boro}, \cite[Section~3]{Gorin}. It would be natural to ask whether these constructions admit a maximal entropy interpretation.

A natural extension is to consider exclusion dynamics on $\mathbb Z$, in both the symmetric and totally asymmetric regimes. On infinite graphs, MERWs are more delicate to define and study (see \cite{Duboux} and \cite{Dovgal}). Under diffusive scaling, one may then expect to recover the classical Dyson Brownian Motion for the eigenvalues of a Hermitian Brownian motion.

Another natural  direction is to introduce a confining potential, for instance quadratic, to connect the model with the Gaussian Unitary Ensemble. From our perspective, this confinement could be incorporated canonically by adding an energetic constraint to the maximal path entropy principle.

\bigskip
\noindent
{\bf 3.} Although the PDE \eqref{PDE0} does not describe the evolution of the spectral distribution of the FUBM, it admits a formal connection in the regime $\alpha\to0$, leading to
\begin{equation}\label{PDE00}
\partial_t f + 4\pi^3\,\partial_x\!\left(f\,\mathcal H f\right)=0.
\end{equation}

The PDE \eqref{PDE00} is a nonlocal continuity equation of the form $\partial_t f+\partial_x(fv)=0$, with velocity field proportional to $\mathcal H f$. On the real line, the corresponding equation is simply obtained by replacing $\mathcal H$ with the standard Hilbert transform.

In both geometries, such PDEs admit a $2$-Wasserstein gradient-flow interpretation for logarithmic interaction potential: $\log|x-y|$ on the real line and $\log|\sin((x-y)/2)|$ on the circle. We refer to \cite[Chap.~11]{Ambrosio} and \cite{Ferreira2,Ferreira1} for instance. Connections with isentropic Euler systems and optimal transport are discussed in \cite{Gangbo,Abanov2,Guionnet2,Majumdar,BilerCrystal}. From the random-matrix viewpoint, equations such as \eqref{PDE00} arise as large-$N$ hydrodynamic limits of Dyson-type systems: see \cite{Mallick,Katori,Guillin,Bertucci1-1,Bertucci1-2,Bertucci2hal} for recent developments, including the circular case. We briefly recall the corresponding complex Burgers equations in both settings to clarify the link with our work.

\medskip

\noindent
{\it On the real line.} Let $\rho_t$ be the large-$N$ limit of the empirical eigenvalue distribution of Hermitian Brownian motion $(H_t)_{t\ge0}$ started at $0$, that is, the law of the free additive Brownian motion. Its Stieltjes transform $G(t,z)$ solves the inviscid complex Burgers equation
\begin{equation}\label{BurgersReal}
\partial_t G(t,z) + G(t,z)\,\partial_z G(t,z)=0,
\end{equation}
yielding the semicircle law
\begin{equation}\label{eq:semicirc}
\rho(t,x)=\frac{1}{2\pi t}\sqrt{4t-x^2}.
\end{equation}

The same Burgers structure governs the asymptotics of matrix integrals such as the HCIZ integral (see \cite{Matytsin,Menon}), and is closely related to developments in combinatorics and representation theory (see \cite{Collins,GuionnetChar,Goulden,Novak}).

\medskip

\noindent
{\it On the unit circle.} Let $\nu_t$ be the large-$N$ limit of the empirical eigenvalue distribution of unitary Brownian motion $(U_t)_{t\ge0}$ started at $I_n$, that is, the spectral measure of the free unitary Brownian motion. One has $\nu_0=\delta_1$, or equivalently $\delta_0$ on $\mathbb R/2\pi\mathbb Z$. Its Herglotz transform $H(t,z)$ (see Remark~\ref{rem:herglotz}) determines the density via its non-tangential boundary values. Free-probabilistic arguments due to Biane \cite{BianeFree,BianeAnalogue} show that $H$ solves the Burgers-type equation
\begin{equation}\label{BurgersCircle}
\partial_t H(t,z) + \frac{z}{2} H(t,z)\,\partial_z H(t,z) = 0,
\end{equation}
a multiplicative analogue of \eqref{BurgersReal}. The density is supported on an arc whose opening increases with time and eventually covers the circle. The boundary value $H(t,e^{ix})$ is the unique solution with positive real part of
\begin{equation}\label{eq:herglotsolution}
\frac{H(t,e^{ix})-1}{H(t,e^{ix})+1}\,e^{\frac{t}{2} H(t,e^{ix})}=e^{ix},
\end{equation}
and $\nu(t,x)=\operatorname{Re} H(t,e^{ix})$. See also \cite[Theorem~2.3]{Hamdi} and \cite{Nizar}.

\medskip

Thus, in both the additive (real-line) and multiplicative (unit-circle) settings, the large-$N$ spectral evolution is governed by boundary values of analytic transforms (Stieltjes or Herglotz) solving inviscid complex Burgers equations. Our approach follows the same spirit: the nonlinear transport equation \eqref{PDE0} is a real-variable analogue of the Burgers-type equation \eqref{eq:derivative0gbis00} satisfied by the analytic transform $g(t,z)$, and provides an implicit description of the dynamics via boundary characteristics as in \eqref{eq:herglotsolution}. This formulation is reminiscent of the Loewner--Kufarev theory for conformal maps in the unit disk. Here, the analytic map $w(t,z)$ plays an analogous role and governs the evolution of the domains $V_t$ (see  Remark~\ref{rem:loewner}).

In particular, the maximally packed configuration in Figure~\ref{densites} mirrors the Dirac initial condition in \eqref{PDE00}: a moving support emerges, with edges determined by the characteristics of \eqref{eq:derivative0gbis00}, and square-root singularities form at the boundary, as the semicircle law \eqref{eq:semicirc} and its circular analogue.

Note that the limiting hydrodynamics depends on the order of the limits. Sending first the number of sites to infinity and then the number of particles yields \eqref{PDE00}, whereas taking the limits jointly leads to \eqref{PDE0}, where the density parameter $\alpha$ persists and generates additional nonlinear terms. Letting $\alpha\to0$ in \eqref{PDE0} formally connects the two regimes (classical Dyson gas and  free-unitary hydrodynamic) while retaining genuinely new nonlinear effects (see Figure~\ref{fig:diag}).

\begin{figure}[!h]
	\centering
	\begin{tikzcd}[row sep=1cm, column sep=2.25cm]
		\mathrm{UDBM}
		\arrow[r, "\substack{\text{hydrodynamic limit}\\ N\to\infty}"]
		&
		\mathrm{FUBM}
		\\
		\mathrm{MESSEP}
		\arrow[u, "\substack{N\ \text{fixed}\\ L\to\infty}"]
		\arrow[r, "{N}/{L}\to\alpha"']
		&
		\substack{\text{nonlinear}\\ \mathrm{FUBM}}
		\arrow[u, "\alpha\to 0"']
	\end{tikzcd}
	\caption{\small  Connections between  MESSEP, UDBM and FUBM}
	\label{fig:diag}
\end{figure}
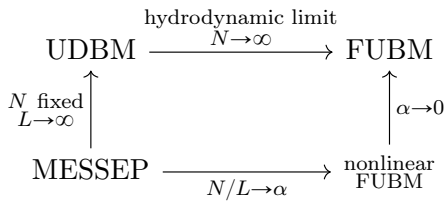

It would be natural to ask whether an analogous bridge exists on the real line. A free-probability interpretation in the spirit of Biane may also be fruitful.

\bigskip

\noindent
{\bf 4.} We need to compare our limiting PDE \eqref{PDE0} with that of the Simple Symmetric Exclusion Process (SSEP), where all allowed jumps are equally likely. We refer to the classical monograph \cite{KipnisLandim} and to the more recent developments in \cite{SimonS,SimonF,Goncalvez}. Most of the literature concerns the continuous-time version, in which each particle is equipped with an independent exponential clock and, when it rings, attempts to jump uniformly to one of its neighboring sites, if allowed.

Let $(\eta_t)_{t\geq 0}$ denote the continuous-time Markov process with state space $\{0,1\}^{\mathbb T_L}$, where $\eta_t(x)=1$ if and only if site $x$ is occupied at time $t$, and let  $\pi_{L,N}(t,dx)$ be the empirical measure. Under the assumptions of Theorem~\ref{thm0}, it is well known that $\pi_ {L,N}(L^2t,dx)$ converges to a density $\rho_c(t,x)$ solving the heat equation on the circle. However, to completely compare with the MESSEP, we need to derive heuristically the corresponding PDE for the discrete-time SSEP. 

To this end, observe that the instantaneous rate of effective jumps at time $t$ is proportional to
\begin{equation}
R(\eta_t)=\sum_{x\in\mathbb T_L}\mathds 1_{\{\eta_t(x)\neq \eta_t(x+1)\}}.
\end{equation}
Besides, under local equilibrium, one has
\begin{equation}
\mathbb P(\eta_{L^2 t}(x)\neq \eta_{L^2 t}(x+1))
\simeq 2\,\rho_c\!\left(t,\frac{2\pi x}{L}\right)
\left(1-\rho_c\!\left(t,\frac{2\pi x}{L}\right)\right).
\end{equation}
Hence, if $A_t$ denotes the total number of effective jumps up to time $t$, one can write
\begin{equation}
\mathbb E[A_{L^2 t}]
\simeq 2L^2\int_0^t \sum_{x\in\mathbb T_L}
\rho_c\!\left(s,\frac{2\pi x}{L}\right)
\left(1-\rho_c\!\left(s,\frac{2\pi x}{L}\right)\right)\,ds
\simeq \frac{L^3}{\pi}\int_0^t\int_0^{2\pi}
\rho_c(s,x)(1-\rho_c(s,x))\,dx\,ds.
\end{equation}
Considering the embedded discrete-time Markov chain $(\xi_n)_{n\geq 0}$ associated with effective jumps, a time-change shows that the density profile of $(\xi_{L^3 t})_{t\geq 0}$ converges to $\rho_d(t,x)$ solving
\begin{equation}
\partial_t\rho_d(t,x)
=\frac{\sigma_\alpha}{\int_{0}^{2\pi}\rho_d(t,x)(1-\rho_d(t,x))\,dx}
\partial_{xx}\rho_d(t,x),
\end{equation}
for some constant $\sigma_\alpha>0$ depending on $\alpha$. The convergence to equilibrium is therefore slower in the discrete-time SSEP (order $L^3$) than in the MESSEP (order $L^2$). Moreover, the factor $\rho(1-\rho)$ plays the role of a nonlinear flux in the TASEP case, whose limiting PDE is of transport type.

\bigskip

\noindent
{\bf 5.} Finally, the MESSEP dynamics at the interface between a saturated region and a macroscopic void  can be described via the Plancherel growth process (see for instance \cite{kerov}), in agreement with its maximal-entropy interpretation in \cite{Dovgal}.

More precisely, assume that only finitely many of the $N$ initially packed particles are displaced, with shifts negligible compared to $N$. Let $n^+$ (resp. $n^-$) denote the number of particles moving in the positive (resp. negative) direction on the unit circle. Their positions are denoted by $x_1,\dots,x_{n^+}$ on the right and $y_1,\dots,y_{n^-}$ on the left, ordered from farthest to nearest relative to the saturated region (see Figure~\ref{fig:plancherel}). The extremities of the moving boundaries are denoted by $\ell^+$ and $\ell^-$.

\begin{figure}[!h]
	\centering
	\includegraphics[width=0.85\textwidth]{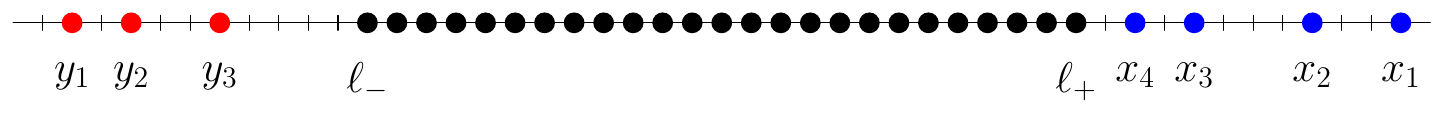}
	\captionsetup{width=15cm}
\caption{\small Interface between a saturated region and a macroscopic void.}
	\label{fig:plancherel}
\end{figure}

Combining the explicit product form of the Perron--Frobenius eigenfunction \eqref{PF} with the eigenvalue asymptotics in the hydrodynamic limit \eqref{eq:riemansum} yields that the probability that the $k$-th particle on the right jumps by $+1$ is asymptotically
\begin{equation}\label{eq:tasep}
\lim_{L,N\to\infty} \mathbb P\big(x_k \to x_k+1\big)
=
\frac{1}{2}
\frac{\prod_{i\neq k}|x_k+1-x_i|}{\prod_{i\neq k}|x_k-x_i|}
\frac{1}{|x_k-\ell_+|}.
\end{equation}
By symmetry, an analogous formula holds on the left for jumps of size $-1$, and the boundary cases at $\ell^+$ and $\ell^-$ are treated similarly.  In contrast, jumps of size $-1$ on the right and $+1$ on the left vanish asymptotically.

It turns out that, up to the factor $1/2$, the transition probability \eqref{eq:tasep} matches that of the TASEP--Plancherel growth process. This suggests  a link between the classical Plancherel limit shape of Young diagrams (see \cite[p.~699]{Romik}) and the limiting density profile in Figure~\ref{densites}.

\section{Model and preliminary results}

In this section, we introduce the MESSEP on the ring and its variants, summarize their basic properties, and present the full spectral decomposition.  Most of the proofs are given in Section \ref{sec:proofsec1}.

\label{sec:1}
\setcounter{equation}{0}

\subsection{The MESSEP on the ring}

\label{sec:model1}

\subsubsection{The graph structure}

Let $N, L$ be integers with $1 \leq N \leq L-1$, and set $\mathbb{T}_L := \{0, \dots, L-1\}$. We view $\mathbb{T}_L$ as a discrete ring, identified with $\mathbb{Z}_L \equiv \mathbb{Z}/L\mathbb{Z}$ or, when convenient, with $\mathbb{U}_L$ (the set of $L$th roots of unity). The parameter $N$ denotes the number of indistinguishable particles on the ring, subject to the exclusion rule (at most one particle per site). 

We endow $\mathbb{T}_L$ with the natural graph structure of a discrete ring, with edges $\{k,k\pm1\}$ for $k\in\mathbb{T}_L$ (where $-1\equiv L-1$ and $L\equiv0$), and identify the graph with its vertex set. A configuration is a subset $\xi=\{\xi_1,\dots,\xi_N\}\subset\mathbb{T}_L$ of cardinality $N$, represented by the ordered particle positions. Denoting by $\mathcal C_{L,N}$ the set of all configurations, we identify 
\begin{equation}\label{state}
\mathcal C_{L,N} = \mathcal P_N(\mathbb T_L) \simeq \left\{\xi \in \mathbb{T}^N_L : \xi_1 < \cdots < \xi_N \right\}.
\end{equation}

Each particle may jump to a nearest empty site, with only one particle moving at each step. This defines an undirected graph structure on $\mathcal C_{L,N}$: we connect $\xi$ and $\eta$ if and only if there exist $1 \leq k \leq N$ and $\epsilon \in \{\pm 1\}$ such that
\begin{equation}\label{edge}
(\eta_1, \cdots, \eta_N) = (\xi_1, \cdots, \xi_{k-1}, \xi_k + \epsilon, \xi_{k+1}, \cdots, \xi_N) \quad (\text{abbreviated as $\eta \sim \xi$}).
\end{equation}
To account for periodicity, we set
\begin{equation}\label{boundary}
(-1, \xi_2, \cdots, \xi_N) \equiv (\xi_2, \cdots, \xi_N, L-1)
\quad \mbox{and} \quad
(\xi_1, \cdots, \xi_{N-1}, L) \equiv (0, \xi_1, \cdots, \xi_{N-1}).
\end{equation}

See Figure \ref{ring} for an illustration. The graph $\mathcal C_{L,N}$ is connected. Let $A$ be its symmetric adjacency matrix, $\rho$ its spectral radius, and $\psi$ a corresponding positive eigenfunction (defined up to a positive multiplicative constant).

\subsubsection{The MERW}

\begin{defin}\label{def:MERW}
	The Maximal Entropy Simple Symmetric Exclusion Process (MESSEP) on the ring is the unique MERW on $\mathcal C_{L,N}$. It is denoted by $({\Xi(n)})_{n \geq 0}$.
\end{defin}

Its Markov kernel $P$ and reversible invariant measure $\mu$ are given, for all $\xi, \eta \in \mathcal C_{L,N}$, by
\begin{equation}\label{transition}
P(\xi, \eta) = A(\xi,\eta) \frac{\psi(\eta)}{\rho\,\psi(\xi)}
\quad \mbox{and} \quad
\mu(\xi) = \frac{\psi(\xi)^2}{Z},
\end{equation}
where $Z$ is the normalization constant. 

Via the map $x \mapsto e^{2i\pi x/L}$ applied componentwise, this process may equivalently be viewed as taking values in $\mathcal C_{L,N}^{\scriptscriptstyle \#} = \mathcal P_N(\mathbb U_L) \subset \mathcal P_N(\mathbb U)$, with $\mathbb U$ the unit circle. Accordingly, functions, measures, kernels, and related objects on $\mathcal C_{L,N}$ are identified with those on $\mathcal C_{L,N}^{\scriptscriptstyle \#}$.

\subsubsection{Extended dynamics}

\label{sec:extend}

The dynamics of $({\Xi(n)})_{n \geq 0}$ can be lifted to retain pathwise information on individual particles. Writing $\dot x$ for the class of $x \in \mathbb Z$ modulo $L$, we set
\begin{multline}\label{eq:weyldiscretes}
{\mathcal W_{\mathbb Z,L,N}}:=\left\{x\in \mathbb Z^N : x_1<\cdots <x_N< x_1+L\right\}\quad \mbox{and}\\ \quad
{\mathcal W_{L,N}}:=\left\{\theta \in \mathbb Z_L^N :  \exists\, x\in {\mathcal W_{\mathbb Z,L,N}},\, \theta_1=\dot x_1,\cdots, \theta_N=\dot x_N\right\}.
\end{multline}
As previously, we identify ${\mathcal W_{L,N}}$ with
\begin{equation}\label{eq:wstar}
{\mathcal W_{L,N}^{\scriptscriptstyle \#}}=\left\{\theta \in \mathbb U_L^N :  \exists\, x\in {\mathcal W_{\mathbb Z,L,N}},\, \theta_1=e^{{2i\pi x_1}/{L}},\cdots, \theta_N=e^{{2i\pi x_N}/{L}}\right\}.
\end{equation}

All these state space  carry a graph structure analogous to that of $\mathcal{C}_{L,N}$: one may add $\pm1$ to a coordinate (or rotate by $\pm2\pi/L$ in the $\mathbb U$-embedding), whenever admissible. The adjacency relation is still denoted by $\sim$ and written as in \eqref{edge}. Note that  the boundary identification \eqref{boundary} is no longer needed with this setting.

Besides, there is a canonical surjection $p_1 : \mathcal W_{\mathbb{Z},L,N} \longtwoheadrightarrow \mathcal W_{L,N}$, induced by the projection from $\mathbb{Z}$ onto $\mathbb{Z}_L $. Similarly, let $p_2$ be the canonical projection $\mathbb Z_L  \longtwoheadrightarrow \mathbb T_L$, again extended componentwise. Let $G = \langle \tau \rangle$ be the cyclic subgroup of $\mathfrak S_{\sss N}$ generated by $\tau = (1\ 2 \cdots N)$, and write $\sigma \cdot \theta$ for the natural action of $\sigma \in \mathfrak S_{\sss N}$ on  $\theta$. Then, for all $\theta \in \mathcal W_{L,N}$, there exists a unique $\sigma \in G$ such that $p_2(\sigma \cdot \theta) \in \mathcal C_{L,N}$. This defines a canonical surjection $\mathcal W_{L,N} \longtwoheadrightarrow \mathcal C_{L,N}$, still denoted by $p_2$.

To summarize, one has
\begin{equation}\label{diagram}
{\mathcal W_{\mathbb Z,L,N}} \overset{p_1}{\longtwoheadrightarrow} {\mathcal W_{L,N}} \simeq {\mathcal W_{L,N}^{\scriptscriptstyle \#}} \overset{p_2}{\longtwoheadrightarrow} \mathcal C_{L,N} \simeq \mathcal C_{L,N}^{\scriptscriptstyle \#},
\end{equation}
which corresponds to different levels of information concerning the particles on the graph:
\begin{enumerate}
	\item[--] ${\mathcal W_{L,N}} \simeq {\mathcal W_{L,N}^{\scriptscriptstyle \#}}$ captures the position on the ring of each particle at any time. Particles are no longer indistinguishable but tagged, in contrast with  $\mathcal C_{L,N} \simeq \mathcal C_{L,N}^{\scriptscriptstyle \#}$.
\item[--] ${\mathcal W_{\mathbb Z,L,N}}$ additionally records, for each particle, the difference between the numbers of its $+1$ and $-1$ moves. In particular, the  winding numbers can be recovered (see Remark \ref{rem:winding}).
\end{enumerate}

\begin{rem}\label{rem:lift1}
	Given a path $(\xi(n))_{n\geq 0}$ in ${\mathcal C_{L,N}}$,  $k \in \{1, \cdots, N\}$, $z \in \mathbb Z_{L}$, with $p_2(z) = \xi_k(0)$, there exists a unique lift of this path  in ${\mathcal W_{L,N}}$, say $(\theta(n))_{n\geq 0}$, such that $\theta_k(0) = z$. Similarly, given a path $(\theta(n))_{n\geq 0}$ in ${\mathcal W_{L,N}}$,  $k \in \{1, \cdots, N\}$,  $y \in \mathbb Z$, with $p_1(y) = \theta_k(0)$, there exists a unique lift of this path  in ${\mathcal W_{\mathbb Z,L,N}}$, say $(x(n))_{n\geq 0}$, such that $x_k(0) = y$.
\end{rem}

By Remark \ref{rem:lift1}, any stochastic process on ${\mathcal C_{L,N}}$ can be uniquely lifted to ${\mathcal W_{L,N}}$ and ${\mathcal W_{\mathbb Z,L,N}}$ by tagging one  particle -- say, the $k$th -- at time zero. If the original process is Markovian, so are its lifts. This motivates the following definition.

\begin{defin}\label{defmessep}
	We denote by $(\Theta(n))_{n\geq 0}$ and $(X(n))_{n\geq 0}$ the random walk counterparts on ${\mathcal W_{L,N}}$ and $\mathcal W_{\mathbb{Z},L,N}$ of the MESSEP $(\Xi(n))_{n\geq 0}$ on $\mathcal C_{L,N}$. All these processes will be referred to as MESSEP.
\end{defin}

The transition kernels and invariant measures of $(\Theta(n))_{n\geq 0}$ and $(X(n))_{n\geq 0}$ follow directly from those of $(\Xi(n))_{n\geq 0}$. In particular, when $\theta \sim \phi$ in ${\mathcal W_{L,N}}$ and $x \sim y$ in $\mathcal W_{\mathbb{Z},L,N}$,
\begin{equation}\label{eq:extendedtransition}
\mathbb{P}(X(n+1) = y \mid X(n) = x) = \frac{\psi_1(y)}{\rho\,\psi_1(x)}, \quad
\mathbb{P}(\Theta(n+1) = \theta \mid \Theta(n) = \phi) = \frac{\psi_2(\theta)}{\rho\,\psi_2(\phi)},
\end{equation}
where $\psi_1$ and $\psi_2$ are the corresponding lifts of $\psi$ which can be defined by
\begin{itemize}
	\item $\psi_2(\theta) = \psi(\tau^l \cdot \eta)$, where $\eta$ is the $N$-tuple in $\mathbb{T}_L$ such that $\theta_k = \dot \eta_k$, $1\leq k\leq N$ and $l$ is the unique integer  in $\{1,\cdots, N-1\}$  such that $(\eta_1, \cdots, \eta_{\tau^l(N)}) \in \mathcal{C}_{L,N}$.
\item $\psi_1(x) = \psi_2(\theta)$, where $\theta_k = \dot x_k$, $1 \leq k \leq N$. 
\end{itemize}

More generally, any function $\varphi$ on $\mathcal{C}_{L,N}$ can be lifted to $\mathcal{W}_{L,N}$ and $\mathcal{W}_{\mathbb{Z},L,N}$ via the projections \eqref{diagram}, and similarly for measures and kernels. By abuse of notation, we often use the same symbols whenever no confusion arises. Likewise, we identify $L$-periodic functions on integer tuples, functions on $\mathbb{Z}_L^N$, and functions on $\mathbb{U}_L^N$, writing for instance
\begin{equation}\label{eq:identificationfunction}
\psi(\eta_1,\cdots,\eta_N) \equiv \psi(\dot \eta_1,\cdots,\dot \eta_N) \equiv \psi\left(e^{2i\pi \eta_1 / L}, \cdots, e^{2i\pi \eta_N / L}\right).
\end{equation}

\begin{rem}\label{rem:Rrec}
	Both $(\Theta(n))_{n\geq 0}$ and $(X(n))_{n\geq 0}$ are the unique MERWs on ${\mathcal W_{L,N}}$ and $\mathcal W_{\mathbb{Z},L,N}$. The former case is immediate. For $\mathcal W_{\mathbb{Z},L,N}$, which is infinite, uniqueness follows from the fact that the graph is $R$-recurrent. We refer to \cite{VJ} together with the results of \cite{Duboux}.
\end{rem}

\subsubsection{A non-colliding random walk} 

The next result shows that the MERW on the ring coincides with the simple random walk conditioned on non-collision. We equip $\mathbb Z_L^N$ with its nearest-neighbour Cayley graph structure.

\begin{prop}\label{noncollinding} 
	Let $(Y(n))_{n\geq 0}$ be the simple random walk on $\mathbb{Z}_L^N$, and let $y \in \mathbb{Z}_L^N$ have pairwise distinct coordinates. Let $\sigma \in \mathfrak{S}_N$ be such that $\sigma \cdot y \in {\mathcal W_{L,N}}$ and set
	\begin{equation} 
	T := \inf\left\{n\geq 0 : \exists\, 1 \leq i \neq j \leq N,\; Y_i(n) = Y_j(n)\right\}.
	\end{equation}
	Then
	\begin{equation}\label{conditionalBM}
	\lim_{m\to\infty} \mathbb{P}_{\sigma\cdot y}(\sigma\cdot Y \in d\omega \mid T > m)
	= \mathbb{P}_{\sigma\cdot y}(\Theta \in d\omega).
	\end{equation}
\end{prop}



\subsection{Full spectral decompositions and spectral gap estimates}\label{graphspec}

\subsubsection{The graph spectrum} 

Denote by $\Sigma_{L,N}$ the spectrum of ${\mathcal C_{L,N}}$ and by $\mathcal{F}_{L,N}$ an orthonormal basis of eigenfunctions in $\ell^2({\mathcal C_{L,N}})$ for the usual scalar product. Both will be described explicitly below. Let us set
\begin{equation}\label{eq:gamma}
{\bgamma}=\left\{\begin{array}{ll}
0, &\text{if $N$ is odd,}\\[5pt]
\frac{1}{2}, & \text{if $N$ is even},
\end{array}\right.
\end{equation}
and, for $\xi \in {\mathcal C_{L,N}}$ and any real $N$-tuple $\eta$, define
\begin{equation}\label{Spectre}
{\psi}_\xi(\eta)=
\frac{1}{L^{N/2}}\;{\rm det}\left(e^{\frac{2i\pi (\xi_{k}+\bgamma)\eta_{j}}{L}}\right)_{1\leq k,j\leq N}
\quad\mbox{and}\quad
\rho_\xi=2\sum_{i=1}^N\cos\left(\frac{2\pi(\xi_{i}+\bgamma)}{L}\right).
\end{equation}
We call the configuration $c \in \mathcal C_{L,N}$ defined below the compact configuration:
\begin{equation}
(c_i)_{1\leq i\leq N}=\left\{\begin{array}{ll}\label{compactconf}
(0,1,\cdots,p,L-p,\cdots,L-1), & \text{if $N=2p+1$},\\[10pt]
(0,1,\cdots,p-1,L-p,\cdots,L-1), & \text{if $N=2p$},
\end{array}\right.
\end{equation}

\begin{rem}\label{rem:compactrep}
	The configuration $\bc = (-p, \cdots, -p + N - 1)$ is a lift of $c$ in $\mathcal{W}_{L,N}$. If $N$ is even, $\bc$ is not symmetric about zero, contrary to the shifted one $\bc + \bgamma = (\bc_i + \bgamma)_{1 \leq i \leq N}$. We call $c + \bgamma = (c_i + \bgamma)_{1 \leq i \leq N}$ the symmetric compact configuration (see Figure \ref{partition}).
\end{rem}

\begin{prop}\label{MESSEPspectre} 
	The Perron--Frobenius eigenfunction $\psi$ is equal to $\psi_c$ (up to a constant of unit modulus), with eigenvalue $\rho=\rho_c$. More precisely, one has
	\begin{equation}\label{PF}
	\psi(\eta) 
	= \frac{2^{\frac{N(N-1)}{2}}}{L^{N/2}}  
	\prod_{1\leq i<j\leq N} \sin\left(\frac{(\eta_{j}-\eta_{i})\pi}{L}\right)  
	\quad \text{and} \quad  
	\rho= \frac{2\sin\left(\frac{N\pi}{L}\right)}{\sin\left(\frac{\pi}{L}\right)}.
	\end{equation}
	Besides, the map $\xi \longmapsto   (\eta \in {\mathcal C_{L,N}} \mapsto \psi_{\xi}(\eta))$ is a bijection	${\mathcal C_{L,N}}$ onto $\mathcal F_{L,N}$ and the eigenvalue associated with $\psi_\xi$ is $\rho_\xi$ for any $\xi \in {\mathcal C_{L,N}}$. In particular, one has $\Sigma_{L,N}=\{\rho_\xi : \xi\in {\mathcal C_{L,N}}\}$.
\end{prop}

\begin{rem}
	One can check that the invariant probability distribution of the MESSEP is a determinantal point process whose reproducing kernel on $\mathbb R/2\pi\mathbb Z$  is given by
	\begin{equation}
	K(x,y)=\frac{1}{L}\frac{\sin\!\left(\frac{\pi N (y-x)}{L}\right)}{\sin\!\left(\frac{\pi (y-x)}{L}\right)}.
	\end{equation}
	This is the discrete Dirichlet kernel. In the regime $L,N\to\infty$ with $N/L\to \alpha$, it converges to the classical sine kernel, yielding the sine process.
\end{rem}

\subsubsection{The MESSEP spectrum} 

Let $P^*$ be the $\ell^2({\mathcal C_{L,N}})$-adjoint of $P$, and let $m$ be the counting measure on ${\mathcal C_{L,N}}$. The, it follows easily from Proposition \ref{MESSEPspectre} that 
\begin{equation}\label{eigenmessep}
\mathcal S_{L,N}:=\left\{\frac{\psi_{\xi}}{\psi_{c}} :  \xi\in {\mathcal C_{L,N}} \right\}
\quad\mbox{and}\quad 
\mathcal S^\ast_{L,N}:=\left\{\overline{\psi_{\xi}}\,{\psi_{c}} :  \xi\in {\mathcal C_{L,N}} \right\},
\end{equation}
form orthonormal eigenbases of $P$ and $P^*$ in $\ell^2({\mathcal C_{L,N}},\psi^2\,dm)$ and $\ell^2({\mathcal C_{L,N}},\psi^{-2}\,dm)$, respectively. The MESSEP spectrum is therefore given by
\begin{equation}\label{MESEPspectrum}
\Sigma_{L,N}^{\rm \sss ME}:=\left\{\frac{\rho_{\xi}}{\rho_{c}}: \xi\in {\mathcal C_{L,N}} \right\}.
\end{equation}
Accordingly, for any function $f$ on ${\mathcal C_{L,N}}$, $\eta\in {\mathcal C_{L,N}}$ and $n\geq 0$, one can write 
\begin{equation}\label{eq:fullspectralmarkov}
P^n f(\eta)
=\sum_{\xi\in {\mathcal C_{L,N}}}\left(\frac{\rho_\xi}{\rho_{c}}\right)^n
\left(\sum_{\zeta\in {\mathcal C_{L,N}}} f(\zeta)\psi_{c}(\zeta)\overline{\psi_\xi(\zeta)}\right)
\frac{\psi_\xi(\eta)}{\psi_{c}(\eta)}.
\end{equation}

\begin{prop}\label{gapestimate}
	Let $\lambda_{L,N}$ be the second eigenvalue of maximal modulus of $P$. Then, the spectral gap $1 - \lambda_{L,N}$ satisfies the following asymptotics as $L \longrightarrow \infty$:
	\begin{enumerate}
		\item Assume that $N/L \longrightarrow 0$ (low-density limit). Then, there exists a positive constant $C$, independent of $L$ and $N$, such that
		\begin{equation}\label{estimategapmessep}
		\left|1 - \lambda_{L,N} - \frac{2\pi^2}{L^2}\right|
		\leq C \left(\frac{N}{L}\right)^2 \frac{1}{L^2}.
		\end{equation}
		
		\item Assume that $N/L \longrightarrow \alpha$ (hydrodynamic limit), with $\alpha$ belonging to a compact set $K \subset (0,1)$. Then, there exists a positive constant $C_K$, independent of $L$, $N$, $\alpha$, such that 
		\begin{equation}\label{estimategapmessep2}
		\left|1 - \lambda_{L,N} - \frac{2\pi^2}{L^2}\right|
		\leq \frac{C_K}{L^3}.
		\end{equation}
	\end{enumerate}
\end{prop}

\subsubsection{Schur polynomial eigenfunctions} 

\label{sec:schur1}

The eigenfunctions $\mathcal S_{L,N}$ of the MESSEP associated with the spectrum $\Sigma_{L,N}^{\rm \sss ME}$ given in \eqref{MESEPspectrum} are precisely the Schur polynomials in $N$ variables, restricted to $\mathbb U_L$. We briefly recall their main properties and refer to \cite{stanley1999enumerative,fulton1997young} for further details.

Let $\lambda = (\lambda_1, \cdots, \lambda_N)$ be an integer partition, that is intergers satisfying  $\lambda_1 \geq \cdots \geq \lambda_N \geq 0$, and set
\begin{equation}\label{eq:schurdet}
Q_{\lambda+\delta}(X_1,\cdots,X_N)=
{\rm det}\left(\begin{array}{llll}
X_1^{\lambda_1+N-1} & X_2^{\lambda_1+N-1} &\cdots & X_N^{\lambda_1+N-1}\\
X_1^{\lambda_2+N-2} & X_2^{\lambda_2+N-2} &\cdots  & X_N^{\lambda_2+N-2}\\
\vdots & \vdots  & \ddots  & \vdots \\
X_1^{\lambda_N} & X_2^{\lambda_N} & \cdots & X_N^{\lambda_N}
\end{array}\right),
\end{equation}
where  $\delta = (N-1, \cdots, 0)$.  Since $Q_{\lambda+\delta}(X_1, \cdots, X_N)=0$ whenever $X_i=X_j$ for some $i\neq j$, the Vandermonde determinant $Q_\delta(X_1, \cdots, X_N)$ (corresponding to $\lambda=0$) divides $Q_{\lambda+\delta}$. 

The Schur polynomial in $N$ variables associated with $\lambda$ is defined by
\begin{equation}\label{eq:schurdef}
s_{\lambda}(X_1, \cdots, X_N) = \frac{Q_{\lambda+\delta}(X_1, \cdots, X_N)}{Q_{\delta}(X_1, \cdots, X_N)}.
\end{equation}It admits the combinatorial expression
\begin{equation}\label{youngtab00}
s_{\lambda}(X_1, \cdots, X_N) = \sum_{T} X_1^{t_1} \cdots X_N^{t_N},
\end{equation}
where the sum runs over all semi-standard Young tableaux of shape $\lambda$, and $t_i$ denotes the number of entries equal to $i$ in $T$ (see Figure~\ref{SSYT}).

\begin{figure}[!h]
	\centering
	\includegraphics[width=3cm]{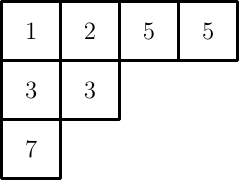}
		\captionsetup{width=10.5cm}
\caption{\small A semi-standard Young tableau $T$ of shape $\lambda=(4,2,1)$, that is, with weakly increasing rows and strictly increasing columns, associated with the monomial $X_1 X_2 X_3^2 X_5^2 X_7$.} 
	\label{SSYT}
\end{figure}

\begin{rem}\label{specialization}
	The schur polynomial $s_\lambda$ may be regarded as a polynomial in infinitely many variables such that  $s_\lambda(X_1,\cdots,X_N)= s_\lambda(X_1,\cdots,X_N,0,\cdots)
	$ and  $s_\lambda(X_1,\cdots,X_N)=0$ if $\lambda_{N+1}\neq 0$.
\end{rem}

The Schur polynomials form a linear  basis of the space of symmetric polynomials. When restricted to $\mathbb U_L^N$, they canonically define functions on $\mathcal{W}_{\mathbb{Z},L,N}$, $\mathcal{W}_{L,N} \simeq \mathcal{W}_{L,N}^{\scriptscriptstyle \#}$, and $\mathcal{C}_{L,N}$ (see the discussion preceding Remark~\ref{rem:Rrec}). Similarly to \eqref{eq:identificationfunction}, we write for arbitrary $k_1,\cdots,k_N \in \mathbb Z$,
\begin{equation}\label{eq:schurecriture}
s_\lambda(k_1,\cdots,k_N)\equiv s_\lambda(\dot k_1,\cdots,\dot k_N)\equiv s_\lambda\left(e^{2\pi i k_1 / L},\cdots,e^{2\pi i k_N / L}\right),
\end{equation}

\begin{prop}\label{Schur}
	The subset of Schur polynomials $\{s_{\lambda} : L-N \geq \lambda_1 \geq \cdots \geq \lambda_N \geq 0\}$ in $N$ variables, restricted to the set $\mathbb U_L$ of $L$th roots of unity, forms an $\ell^2({\mathcal C_{L,N}},\psi^2\,dm)$-orthonormal eigenbasis of the MESSEP kernel given in \eqref{transition}.
\end{prop}

\begin{rem}\label{rem:corres}
	For the one-to-one correspondence between partitions $\lambda$ and configurations $\xi$, and the resulting correspondence between  $s_\lambda$ and $\psi_\xi$, we refer to the proof of Proposition \ref{Schur} in Section \ref{sec:schur} and in particular to Remark \ref{rem:minimalpartition}) and  Figure \ref{partition}.
\end{rem}

\begin{figure}[!h]
	\centering
	\includegraphics[width=5cm]{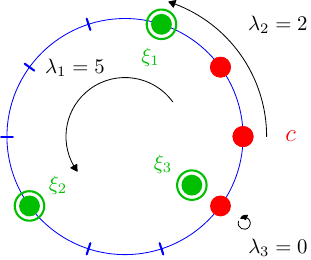}
		\captionsetup{width=10cm}
	\caption{\small  Correspondence between the partition $\lambda=(5,2,0)$ and the  configuration $\xi=(2,6,9)$ with $N=3$ and $L=10$.} 
	\label{partition}
\end{figure}

\section{Low-density limit: main results}

\label{sec:low}

This section is devoted to the low-density regime. We state convergence of the MESSEP to UDBM and its explicit spectral decomposition, the proofs are detailed  in Section~\ref{sec:sec3proof}. 

\subsection{The scaling limits}

The limiting process $(\bX(t))_{t\geq 0}$ in Theorem \ref{thm:0} and Theorem \ref{scalingthm1} below, namely the UDBM, is introduced rigorously in Section~\ref{sec:dyson}, together with its variants $(\boldsymbol{\Theta}(t))_{t\geq 0}$ and $(\boldsymbol{\Xi}(t))_{t\geq 0}$.

\begin{theo}\label{scalingthm1}
	Let $(X(n))_{n\geq 0}$ be the MESSEP on ${\mathcal W_{\mathbb Z,L,N}}$. Denote by $(X(t))_{t\geq 0}$ the continuous linear interpolation of this process, and assume that $X(0) = x$ depends on $L$ in such a way that  
	\begin{equation}\label{eq:convstaringpoint}
	\frac{2\pi x}{L} \xrightarrow[L\to\infty]{} \boldsymbol{x},
	\end{equation}
	for some $\boldsymbol{x} = (\boldsymbol{x}_1, \cdots, \boldsymbol{x}_N)$ such that $\boldsymbol{x}_1 \leq \cdots \leq \boldsymbol{x}_N \leq \boldsymbol{x}_1 + 2\pi$. Then, the following distributional convergence holds in the space of continuous functions $C(\mathbb R_+, \mathbb R^N)$: 
	\begin{equation}\label{scalinglim0}
	\left(\frac{2\pi X(L^2t)}{L}\right)_{t\geq 0} \xRightarrow[L\to\infty]{} (\boldsymbol{X}(t))_{t\geq 0},
	\end{equation}  
	where $(\boldsymbol{X}(t))_{t\geq 0}$ is the  UDBM, starting from $\bx$, satisfying the stochastic differential equations  \eqref{Dyson1} and \eqref{Dyson2},  and  taking its values in 
	\begin{equation}\label{eq:weylR}
	\overline{\boldsymbol {\mathcal W}}_{\mathbb R,2\pi,N} := \left\{\bx \in \mathbb R^N : \bx_1 \leq \cdots \leq \bx_N \leq \bx_1 + 2\pi \right\}.
	\end{equation}
\end{theo}

Theorem~\ref{scalingthm1} admits analogous statements for $(\Theta(n))_{n\geq 0}$ (resp.\ $(\Xi(n))_{n\geq 0}$), viewed as processes in $\mathbb U \subset \mathbb C$ with values in ${\mathcal W}_{L,N}^{\scriptscriptstyle \#}$ (resp.\ ${\mathcal C}_{L,N}^{\scriptscriptstyle \#}$), and for their scaling limits $(\boldsymbol{\Theta}(t))_{t\geq 0}$ (resp.\ $(\boldsymbol{\Xi}(t))_{t\geq 0}$). See Definitions \ref{defmessep} and \ref{def:merwcontinuous}, and Corollary \ref{coro:scalinggen}. The functional limits read
\begin{equation}\label{scalinglim-theta}
\left(\Theta(L^2t)\right)_{t\geq 0} \xRightarrow[L\to\infty]{} (\boldsymbol{\Theta}(t))_{t\geq 0}
\quad \left(\text{resp.} \quad  
\left(\Xi(L^2t)\right)_{t\geq 0} \xRightarrow[L\to\infty]{} (\boldsymbol{\Xi}(t))_{t\geq 0} \right).
\end{equation}
Note that the factor $1/L$ does not appear in \eqref{scalinglim-theta}, unlike in \eqref{scalinglim0}, since the scaling is already incorporated into the corresponding state spaces. We refer to \eqref{eq:weyldiscretes} and \eqref{eq:wstar} for the discrete state spaces, and to \eqref{eq:weyl1} and \eqref{eq:weyl2} for their continuous analogues ${\bCW}_{2\pi,N}$ and ${\bCC}_{2\pi,N}$, together with their $\#$-counterparts. Their compactifications are denoted by $\overline{{\bCW}}_{2\pi,N}$ and $\overline{{\bCC}}_{2\pi,N}$. The boundary corresponds, as in \eqref{eq:weylR}, to configurations with colliding particles. 

Although particles may coincide at time $0$, this almost surely never occurs for $t>0$ (see Proposition~\ref{cepa} and the discussion after Definition~\ref{def:merwcontinuous}). Roughly speaking, $(\bXi(t))_{t\geq 0}$ models indistinguishable particles and, for $t>0$, one can write
\begin{equation}
\bXi(t) = \left\{ \bTheta_1(t), \cdots, \bTheta_N(t) \right\}
= \left\{ e^{i \bX_1(t)}, \cdots, e^{i \bX_N(t)} \right\}.
\end{equation}

\subsection{Full spectral decomposition}

This section parallels Section~\ref{graphspec}. We stress the correspondence between the eigenfunction constructions for the MESSEP and the UDBM (in the indistinguishable setting), using consistent notation to make the analogy transparent.

\subsubsection{Orthonormal Hilbert eigenbasis}

Let $\mathfrak{W}_N \subset \mathcal{P}(\mathbb{Z})$ be the set of subsets of $\mathbb{Z}$ of cardinality $N$, which can be viewed as strictly ordered $N$-tuples as in \eqref{state}.  For any $m \in \mathfrak{W}_N$, define for all $x\in \mathbb{R}^N$,
\begin{equation}\label{eq:Psim}
\Psi_m(x) := \frac{1}{(2\pi)^{N/2}} \det\left(e^{i(m_k + \bgamma) x_l}\right)_{1 \leq k, l \leq N}.
\end{equation}

Note that $\Psi_m$ is antisymmetric and invariant under translation by $2\pi(1,\cdots,1)$: 
\begin{equation}\label{eq:sym1}
\forall \ell\in\mathbb Z,\quad \Psi_m(x_1 + \ell 2\pi, \cdots, x_N + \ell 2\pi) = e^{2i\ell \pi N \bgamma}\, \Psi_m(x) = \Psi_m(x),
\end{equation}
and satisfy
\begin{equation}\label{eq:sym2}
\Psi_m(x_1, \cdots, x_N) = \Psi_m(x_2, \cdots, x_N, x_1 + 2\pi).
\end{equation}
We recall that $\bgamma$ is defined in \eqref{eq:gamma}. Remark also that  $\Psi_m$ vanishes on the boundary of ${\boldsymbol{\mathcal W}}_{\mathbb{R},2\pi,N}$, i.e.\ whenever $x_i = x_j$ for some $i \neq j$ or $x_N = x_1 + 2\pi$. 

These symmetries allow us to regard $\Psi_m$ as a function on ${\boldsymbol{\mathcal{W}}}_{2\pi,N}$ (defined in \eqref{eq:weyl1}) and ${\boldsymbol{\mathcal{C}}}_{2\pi,N}$ (defined in \eqref{eq:weyl2}) and their $\mathbb U$-embeddings, still vanishing on the boundary. In particular,
\begin{equation}\label{eq:symccl}
\Psi_m(\xi_1, \cdots, \xi_N)
= \Psi_m(\overline{x}_1, \cdots, \overline{x}_N)
\simeq \Psi_m(e^{i x_1}, \cdots, e^{i x_N})
= \Psi_m(x_1, \cdots, x_N),
\end{equation}
for $\xi \in \overline{\boldsymbol{\mathcal{C}}}_{2\pi,N}$ and $x \in \overline{\boldsymbol{\mathcal{W}}}_{\mathbb{R},2\pi,N}$ such that $\xi_k = \overline{x}_k \simeq e^{i x_k}$ for $1 \leq k \leq N$, where $\overline{x}$ denotes the representative of $x$ modulo $2\pi$ in $[0,2\pi)$. 

Similarly to \eqref{compactconf} and  Remark~\ref{rem:compactrep}, set
\begin{equation}
(\bc_i)_{1\leq i\leq N} =
\left\{
\begin{array}{ll}\label{symconf}
(-p, \cdots, 0, \cdots, p),   & \text{if $N = 2p+1$},\\[10pt]
(-p, \cdots, 0, \cdots, p-1), & \text{if $N = 2p$}.
\end{array}
\right.
\end{equation}
Since $\bc\in \mathfrak{W}_N$ is obtained from $\delta = (N-1, \cdots, 0)$ by permutation and translation, one can check that $\Psi_{\boldsymbol c}(x) = \varepsilon\, \Psi(x)$ for some $\varepsilon\in\mathbb U$  with
\begin{equation}\label{eq:Psi}
\Psi(x) := \frac{2^{\frac{N(N-1)}{2}}}{(2\pi)^{N/2}} \prod_{1 \leq i < j \leq N} \sin\left( \frac{x_j - x_i}{2} \right).
\end{equation}

Finally, consider 
\begin{equation}\label{eq:Psicarre}
\boldsymbol{\mu}(dx) = \Psi^2(x)\,\mathds 1_{\bCC_{2\pi,N}}(x)\,\lambda(dx),
\end{equation}
where $\lambda$ is the $N$-dimensional Lebesgue measure. Note that $\boldsymbol{\mu}$ coincides with the joint distribution of the eigenangles of Haar unitary matrices, i.e.\ the circular $\beta$-ensemble with $\beta=2$ (CUE).

\begin{prop}\label{orthonormal}
	The family $\{\Psi_m : m \in \mathfrak{W}_N\}$ is an orthonormal basis of $\mathbb{L}^2(\bCC_{2\pi,N}, \lambda)$. Moreover,
	\begin{equation}\label{inclusion}
	\boldsymbol{\mathcal{S}}_{2\pi,N} := \left\{ \frac{\Psi_m}{\Psi_{\bc}} : m \in \mathfrak{W}_N \right\},
	\end{equation}
	is an orthonormal basis of $\mathbb{L}^2(\bCC_{2\pi,N}, \boldsymbol{\mu})$, consisting of bounded smooth functions on $\overline{\bCC}_{2\pi,N}$. Its linear span is dense in $C(\overline{\bCC}_{2\pi,N})$ for the uniform topology.
\end{prop}

\subsubsection{Spectral decomposition, smoothing properties and contractivity}

Let $(\bP_t)_{t\geq 0}$ be the Markov semigroup of $(\bXi(t))_{t\geq 0}$, defined for any measurable nonnegative $f$ on $\overline{\bCC}_{2\pi,N}$ and $\bxi \in \overline{\bCC}_{2\pi,N}$ by
\begin{equation}\label{semigroup}
\bP_t f(\bxi) = \mathbb{E}_{\bxi}[f(\boldsymbol{\Xi}(t))].
\end{equation}

\begin{rem}\label{rem:symper}
	Equivalently, given $\bx\in \overline{\boldsymbol{\mathcal W}}_{\mathbb R,2\pi,N}$  such that $\overline{\bx}_k=\bxi_k$ for $1\leq k\leq N$, one has, for every symmetric and componentwise $2\pi$-periodic function $g$ on $\mathbb R^N$,
	\begin{equation}\label{semigroup2}
	\bP_t g(\bxi) = \mathbb{E}_{\bx}[g(\boldsymbol{X}(t))].
	\end{equation}
\end{rem}

It turns out that family $\boldsymbol{\mathcal{S}}_{2\pi,N}$ forms a complete set of eigenfunctions of $(\bP_t)_{t\geq 0}$, and $\boldsymbol{\mu}$ is the corresponding invariant probability measure. More exactly, for any $m \in \mathfrak{W}_N$, set
\begin{equation}\label{energyspectrum}
E_m = \frac{2\pi^2}{N} \sum_{i=1}^N \left( (m_i + \bgamma)^2 - (\boldsymbol{c}_i + \bgamma)^2 \right).
\end{equation}
Note that  $E_m \in \{2\pi^2 n / N : n \geq 0\}$, $E_m=0$ if and only if $m=\boldsymbol{c}$, and
\begin{equation}\label{eq:spectralgap}
\inf_{m \neq \boldsymbol{c}} E_m = 2\pi^2.
\end{equation}
Then, the eigenvalues of the infinitesimal generator of the UDBM associated with the eigenfunctions in \eqref{inclusion} are precisely the $E_m$. The next result synthesizes and extends these statements, it is the continuous counterpart of the spectral decomposition \eqref{eq:fullspectralmarkov}.

\begin{theo}\label{spectraldecompodyson}
	The distribution $\boldsymbol{\mu}$ is the unique invariant probability measure of the continuous-time Markov process $(\boldsymbol{\Xi}(t))_{t \geq 0}$. Moreover, $\boldsymbol{\mu}$ is a reversible measure for this process, and for all $f \in \mathbb{L}^2(\bCC_{2\pi,N}, \boldsymbol{\mu})$, $\bxi \in \overline{\bCC}_{2\pi,N}$, and $t \geq 0$, one has
	\begin{equation}\label{decompo}
	\bP_t f(\bxi) = \sum_{m \in \mathfrak{W}_N} e^{-E_m t} \left( \int_{\bCC_{2\pi,N}} f\, \overline{\Psi_m} \Psi_{\bc}\, d\lambda \right) \frac{\Psi_m(\bxi)}{\Psi_{\bc}(\bxi)}.
	\end{equation}
\end{theo}

As a consequence of the above theorem and of \eqref{eq:spectralgap}, we may state:

\begin{coro}\label{coro:convergenceL2}
	The semigroup $(\bP_t)_{t \geq 0}$ maps $\mathbb{L}^2(\bCC_{2\pi,N}, \boldsymbol{\mu})$ continuously into ${\rm C}^\infty(\overline{\bCC}_{2\pi,N})$ for any $t > 0$. Moreover, for all $f \in \mathbb{L}^2(\bCC_{2\pi,N}, \boldsymbol{\mu})$ and $t \geq 0$, one has
	\begin{equation}\label{speed}
	\|\bP_t f - \boldsymbol{\mu} f\|_{\mathbb{L}^2(\bCC_{2\pi,N}, \boldsymbol{\mu})} 
	\leq e^{-2\pi^2 t} \|f - \boldsymbol{\mu} f\|_{\mathbb{L}^2(\bCC_{2\pi,N}, \boldsymbol{\mu})}.
	\end{equation}
\end{coro}

%

\subsubsection{Schur polynomial eigenfunctions}

A closely related result to Proposition~\ref{Schur} appears here: the eigenfunctions of $(\bP_t)_{t\geq 0}$ are Schur polynomials on $\mathbb U^N$ multiplied by an exponential factor involving the center of mass of the angles, as in the Calogero--Sutherland Hamiltonian \eqref{eq:calo}.

\begin{prop}\label{prop:schureigen} One has 
\begin{equation}\label{eq:schursetequal}
	\boldsymbol{\mathcal{S}}_{2\pi,N}=
\left\{ e^{i\delta \sum_{i=1}^N x_i}s_\lambda\left(e^{i x_1},\cdots,e^{i x_N}\right) : \delta\in \mathbb Z,\, \lambda_1\geq \cdots\geq \lambda_N\geq 0\right\}.
\end{equation}	
\end{prop}

\begin{proof}[Proof of Proposition \ref{prop:schureigen}]
	As in the proof of Proposition~\ref{Schur} in Section~\ref{sec:schur},  one checks that every element of $\boldsymbol{\mathcal{S}}_{2\pi,N}$ can be written as in the right-hand side of \eqref{eq:schursetequal}. Conversely, given a partition $\lambda$ and $\delta \in \mathbb{Z}$, equality
	\begin{equation}\label{schureigen}
	\frac{\Psi_m(x)}{\Psi_{\bc}(x)} = e^{i\delta \sum_{i=1}^N x_i}\, s_\lambda\left(e^{i x_1}, \cdots, e^{i x_N}\right),
	\end{equation}
	holds by choosing $\ell \in \mathbb{Z}$ such that $p + \ell = \delta$ (recall that $N = 2p$ or $N = 2p + 1$), and by setting $m_i = \lambda_{N+1-i} + i - 1 + \ell$ for $1 \leq i \leq N$. Observe that $m_1 < \cdots < m_N$, so $m \in \mathfrak{W}_N$.
\end{proof}

\begin{rem}\label{rem:mini}
	The pair $(\lambda,\delta)$ is uniquely determined under the normalization $\lambda_N=0$. Such a $\lambda$ is called the minimal partition associated with $m\in\mathfrak{W}_N$, in analogy with Definition~\ref{def:minimal}. If $\lambda$ in \eqref{eq:schursetequal} is allowed to have negative parts (a signature, also called a mixed partition or staircase), one may take $\delta=0$, and $\lambda$ is again uniquely determined. In this case, $s_\lambda(z_1,\cdots,z_N)$ is a rational, rather than polynomial, symmetric function. See the proof of point~3 below \eqref{eq:prodsignature}.
\end{rem}

\section{Hydrodynamic limit: main results}
\label{sec:hydro}

This section presents the hydrodynamic limit of the MESSEP. We state the limiting PDEs for the density and its moment generating function, outline the strategy of proof, and highlight the main structural ingredients of the method. Detailed proofs are deferred to Section~\ref{sec:proofhydro}.

\subsection{Preliminaries}

We assume that $N\sim \alpha L$ for some $\alpha\in(0,1)$ as $L\to\infty$. Throughout this section, $L,N\to\infty$ refers to this joint limit. We study the convergence of the empirical (occupation) measure
\begin{equation}\label{eq:occupmeasure}
\upsilon_{L,N}(t,d\theta)
=\frac{1}{N}\sum_{k=1}^N\delta_{e^{2i\pi X_k(L^2 t)/L}}(d\theta)
=\frac{1}{N}\sum_{k=1}^N\delta_{\Theta_k(L^2 t)}(d\theta)
=\frac{1}{N}\sum_{k=1}^N\delta_{\Xi_k(L^2 t)}(d\theta).
\end{equation}

By a slight abuse of notation, we write $\upsilon_{L,N}(t,dx)$ for the associated $2\pi$-periodic probability measure on $\mathbb R/2\pi\mathbb Z\simeq[0,2\pi)$. Functions on the unit circle and their $2\pi$-periodic counterparts on $\mathbb R$ or $\mathbb R/2\pi\mathbb Z$ will be identified whenever convenient. In \eqref{eq:occupmeasure}, $(\Theta(n))_{n\geq 0}$ and $(\Xi(n))_{n\geq 0}$ are the $\mathbb U$-embedding versions of the MESSEP random walks.

\begin{defin}\label{def:cvgprob}
We say that $\upsilon_{L,N}(t,dx)$ converges in probability to a (possibly random) probability measure $\upsilon(t,dx)$ if, for every continuous function $\varphi$ on $\mathbb R/2\pi\mathbb Z$, one has
\begin{equation}\label{eq:convprob}
\forall \varepsilon>0,\quad
\lim_{L,N\to\infty}
\mathbb P\!\left(
\left|\int_0^{2\pi}\varphi(x)\,\upsilon_{L,N}(t,dx)
-\int_0^{2\pi}\varphi(x)\,\upsilon(t,dx)\right|
>\varepsilon
\right)
=0.
\end{equation}
\end{defin}

As a preliminary observation, any limit point of the empirical measures must be absolutely continuous, as explained in the following remark.

\begin{rem}\label{rem:packed}
	Since the proportion of particles lying in an interval $(a,b)\subset[0,2\pi)$ satisfies
	\begin{equation}
	\frac{\#\{1\leq k\leq N:\ X_k(L^2 t)\in (a,b)\;\mathrm{mod}\;2\pi\}}{N}
	\leq \frac{(b-a)L}{2\pi N},
	\end{equation}
	any limit point $\upsilon(t,dx)$ necessarily admits a random density $f(t,x)$ satisfying
	\begin{equation}\label{eq:bounddesnity}
	0\leq f(t,x)\leq \frac{1}{2\pi\alpha}.
	\end{equation}
\end{rem}

Given a density $f(t,x)$ as previously, we denote by $\mathcal H f(t,x)$ its spacial Hilbert transform, defined either in the Cauchy principal value sense or, equivalently, in Fourier series form by
\begin{equation}\label{eq:hilbert}
\mathcal H f(t,x)
=\operatorname{p.v.}\,\frac{1}{2\pi}\int_{-\pi}^{\pi}
\cot\!\left(\frac{y-x}{2}\right) f(t,y)\,dy 
\quad \left(= -i\sum_{n\in\mathbb Z}{\rm sgn}(n)\,c_n(t)e^{inx}\right),
\end{equation}
whenever $f(t,x)=\sum_{n\in\mathbb Z} c_n(t) e^{inx}$, with ${\rm sgn}(0)=0$. Finally, we define the complex moments $(\mathfrak m_n(t))_{n\geq1}$ of a density $f(t,x)$, together with the associated moment generating function, by 
\begin{equation}
\mathfrak m_n(t)
=\frac{1}{2\pi}\int_0^{2\pi} e^{inx}f(t,x)\,dx
=\int_{\mathbb U}\theta^n f(t,\theta)\,d\theta
\quad\text{and}\quad
g(t,z)=\sum_{n=1}^\infty \mathfrak m_n(t)\,z^n.
\end{equation}
Here $d\theta$ denotes the Lebesgue measure on the unit circle, normalized to be a probability measure. We write $g_0(z)=g(0,z)$ and $f_0(x)=f(0,x)$. The series $g(t,z)$ is absolutely convergent in the open unit disk $\mathbb D$ when $f(t,x)$ satisfies \eqref{eq:bounddesnity}.

%
%

\subsection{Hydrodynamic PDEs for the density and the moments}

We now state the main result of this paper, describing the macroscopic evolution of the particle density and of the moment generating function.

\begin{theo}\label{thm:hydro}
	Assume that $\upsilon_{L,N}(0,dx)$ converges in probability to a (possibly random)  density profile $f_0(x)$. Then $(\upsilon_{L,N}(t,dx))_{t\geq0}$ converges in probability to a family of (random) densities $(f(t,x))_{t\geq0}$ satisfying, in a weak sense, the nonlinear and nonlocal PDE
	\begin{equation}\label{PDE}
	\frac{\partial f}{\partial t}
	+\frac{1}{\alpha\sin(\pi\alpha)}\,
	\frac{\partial\,\mathcal J_\alpha f}{\partial x}
	=0,
	\quad\text{with}\quad
	\mathcal J_\alpha f
	=\sin\!\left(2\pi^2\alpha\,f\right)\,
	\sinh\!\left(2\pi^2\alpha\,\mathcal H f\right).
	\end{equation}

Equivalently, the associated (random) moment generating function $g(t,z)$ is the unique solution on $[0,\infty)\times\mathbb D$ of the nonlinear transport equation
	\begin{equation}\label{eq:derivative0gbis000000}
	\frac{\partial g(t,z)}{\partial t}
	+z\,\mathcal V_\alpha\left(g(t,z)\right)\,
	\frac{\partial g(t,z)}{\partial z}
	=0,
	\quad\text{with}\quad
	\mathcal V_\alpha(z)
	=2\pi^2\,\frac{\sin\!\left(\pi\alpha(1+2z)\right)}{\sin(\pi\alpha)}.
	\end{equation}
\end{theo}

As explained in the introduction, the regularity of $f(t,x)$ depends on whether the initial profile saturates the bounds~\eqref{eq:bounddesnity}, corresponding to macroscopic voids or fully packed regions. In all cases, the dynamics exhibit finite-time regularization, followed by exponential convergence to equilibrium.

\Needspace{2\baselineskip}
\begin{theo}\label{prop:regularityPDE}
	The PDE~\eqref{PDE} enjoys the following  properties:
	\begin{enumerate}
		\item Assume that there exists $\eta>0$ such that, for all $x\in\mathbb R/2\pi\mathbb Z$,
		\begin{equation}\label{eq:upperlowerf0}
		\eta \leq f_0(x)\leq \frac{1}{2\pi\alpha}-\eta.
		\end{equation}
		Then the solution $(t,x)\mapsto f(t,x)$ is analytic on $(0,\infty)\times \mathbb R/2\pi\mathbb Z$ and satisfies~\eqref{PDE} in the strong sense on this domain.
		
		\item For any initial density profile $f_0(x)$, the following hold:
		\begin{enumerate}
			\item There exists $t^\ast>0$ such that $(t,x)\mapsto f(t,x)$ is analytic on $(t^\ast,\infty)\times \mathbb R/2\pi\mathbb Z$ and satisfies~\eqref{PDE} in the strong sense on this domain.
			\item The solution relaxes exponentially fast toward the uniform equilibrium, that is
			\begin{equation}
			\limsup_{t\to\infty} \frac{1}{t}\log\!\left(\left\|f(t,\cdot)-\frac{1}{2\pi}\right\|_\infty\right)<0.
			\end{equation}
		\end{enumerate}
	\end{enumerate}
\end{theo}

As shown in Section~\ref{sec:pdet0}, analytic initial data $f_0$ yield, locally in time, analytic solutions to~\eqref{PDE}. However, as already emphasized in the introduction, such regularity may break down at finite time, even when $f_0$ is entire. We refer  to the  example of Section~\ref{sec:periodicdens}.

We now briefly recall -- still building on the overview given in the introduction -- the mechanism linking~\eqref{PDE} and~\eqref{eq:derivative0gbis000000}. As explained there,  the density is recovered through the identity~\eqref{eq:linkgfh} via the method of characteristics (see also Section~\ref{sec:charc}). More precisely, Lemma~\ref{lem:tech1} provides a decreasing family of open sets $V_t\subset\mathbb D$, with $V_0=\mathbb D$ and $\bigcap_{t\ge0}V_t=\{0\}$, together with conformal maps $w(t,\cdot)$ from $\mathbb D$ onto $V_t$ satisfying
\begin{equation}\label{eq:implicit}
\Phi_t(w(t,z))=w(t,z)\,e^{tA_0(w(t,z))}=z,
\quad\text{with}\quad
A_0(w)=\mathcal V_\alpha\bigl(g_0(w)\bigr).
\end{equation}
The method of characteristics yields $g(t,z)=g_0(w(t,z))$. By~\eqref{eq:linkgfh}, recovering $f(t,x)$ amounts to extending $w(t,\cdot)$ to $\partial V_t$, leading to the representation~\eqref{eq:represent}. The regularity of $f(t,x)$ is thus governed by that of $\partial V_t$, as described in the introduction and illustrated in Figure~\ref{fig:boundary}.

\begin{rem}
	The method of characteristics is not merely conceptual: it also provides an effective numerical procedure to compute $f(t,x)$ from the initial generating function $g_0$ via the geometry of the associated level sets. This is the method used to obtain the density evolution in Figure~\ref{densites}.
\end{rem}

Any point $w_0=e^{ix_0}\in \partial V_t\cap\partial\mathbb D$ corresponds to an extremal values of $f_0(x_0)$. At such points, irregularities may arise, whereas points in $\partial V_t\cap\mathbb D$ correspond to non-extremal values and yield smooth evolution. A detailed analysis is given in Propositions~\ref{prop:anal}, \ref{prop:cont}, and~\ref{prop:deriv}.

We now turn to the maximally packed initial configuration
\begin{equation}\label{eq:stepprofile}
f_0(x)=\frac{1}{2\pi\alpha}\,\mathds{1}_{[-\pi\alpha,\pi\alpha]}(x),
\end{equation}
which plays the role of a Dirac mass in this constrained setting and exhibits sharp regularity and support properties. Setting $w_\alpha=e^{i\pi\alpha}$, one computes
\begin{equation}\label{eq:g0A0}
g_0(z)=\frac{1}{2i\pi\alpha}\log\left(\frac{1-z\,\overline{w_\alpha}}{1-z\,w_\alpha}\right)
\quad\text{and}\quad 
A_0(w)=2\pi^2\,\frac{(1-w)(1+w)}{(w-w_\alpha)(w-\overline{w_\alpha})}.
\end{equation}

By symmetry and particle--hole duality (see Remark~\ref{rem:alphasym}), we may assume $\alpha\leq 1/2$. The regularity properties of $f(t,x)$ are stated below. They are illustrated in Figure~\ref{densites}, while the geometry of the associated evolving domains $V_t$, which governs these properties, is depicted in Figure~\ref{fig:levelsetstep}. The proof of the following result is given in Section \ref{sec:packedproof}.

\begin{theo}\label{thm:pack}
	Let $f_0$ be given by~\eqref{eq:stepprofile}, and let $w(t,\cdot):\mathbb D\to V_t$ be the associated conformal map. For every $t>0$, the map $w(t,\cdot)$ extends continuously to $\partial\mathbb D$, and its image is a piecewise analytic Jordan curve $\partial V_t\subset\overline{\mathbb D}\setminus\{\omega_\alpha,\overline{\omega_\alpha}\}$. The density $f(t,x)$ defined by~\eqref{eq:represent} is continuous on $(0,\infty)\times\mathbb R/2\pi\mathbb Z$. Moreover, the time evolution of $f(t,x)$ admits the following description. 
	
	There exist explicit times $0<t_\ast\leq t^\ast$  (see identities \eqref{eq:tempscritiques}), with $t_\ast<t^\ast$ if $\alpha\neq 1/2$ and $t_\ast=t^\ast$ otherwise, together with smooth functions (see Section \ref{sec:final})
	\begin{equation}
	\phi_1:(0,t_\ast)\to(0,\pi\alpha),
	\qquad 
	\phi_2:(0,t^\ast)\to(\pi\alpha,\pi),
	\end{equation}
	such that $\phi_1$ decreases from $\pi\alpha$ to $0$, while $\phi_2$ increases from $\pi\alpha$ to $\pi$. The evolution then exhibits three distinct regimes:
	
	\begin{enumerate}
		\item For all $t\in(0,t_\ast]$, one has
		\begin{equation}
		f(t,x)=
		\begin{cases}
		\dfrac{1}{2\pi\alpha}, & \text{if } x\in[-\phi_1(t),\phi_1(t)],\\[10pt]
		0, & \text{if } x\notin(-\phi_2(t),\phi_2(t)),
		\end{cases}\quad\text{and}\quad 0<f(t,x)<\frac{1}{2\pi\alpha}\; \text{elsewhere},
		\end{equation}
		while  $f(t,x)$ is analytic except along the curves $(t,\pm\phi_1(t))$ and $(t,\pm\phi_2(t))$.
		
		\item For all $t\in(t_\ast,t^\ast]$, one has $f(t,x)=0$, if  $x\notin(-\phi_2(t),\phi_2(t))$, and $0<f(t,x)<1/(2\pi\alpha)$ elsewhere, while $f(t,x)$ is analytic except along the curves $(t,\pm\phi_2(t))$.
		
		\item For all $t\in(t^\ast,\infty)$, one has $0<f(t,x)<1/(2\pi\alpha)$ and $f(t,x)$ is analytic.
	\end{enumerate}
\end{theo}

\begin{rem}
	At points where $x\mapsto f(t,x)$ fails to be differentiable, the singularities are of square-root or cubic-root type. Further details are given in the corresponding parts of the proofs.
\end{rem}

\subsection{Sketch of the proof and key lemmas}
\label{sec:sketch}

The proof relies on the method of moments and, as in the low-density regime, crucially uses that the eigenfunctions of the MESSEP kernel are Schur polynomials restricted to $\mathbb U_L^N$. By a standard density argument, convergence in probability~\eqref{eq:convprob} reduces to the convergence of all complex moments  of the empirical measure. This allows us to work entirely at the level of moments, where the spectral structure of the dynamics is most effective.

\subsubsection{Schur and power-sum bases: Frobenius formulas}

A key ingredient  is the relation between Schur polynomials and the power-sum symmetric polynomials. The latter form a linear basis of symmetric polynomials indexed by partitions $\pi$:
\begin{equation}\label{eq:powersums}
p_\pi = \prod_{i \geq 1} p_{\pi_i} = \prod_{k \geq 1} p_k^{l_k}, 
\quad \text{with} \quad 
p_n(X) = X_1^n + \cdots + X_N^n,
\end{equation}
where $l_k$ denotes the multiplicity of $k$ in $\pi$. We write $\pi\vdash n$ when $n=\sum_i \pi_i=\sum_k l_k k$, and we set $\ell(\pi)=\sum_k l_k$. Note that $p_n$ extends to negative integers as a rational function.

Irreducible representations of $\mathfrak S_n$ are also indexed by partitions $\lambda\vdash n$. We denote by $\chi^\lambda_\pi$ the value of the character on permutations of cycle type $\pi$. The Frobenius character formulas are
\begin{equation}\label{eq:frobenius}
p_\pi = \sum_{\lambda \vdash n} \chi_\pi^\lambda\, s_\lambda 
\quad \text{and} \quad 
s_\lambda = \sum_{\pi \vdash n} \frac{\chi_\pi^\lambda}{z_\pi} \, p_\pi, 
\quad \text{with} \quad 
z_\pi = \prod_{k \geq 1} k^{l_k} l_k!.
\end{equation}

The Murnaghan--Nakayama rule (see \cite{Macdonald,stanley1999enumerative}) is an important tool to compute characters.  For the hook partition $\{n|k\}:=(n-k,1^k):=(n-k,1,\cdots,1)$ (see Figure \ref{fig:confighook}) and $\pi=(n)$, it is well known that  $\chi^{\{n|k\}}_{(n)}=(-1)^k$. It then follows from~\eqref{eq:frobenius} that
\begin{equation}\label{eq:powersumhook}
p_n = \sum_{k=0}^{n-1} (-1)^k s_{\{n|k\}}.
\end{equation}


\subsubsection{Complex moments and their expectation}

For all $n\in\mathbb Z$, we denote by ${\mathcal M}_n^{L,N}(t)$ the random $n$th complex moment of  $\upsilon_{L,N}(t,d\theta)$  and by $\mu_n^{L,N}(t)$ its conditional expectation given $\Theta(0)=\vartheta$. Equivalently, $\mu_n^{L,N}(t)$ is the $n$th moment of the mean occupation measure $\overline{\upsilon}_{L,N}(t,d\theta)=\mathbb{E}_{\vartheta}[\upsilon_{L,N}(t,d\theta)]$. With these notations,
\begin{equation}
{\mathcal M}_n^{L,N}(t) = \frac{p_n(\Theta(L^2 t))}{N} = \int_{\mathbb{U}} \theta^n\, \upsilon_{L,N}(t,d\theta) = \frac{1}{2\pi}\int_0^{2\pi} e^{i n x}\, \upsilon_{L,N}(t,dx),
\end{equation}
and
\begin{equation}\label{eq:expectedmoment}
\mu_n^{L,N}(t) = \mathbb{E}_{\vartheta}\left[ {\mathcal M}_n^{L,N}(t) \right] = \int_{\mathbb{U}} \theta^n\, \overline \upsilon_{L,N}(t,d\theta) = \frac{1}{2\pi}\int_0^{2\pi} e^{i n x}\, \overline \upsilon_{L,N}(t,dx).
\end{equation}

We assume that the initial empirical measure $\upsilon_{L,N}(0,\cdot)$ associated with $\vartheta$ converges in probability, in the sense of~\eqref{eq:convprob}, to a random probability measure $f_0(\theta)d\theta$ on $\mathbb U$. Denoting by $m_n$ its (random) $n$th complex moment, one has the convergence in probability
\begin{equation}\label{eq:momentinitit}
\frac{p_n(\vartheta)}{N} = \frac{\vartheta_1^n + \cdots + \vartheta_N^n}{N} \xrightarrow[L,N \to \infty]{} m_n = \int_{\mathbb{U}} \theta^n  f_0(\theta)d\theta = \frac{1}{2\pi}\int_0^{2\pi} e^{i n x} f_0(x)dx.
\end{equation}

Observe that ${\mathcal M}_{-n}^{L,N}(t)$ (resp.\ $\mu_{-n}^{L,N}(t)$) is the complex conjugate of ${\mathcal M}_n^{L,N}(t)$ (resp.\ $\mu_n^{L,N}(t)$). It therefore suffices to consider $n\geq 0$.  To go further, using~\eqref{eq:powersumhook}, one obtains
\begin{equation}\label{eq:sanspartieentiere}
\mu_n^{L,N}(t) = \frac{1}{N} \sum_{k=0}^{n-1} (-1)^k\, \br_{\{n|k\}}^{\lfloor L^2 t \rfloor} \, s_{\{n|k\}}(\vartheta) + \mathcal{O}\!\left( \frac{1}{N L} \right),\quad\text{with}\quad  \br_{\{n|k\}}=\frac{\rho_\xi}{\rho_c}.
\end{equation}
Recall that $\rho_\xi/\rho_c$ denotes the MESSEP eigenvalue associated with the configuration $\xi$ corresponding to the hook partition $\{n|k\}$. We refer to Proposition~\ref{Schur}, Remark~\ref{rem:corres}, and Figure~\ref{fig:confighook}.

\begin{figure}[!h]
	\centering
	\includegraphics[width=7.5cm]{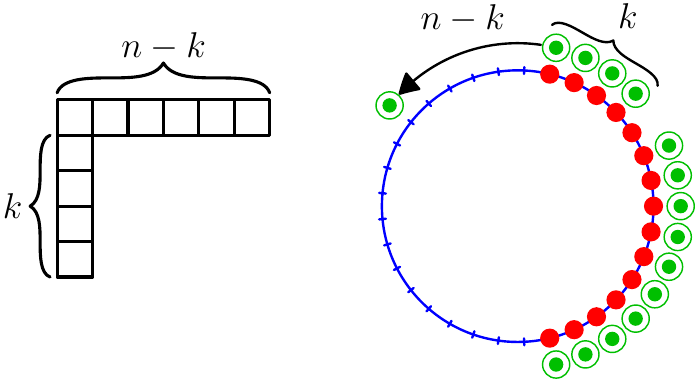}
	\captionsetup{width=13.9cm}
	\caption{\small  Correspondence between  hook partition $\{n|k\}$ and  configuration $\xi$ (circled dots). The points on the circle represent the symmetric compact configuration $\bc$.}
	
	\label{fig:confighook}
\end{figure}


\subsubsection{The key character and technical lemmas}

We now outline the derivation of the limit of $\mu_n^{L,N}(t)$ in~\eqref{eq:sanspartieentiere}. We begin with a deliberately flawed argument, which exposes the main obstacles and clarifies the essential role of the technical Lemma~\ref{lem:eigenhookdevasym} and the algebraic Lemma~\ref{lem:characteridentities}.

\bigskip

\noindent
\textit{A first heuristic but ultimately misleading approach.} Standard estimates used in the proof of Proposition~\ref{gapestimate} in Section~\ref{sec:gap} (see also \eqref{eq:hookeigenvalueprecise}) allow us to write
\begin{equation}\label{eq:eigenhook}
\br_{\{n|k\}}^{\lfloor L^2 t\rfloor}=\left(1-\frac{2\pi^2 n}{L^2}+o\left
(\frac{1}{L^2}\right)\right)^{\lfloor L^2 t\rfloor}
\xrightarrow[{L,N\to\infty}]{} 
e^{-2\pi^2 n t}.
\end{equation}
However, it would be premature to conclude that
\begin{align}
\lim_{L,N\to\infty}  \mu_n^{L,N}(t)
&\overset{\eqref{eq:sanspartieentiere}}{=} \lim_{L,N\to\infty} \frac{1}{N} \sum_{k=0}^{n-1} (-1)^k\, \br_{\{n|k\}}^{\lfloor L^2 t \rfloor} \, s_{\{n|k\}}(\vartheta)
\label{eq:avantfalse} \\[1ex]
&\overset{\eqref{eq:eigenhook}}{=} e^{-2\pi^2 n t} \lim_{L,N\to\infty} \frac{1}{N} \sum_{k=0}^{n-1} (-1)^k \, s_{\{n|k\}}(\vartheta)
\label{eq:false} \\[1ex]
&\overset{\eqref{eq:powersumhook}}{=}\;e^{-2\pi^2 n t} \lim_{L,N\to\infty} \frac{p_n(\vartheta)}{N}\\
&
\overset{\eqref{eq:momentinitit}}{=} e^{-2\pi^2 n t} \,  m_n. 
\label{eq:limite_mu}
\end{align}
Since $s_{\{n|k\}}(\vartheta)$ depends on $L$ and $N$ and is not uniformly bounded, invoking~\eqref{eq:eigenhook} to pass from~\eqref{eq:avantfalse} to~\eqref{eq:false} is not justified here. It turns out that  the conclusion~\eqref{eq:limite_mu} is indeed incorrect.

\begin{rem}
One can observe that the (incorrect) moment formula \eqref{eq:limite_mu} would lead to the classical Laplace equation on the cylinder $\mathbb{R}^+ \times \mathbb{R}/2\pi\mathbb{Z}$:
	\begin{equation}
	\frac{\partial^2 f}{\partial t^2} + (2\pi^2)^2 \frac{\partial^2 f}{\partial x^2} = 0.
	\end{equation}
\end{rem}

 \bigskip

\noindent
\textit{Asymptotic expansion of the hook-Schur eigenvalues.}
To compute rigorously the limit in~\eqref{eq:avantfalse}, we need to refine the approximation~\eqref{eq:eigenhook}. A straightforward  computation shows that
\begin{equation}\label{eq:hookeigenvalueprecise}
\br_{\{n|k\}}=1-\frac{2\sin\left(\frac{\pi}{L}\right)\sin\left(\frac{\pi n}{L}\right)\sin\left(\frac{\pi N}{L}+\frac{\pi(n-2k-1)}{L}\right)}{\sin\left(\frac{\pi N}{L}\right)}.
\end{equation}

\begin{lem}\label{lem:eigenhookdevasym}
	Let $T\geq 0$ and $n,d\geq 1$ be fixed. For all $0\leq t\leq T$ and $0\leq k\leq n-1$, one has
	\begin{equation}\label{eq:expansiorhoschurhook}
	\br_{\{n|k\}}^{\lfloor L^2 t\rfloor}=\sum_{j=0}^{d-1} \frac{a^{L,N}_{n,j}(t)}{j!}\left(\frac{\pi(n-2k-1)}{L}\right)^j+\mathcal O\left(\frac{1}{L^d} \right),
	\end{equation}
	as $L,N\to\infty$ (still under the assumption $N\sim \alpha L$), where the implicit constant in the $\mathcal O$-term depends only on $T$, $n$, and $d$.  Moreover, define
	\begin{equation}\label{eq:deriventh}
	h_t(x)=\exp\left(-2\pi^2 t \frac{\sin( \alpha \pi + x)}{\sin( \alpha \pi )}\right).
	\end{equation}
	Then, uniformly with respect to $0\leq t\leq T$, one has
	\begin{equation}\label{eq:deriventh0}
	\lim_{L,N\to\infty} a_{n,j}^{L,N}(t)=
	\left.\frac{\partial^j h_t^n(x)}{\partial x^j}
	\right|_{x=0}.
	\end{equation}
\end{lem}

\begin{rem}
	Note that the velocity field $\mathcal V_\alpha$ introduced in~\eqref{eq:derivative0gbis000000} appears in~\eqref{eq:deriventh}.
\end{rem}

Using Lemma \ref{lem:eigenhookdevasym}, we proceed as follows: we rewrite the right-hand side of \eqref{eq:avantfalse} using the Frobenius formula applied to $s_{\{n|k\}}$, and then apply the expansion \eqref{eq:expansiorhoschurhook} with $d:=\ell(\pi)$. For simplicity and readability, we omit the dependence on $\vartheta$. We obtain
\begin{multline} \label{eq:avoidmistake}
\frac{1}{N} \sum_{k=0}^{n-1} (-1)^k\, \br_{\{n|k\}}^{\lfloor L^2 t \rfloor} \, s_{\{n|k\}}=\sum_{\pi\vdash n}\frac{1}{z_\pi}\left(\sum_{k=0}^{n-1}(-1)^k \, \br_{\{n|k\}}^{\lfloor L^2 t \rfloor} \,\chi_\pi^{\{n|k\}}\right)\frac{p_\pi}{N}=\\
\sum_{\pi\vdash n}\frac{1}{z_\pi}\Bigg(\sum_{j=0}^{\ell(\pi)-1} \frac{a^{L,N}_{n,j}(t) \pi^j}{j!} \underbrace{\left(\sum_{k=0}^{n-1}(-1)^{k}\,{\chi_\pi^{\{n|k\}}} \left({n-2k-1}\right)^j\right)}_{\displaystyle =\mathfrak X_{\pi}^{(j)}}\Bigg) 
\frac{p_\pi}{N L^j} +  \mathcal O\left(\frac{1}{L}\right). 
\end{multline}

The final $\mathcal{O}$-term in \eqref{eq:avoidmistake} originates from the asymptotic estimate of $p_\pi(\vartheta)$ for partitions $\pi\vdash n$. Indeed, by \eqref{eq:momentinitit}, one has
$p_\pi(\vartheta)\sim m_\pi N^{\ell(\pi)}$,  and since $N\sim \alpha L$, this yields
\begin{equation}\label{eq:boundppi}
\frac{p_\pi(\vartheta)}{N L^{\ell(\pi)}}=\mathcal O\left(\frac{1}{L}\right).
\end{equation}

Thus, the crucial step in computing the limit of the mean empirical moments is the evaluation of the alternating weighted sum of hook characters in~\eqref{eq:avoidmistake}, denoted by $\mathfrak X_{\pi}^{(j)}$. This is the content of the next lemma. To our knowledge, the resulting identity is new and relies essentially on the hook structure and the recursive Murnaghan--Nakayama rule.

\begin{lem}\label{lem:characteridentities}
	For any partition $\pi \vdash n$ and integer $0\leq j\leq \ell(\pi)-1$, one has
	\begin{equation}\label{lem:character}
\mathfrak X_{\pi}^{(j)} =	\sum_{k=0}^{n-1} (-1)^k \chi_{\pi}^{\{n|k\}} (n - 2k - 1)^j =
	\begin{cases}
	\displaystyle 2^{\ell(\pi) - 1} (\ell(\pi) - 1)! \prod_{i=1}^{\ell(\pi)} \pi_i, & \text{if }  j = \ell(\pi) - 1,\\[5pt]
	0, & \text{if }  0 \leq j \leq \ell(\pi) - 2.
	\end{cases}
	\end{equation}
\end{lem}


Then, combining Lemma~\ref{lem:characteridentities} with \eqref{eq:avoidmistake} and \eqref{eq:deriventh}, a straightforward computation yields
\begin{equation}\label{eq:momentasymptotic}
\lim_{L,N\to\infty} \mu_n^{L,N}(t) = \underbrace{\sum_{\pi\vdash n}
	\frac{(2\pi\alpha)^{\ell(\pi)-1}}{\prod_{i}l_i!}\,
	\left.\frac{\partial^{\ell(\pi)-1} h^n_t(x)}{\partial x^{\ell(\pi)-1}}
	\right|_{x=0} \,m_\pi}_{\displaystyle =\mathfrak m_n(t)}.
\end{equation}

\begin{rem}
	Note that the term in the sum \eqref{eq:momentasymptotic} corresponding to  $\pi=(n)$ contributes exactly to the right-hand side of \eqref{eq:limite_mu}.
\end{rem}

At this stage, deriving the PDEs \eqref{eq:derivative0gbis000000} and \ref{PDE} associated with the moments  $\mathfrak m_n(t)$ requires technical computations involving the Faà di Bruno formula, combinatorial identities, generating functions, and Fourier analysis. We refer to Section~\ref{sec:proofhydro} for details.

\subsubsection{Variance of the $n$th random moment}

To prove convergence in probability of the occupation measure $\upsilon_{L,N}(t,\cdot)$, it suffices to show that the variance of its $n$th moment vanishes. This variance is given by
\begin{equation}\label{eq:variancemoment}
\mathbb V\left({\mathcal M}_n^{L,N}(t)\right)=\mathbb E_{\vartheta}\left[\frac{(p_n p_{-n})(\Theta(L^2t))}{N^2}\right] -
\left|\mu^{L,N}_n(t)\right|^2,
\end{equation}
where $p_{-n}(z_1,\cdots,z_N):=p_n(z_1^{-1},\cdots,z_N^{-1})$. Hence, it remains to prove
\begin{equation}\label{eq:varianceproduct}
\lim_{L,N\to\infty}\mathbb E_{\vartheta}\left[\frac{(p_n p_{-n})(\Theta(\lfloor L^2 t\rfloor))}{N^2}\right]
= |\mathfrak m_n(t)|^2 = \mathfrak m_n(t)\,\overline{\mathfrak m_n(t)}.
\end{equation}

The difficulty here is that $p_n(z_1,\ldots,z_N)\,p_{-n}(z_1,\ldots,z_N)$ is not a symmetric polynomial but a symmetric Laurent polynomial. This prevents us from directly using the Schur and power-sum bases of the ring of symmetric polynomials.

\begin{rem}
	Although $p_n p_{-n} = p_n p_{L-n}$ on the set of $L$th roots of unity, exploiting this identity to return to a polynomial framework would be impractical, as the degree $L$ diverges.  The product in~\eqref{eq:varianceproduct} suggests decomposing $p_n p_{-n}$ into products of hook Schur functions and their conjugates.
\end{rem}

\subsubsection{Double hook-Schur decompositions}

While the expectation in~\eqref{eq:sanspartieentiere} involves hook partitions $\{n|k\}$, the variance analysis naturally leads to double-hook partitions (see also Figure~\ref{fig:doublehook}), defined by
\begin{equation}\label{eq:doublehookk}
\{n|k,l\}=\left(L-n-(N-1-l),\,n-k+1,\,2^{k},\,1^{N-2-l-k}\right).
\end{equation}

\begin{figure}[!h]
	\centering
	\includegraphics[width=10cm]{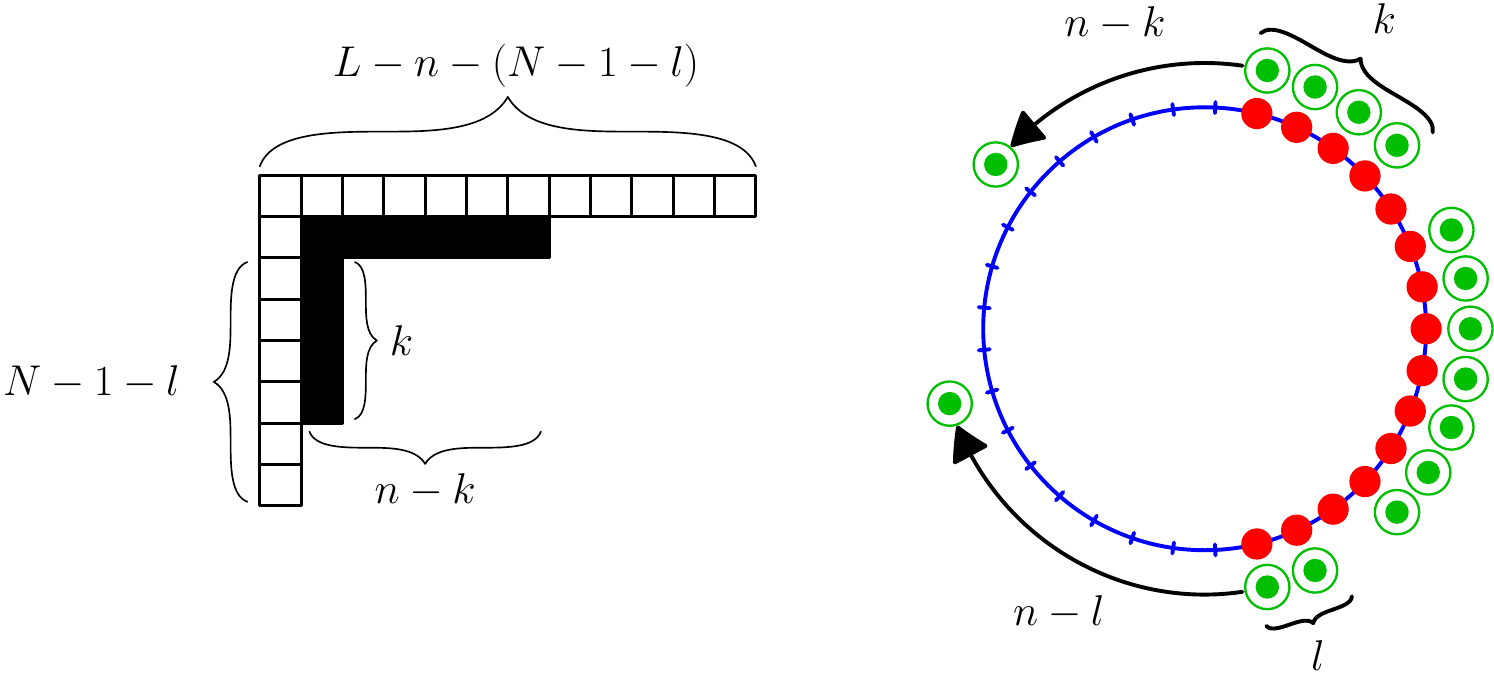}
    \captionsetup{width=11cm}
\caption{\small  Correspondence between the double-hook partition $\{n|k,l\}$ and the associated configuration (circled dot)}

	\label{fig:doublehook}
\end{figure}

To obtain the variance asymptotics, we use Lemma~\ref{lem:variancehook}, which reduces the computation to products of hook Schur polynomials. We briefly recall the definition of skew Schur polynomials. Let $\mu\subset\lambda$ be partitions, i.e.\ $\mu_i\leq\lambda_i$ for all $i$. The skew diagram $\lambda\setminus\mu$ is obtained by removing the boxes of $\mu$ from $\lambda$. The associated skew Schur polynomial $s_{\lambda\setminus\mu}$ write as
\begin{equation}\label{eq:littlewood}
s_{\lambda\setminus \mu} = \sum_{\nu} c^{\lambda}_{\mu\nu} \, s_\nu,
\end{equation}
where $c^{\lambda}_{\mu\nu}$ are the so-called Littlewood--Richardson coefficients. Combinatorially, $c^{\lambda}_{\mu\nu}$ counts semi-standard Young tableaux of shape $\lambda\setminus\mu$ and content $\nu$ whose row-reading word is a lattice word (see also Figure~\ref{fig:littlewood}). By convention, $s_{\lambda\setminus\mu}=0$ if $\mu\not\subset\lambda$. We denote by $\lambda'$ the conjugate partition  obtained by transposing the Young diagram of $\lambda$. Note that $\{n|k\}'=\{n|n-k-1\}$.

The next lemma provides the algebraic identities needed to decompose $p_n p_{-n}$ into hook Schur polynomials, allowing us  to leverage Lemma~\ref{lem:characteridentities}. 

\begin{lem}\label{lem:variancehook}
	The following identities hold for all $z = (z_1, \cdots, z_N) \in \mathbb{U}_L^N$:
	\begin{enumerate}
		\item For all $n \geq 1$ and $0 \leq k \leq n - 1$,
		\begin{equation}\label{eq:conjhook}
		\overline{s_{\{n|k\}}(z)} = (-1)^{2\bgamma} \, s_{\{L - n \,|\, N - k - 1\}}(z).
		\end{equation}
		
		\item For all $n \geq 1$, one has the decomposition:
		\begin{equation}\label{eq:powersumdoublehook}
		p_n(z) \, p_{-n}(z) = (-1)^{2\bgamma} \sum_{k=0}^{n-1} \sum_{l=0}^{n-1} (-1)^{k + l} \, s_{\{n|k,l\}}(z) + n.
		\end{equation}
		
		\item For all $n \geq 1$ and $0 \leq k, l \leq n - 1$,  one has the expansion:
		\begin{equation}\label{eq:rationalschurconj}
		s_{\{n|k,l\}}(z) = (-1)^{N - 1} \sum_{\tau} (-1)^{|\tau|} \, s_{\{n|k\} \setminus \tau}(z) \, \overline{s_{\{n|l\} \setminus \tau^\prime}(z)},
		\end{equation}
		where the sum runs over all partitions $\tau$ such that $\tau \subset \{n|k\}$ and $\tau^\prime \subset \{n|l\}$. In particular, $\tau$ is either empty or equal to $\{m|i\}$ for some $1 \leq m \leq n$ and $0 \leq i \leq m - 1$.
		
		\item For all $1 \leq m \leq n$ and $0 \leq i \leq m - 1$, one has
		\begin{equation}\label{eq:skewhooklittle1}
		s_{\{n|k\} \setminus \{m|i\}}(z) = s_{\{n - m \,|\, k - i\}}(z) + s_{\{n - m \,|\, k - i - 1\}}(z),
		\end{equation}
		and
		\begin{equation}\label{eq:skewhooklittle2}
		s_{\{n|l\} \setminus \{m|i\}^\prime}(z) = s_{\{n - m \,|\, l - (m - i - 1)\}}(z) + s_{\{n - m \,|\, l - (m - i - 1) - 1\}}(z),
		\end{equation}
		with  for all $r \geq 1$ and $j \in \mathbb{Z}$, $s_{\{r|j\}}(z) \equiv 0$, if $j < 0$ or $j > r - 1$, and $s_{\{0|j\}}(z) \equiv \delta_{0j}$.
	\end{enumerate}
\end{lem}

Using~\eqref{eq:powersumdoublehook} and~\eqref{eq:varianceproduct}, it suffices to show that
\begin{equation}\label{eq:goealvar}
\lim_{L,N\to\infty}
\frac{(-1)^{2\bgamma}}{N^2} \sum_{k=0}^{n-1} \sum_{l=0}^{n-1} (-1)^{k+l} \, \br_{\{n|k,l\}}^{\lfloor L^2 t \rfloor} \, s_{\{n|k,l\}}
= \mathfrak m_n(t)\,\overline{\mathfrak m_n(t)},
\end{equation}
where $\br_{\{n|k,l\}}$ denotes, similarly to \eqref{eq:sanspartieentiere}, the MESSEP eigenvalue associated with $s_{\{n|k,l\}}$. Applying Lemma~\ref{lem:variancehook}, in particular~\eqref{eq:rationalschurconj}, and using $(-1)^{N-1+2\bgamma}=1$, we obtain

\begin{multline}
\frac{(-1)^{2\bgamma}}{N^2} \sum_{k=0}^{n-1} \sum_{l=0}^{n-1} (-1)^{k+l} \, \br_{\{n|k,l\}}^{\lfloor L^2 t \rfloor} \, s_{\{n|k,l\}}(\vartheta)
= \frac{1}{N^2} \sum_{k=0}^{n-1} \sum_{l=0}^{n-1} (-1)^{k+l} \, \br_{\{n|k,l\}}^{\lfloor L^2 t \rfloor} \, 
s_{\{n|k\}}(\vartheta) \, \overline{s_{\{n|l\}}(\vartheta)} \\
+ \frac{1}{N^2} \sum_{m=1}^{n}\sum_{i=0}^{m-1}  \sum_{k=0}^{n-1} \sum_{l=0}^{n-1} (-1)^m (-1)^{k+l} \, \br_{\{n|k,l\}}^{\lfloor L^2 t \rfloor} \, 
s_{\{n|k\} \setminus \{m|i\}}(\vartheta) \, \overline{s_{\{n|l\} \setminus \{m|i\}^\prime}(\vartheta)}.
\end{multline}

We also distinguish the cases $\tau=\emptyset$ and $\tau\neq\emptyset$ in \eqref{eq:rationalschurconj} and use the remark below this identity to get the latter equation. The proof of  \eqref{eq:goealvar} is therefore reduced to the following lemma.

\begin{lem}\label{lem:finallemvar}
	The following  limits hold:
	\begin{equation}\label{eq:techlem1}
	\lim_{L,N\to\infty}\frac{1}{N^2} \sum_{k=0}^{n-1} \sum_{l=0}^{n-1} (-1)^{k+l} \, \br_{\{n|k,l\}}^{\lfloor L^2 t \rfloor} \, 
	s_{\{n|k\}}(\vartheta)\, \overline{s_{\{n|l\}}(\vartheta)} = \mathfrak m_n(t)\, \overline{\mathfrak m_n(t)},
	\end{equation}
	and
	\begin{equation}\label{eq:techlem2}
	\lim_{L,N\to\infty} \frac{1}{N^2} \sum_{m=1}^{n} \sum_{i=0}^{m-1} \sum_{k=0}^{n-1} \sum_{l=0}^{n-1} (-1)^{m+k+l} \, \br_{\{n|k,l\}}^{\lfloor L^2 t \rfloor} \, 
	s_{\{n|k\} \setminus \{m|i\}}(\vartheta)\, \overline{s_{\{n|l\} \setminus \{m|i\}^\prime}(\vartheta)} = 0.
	\end{equation}
\end{lem}

To prove Lemma~\ref{lem:finallemvar}, we again use Lemma~\ref{lem:characteridentities} together with~\eqref{eq:skewhooklittle1} and~\eqref{eq:skewhooklittle2}. We first extend Lemma~\ref{lem:eigenhookdevasym}. As in~\eqref{eq:hookeigenvalueprecise}, a direct computation yields

\begin{equation}\label{eq:doublehookeigenvalueprecise}
\br_{\{n|k,l\}}=1-\frac{2\sin\left(\frac{\pi}{L}\right)\sin\left(\frac{\pi n}{L}\right)\left(\sin\left(\frac{\pi N}{L}+\frac{\pi(n-2k-1)}{L}\right)+\sin\left(\frac{\pi N}{L}+\frac{\pi(n-2l-1)}{L}\right)\right)}{\sin\left(\frac{\pi N}{L}\right)}.
\end{equation}

\begin{lem}\label{lem:eigendoublehookdevasym}
	Let $T \geq 0$ and $n, d \geq 1$ be fixed. For all $0 \leq t \leq T$ and $0 \leq k, l \leq n-1$, one can write
	\begin{equation}\label{eq:expansiorhoschurdoublehook}
	\br_{\{n|k,l\}}^{\lfloor L^2 t\rfloor}
	= \sum_{j_1 + j_2 \leq d-1} \frac{a^{L,N}_{n,j_1,j_2}(t)}{j_1! j_2!} \left( \frac{\pi(n - 2k - 1)}{L} \right)^{j_1} \left( \frac{\pi(n - 2l - 1)}{L} \right)^{j_2}
	+ \mathcal{O}\left( \frac{1}{L^d} \right),
	\end{equation}
	as $L, N \to \infty$, where the constant in the $\mathcal{O}$-term depends only on $T,n,d$. Moreover,
	\begin{equation}\label{eq:deriventhdouble}
	\lim_{L,N\to\infty} a_{n,j_1,j_2}^{L,N}(t)
	= \left.\frac{\partial^{j_1} h^n_t(x)}{\partial x^{j_1}}
	\right|_{x=0} \,\cdot\, \left.\frac{\partial^{j_2} h^n_t(x)}{\partial x^{j_2}}\right|_{x=0},
	\end{equation}
	where $h_t(x)$ is the function defined in \eqref{eq:expansiorhoschurhook}.
\end{lem}

\section{Proofs of the results in Section \ref{sec:1}}

\label{sec:proofsec1}

\setcounter{equation}{0}

\begin{proof}[Proof of Proposition \ref{noncollinding}] Observe that if $T>n$ then  $\sigma\cdot Y(k)\in {\mathcal W_{L,N}}$ for every $0\leq k\leq n$ and thus $(\sigma\cdot Y(k))_{0\leq k\leq n}$ is a finite path in ${\mathcal W_{L,N}}$. The probability $\mathbb P_{\theta}(T>n)$ depends only on the projection  of $\theta\in {\mathcal W_{L,N}}$ on ${\mathcal C_{L,N}}$. We set  $\mathbb P_{\xi}(T>n):=\mathbb P_{\theta}(T>n)$ for any $\xi\in {\mathcal C_{L,N}}$ and any $\theta$ such  $p_2(\theta):=\xi$.
	
	Then, we obtain easily from the Markov property that for every $\xi\in {\mathcal C_{L,N}}$ and $n\geq 0$,
	\begin{equation}\label{eq:powerA}
	\mathbb P_{\xi}(T>n)=(2N)^{-n} A^n\mathds 1(\xi),
	\end{equation}
	where $A$ is the adjacency matrix of ${\mathcal C_{L,N}}$ and $\mathds 1$ the constant function equal to one. Note that $A$ is 2-periodic and thus  ${\mathcal C_{L,N}}$ can be written as $W_1\sqcup W_2$ where $W_1$ and $W_2$ are disjoint stable state space of $A^2$. Therefore,  it comes from the spectral decomposition of $A$  that  
	\begin{equation}
	\mathbb P_{\xi}(T>2m)\underset{m\to\infty}{\sim} (2N)^{-2m}\rho^{2m}\left[\sum_{i\in\{1,2\}}\left(\frac{\sum_{\eta\in W_i}\psi(\eta)}{\sum_{\eta\in W_i}\psi^2(\eta)}\right)\mathds 1_{W_i}( \xi)\right]\psi(\xi).
	\end{equation}
	A similar asymptotic holds replacing $2m$ by $2m+1$. 
	
	We get that, for any path $(e(0),\cdots,e(n))$ in ${\mathcal W_{L,N}}$, with $e(0)=\sigma\cdot y$, and $m\geq n$,
	\begin{multline}
	\mathbb P_{\sigma\cdot y}(\sigma\cdot Y(1)=e(1),\cdots,\sigma\cdot Y(n)=e(n)|T>m)\\
	=\frac{1}{(2N)^n}\frac{\mathbb P_{\sigma\cdot e(n)}(T>m-n)}{\mathbb P_{\sigma\cdot y}(T>m)}\underset{m\to\infty}{\sim} \frac{1}{\rho^n}\frac{\psi(\sigma\cdot e(n))}{\psi(\sigma\cdot  y)},
	\end{multline}
	which completes the proof.
\end{proof}	

\label{proof:spectre}

\begin{proof}[Proof of Proposition \ref{MESSEPspectre}]
First, we shall prove that
\begin{equation}\label{ortho}
\sum_{\eta\in \{0,\cdots,L-1\}^N} \psi_{\xi}(\eta) \overline{\psi_{\xi^\prime}(\eta)} = N!\,\delta_{\xi=\xi^\prime}.
\end{equation} 
To this end, set $\omega=e^{2i\pi/L}$ and observe that $\sum_{\eta=0}^{L-1} \omega^{\eta l} = L\,\delta_{l\in L\mathbb{Z}}$ for all $l \in\mathbb Z$. Hence, the left-hand side of (\ref{ortho}) is equal to
	\begin{equation}\label{ortho2}
	\frac{1}{L^N}\sum_{\eta_1=0}^{L-1}\cdots\sum_{\eta_N=0}^{L-1}\sum_{\sigma,\sigma^\prime\in\mathfrak S_N}\epsilon(\sigma)\epsilon(\sigma^\prime)\omega^{\sum_{i=1}^N\eta_i\left(\xi_{\sigma(i)}-{\xi^\prime_{\sigma(i)}}\right)} \\
	=\sum_{\sigma,\sigma^\prime\in\mathfrak S_N}\epsilon(\sigma)\epsilon(\sigma^\prime)\delta_{\sigma\cdot\xi =\sigma^\prime\cdot \xi^\prime }.
	\end{equation} 
Since  $\xi$ and $\xi^\prime$ are both strictly ordered, the right-hand side of \eqref{ortho} is equal to $N!\, \delta_{\xi=\xi^\prime}$. We get that   the family $\psi_\xi$, for $\xi\in \mathcal C_{L,N}$, is orthonormal in $\ell^2(\mathcal C_{L,N})$ since
\begin{equation}\label{antisymsum}
\sum_{\eta\in \{0,\cdots,L-1\}^N  } \psi_\xi(\eta)\overline{\psi_{\xi^\prime}(\eta)}
= \sum_{\eta\in \mathcal C_{L,N}} \sum_{\sigma\in\mathfrak{S}_N} \psi_\xi(\sigma\cdot\eta)\overline{\psi_{\xi^\prime}(\sigma\cdot\eta)}
= N! \sum_{\eta\in \mathcal C_{L,N}} \psi_\xi(\eta)\overline{\psi_{\xi^\prime}(\eta)}.
\end{equation}
Recall that $\psi_\xi(\eta)$ vanishes as soon as  $\xi_i=\xi_j$ for $i\neq j$.  Noting that the cardinal of $\mathcal C_{L,N}$ is $\binom{L}{N}$, we obtain that this family is, in addition, a basis of $\ell^2(\mathcal C_{L,N})$.

	It remains to prove they are eigenfunctions associated with the appropriate eigenvalues. Let us introduce for generic  $N$-tuples  $Z=(Z_1,\cdots,Z_N)$ and $Y=(Y_1,\cdots,Y_N)$  the formal monomial 
	$Q_Z(Y):=Z_1^{Y_1}\cdots Z_N^{Y_N}$ and observe that 
	\begin{equation}\label{bethe}
	\sum_{\substack{1\leq i\leq N\\\epsilon\in\{\pm 1\}}}Q_Z(Y_1,\cdots,Y_i,Y_i+\epsilon,Y_{i+1},\cdots,Y_N)=\left(\sum_{i=1}^N\left(Z_i+\frac{1}{Z_i}\right)\right) Q_Z(Y).
	\end{equation}
	This relation  still holds replacing $Q_Z(Y)$ by $Q_Z^{s}(Y)={\rm det}(Z_i^{Y_j})_{1\leq i,j\leq N}$.
	Note that $Q_Z^s(Y)=0$ as soon as at least two of the $Y_i$  are equal. 
	
	Those eigenfunctions must extend to the discrete $L$-periodic Weyl chamber ${\mathcal W_{\mathbb{Z},L,N}}$ and thus  they must satisfy
\begin{equation}\label{eq:peridodic}
Q^s_Z(Y_1,\cdots,Y_N) = Q_Z^s(Y_2,\cdots,Y_N,Y_1+L).
\end{equation}
This relation holds  when we substitute  
\begin{equation}\label{eq:root}
Z_i\equiv e^\frac{2i\pi (\xi_i+\bgamma)}{L}	
\quad\mbox{and}\quad 
Y_j\equiv\eta_j.
\end{equation}

Even though this is straightforward, we provide the computations as they clarify why we require the additive parameter  $\bgamma$.  Let $\tau_L$ be define on $N$-tuples by 
	\begin{equation}\label{tauchap}
	\tau_L\cdot \eta=(\eta_2,\cdots, \eta_N,\eta_1+L)=\tau\cdot \eta +(0,\cdots,0,L),
	\end{equation}
	where  $\tau$ is the cycle $(1\, 2\,\cdots \,N)$. Since  $\varepsilon(\tau)=(-1)^{N+1}$ and $\omega^{(k+\bgamma)(l+L)}=(-1)^{2\bgamma}\omega^{(k+\bgamma)l}$ for any integers $k,l$ and $\bgamma\in\{0,1/2\}$, we get 
	\begin{equation}
	{\rm det}\left(e^{\frac{2i\pi (\xi_k+\bgamma) \tau_L\cdot \eta_j}{L}}\right)_{1\leq k,j\leq N} =(-1)^{N+1}(-1)^{2\bgamma}  
	\;{\rm det}\left(e^{\frac{2i\pi (\xi_k+\bgamma) \eta_j}{L}}\right)_{1\leq k,j\leq N}.
	\end{equation}
	This is the reason  why we set $\bgamma=0$ when $N$ is odd and $\bgamma=1/2$ when $N$ is even. 
	
Thereafter,  using \eqref{bethe}, one can check that $\psi_\xi$ is an eigenfunction with eigenvalue $\rho_\xi$.

\begin{rem}\label{rem:roots}
	The eigenfunction $\psi_\xi$ need not be real (even up to a multiplicative constant) when the roots $Z_i$, $i\in\{1,\cdots,N\}$ defined in~\eqref{eq:root} are not invariant under conjugation, equivalently when the eigenspace associated with $\rho_\xi$ is not one-dimensional. 
\end{rem}

Finally, one checks that $\rho=\rho_c$ is the spectral radius, given by the right-hand side of~\eqref{PF}. The associated positive eigenvector $\psi$ coincides with $\psi_c$ up to a multiplicative constant, namely a classical Vandermonde determinant. We provide the explicit computation.

Set $\delta=(N-1,\cdots,0)$ and let $p$ be such that  $N=2p+1$ or $N=2p$. Note that
	\begin{equation}\label{eq:tauchapappli}
	c=\tau^{p}_L \sigma\cdot\left(\delta-\left(p,\cdots,p\right)\right)=\underbrace{\,\tau_L\circ \cdots \circ \tau_L\,}_{\text{$p$ times}} \cdot \left(\sigma\cdot\left(\delta-\left(p,\cdots,p\right)\right)\right),	
	\end{equation}
 where  $\sigma$ is the permutation mapping $i$ to $N-i+1$ for every $1\leq i\leq N$.  We deduce that 
	\begin{equation}
	\psi_{c}(\eta)=\varepsilon(\tau)^p\varepsilon(\sigma)  e^{\frac{2i\pi (\bgamma-p)\sum_{k=1}^N \eta_k}{L}}   \frac{1}{L^{{N}/{2}}}\;{\rm det} \left(e^{\frac{2i\pi {(N-k)}\eta_j}{L}}\right)_{1\leq k,j\leq N}.
	\end{equation}
Then, the explicit computation of $\psi$ is straightforward, which completes the proof.
\end{proof}

\label{sec:gap}

\begin{proof}[Proof of Proposition \ref{gapestimate}] In the sequel, all constants in the $\mathcal O$-terms are uniform in $L,N,\alpha$. 
	
	First, observe that the second largest eigenvalue is given by 
	\begin{equation}
	\lambda_{L,N} = \frac{\rho_{\tilde c}}{\rho_c}, \quad \text{with}\quad \tilde c_N = c_N + 1  \quad \text{and}\quad  \forall i \in \{1,\cdots,N-1\},\;\; \tilde c_i = c_i.
	\end{equation}
Note that $c_N+\bgamma=(N-1)/2$. Besides, one has
	\begin{equation}\label{asymptoticexp0}
	\cos\left(\frac{2\pi (x+1)}{L}\right)=\cos\left(\frac{2\pi x}{L}\right)-\sin\left(\frac{2\pi x}{L}\right)\frac{2\pi}{L}-\frac{1}{2}\cos\left(\frac{2\pi x}{L}\right)\left(\frac{2\pi}{L}\right)^2
	+\mathcal O \left(\frac{|x|+1}{L^4}\right).
	\end{equation}
	
The only difference between $\rho_{c}$ and $\rho_{\tilde c}$ (see  \eqref{Spectre}) is the last term, respectively given by $2\cos\left({2\pi x_N}/{L}\right)$ and $2\cos\left({2\pi (x_N+1)}/{L}\right)$, with $x_N=(N-1)/2$. We obtain  that
	\begin{equation}\label{gap}
	1-\lambda_{L,N}=\frac{2}{\rho}
	\left[\sin\left(\frac{2\pi x_N}{L}\right)\frac{2\pi}{L}+\frac{1}{2}\cos\left(\frac{2\pi x_N}{L}\right)\left(\frac{2\pi}{L}\right)^2 +\mathcal O\left(\frac{x_N+1}{L^4}\right)\right].
	\end{equation}
In the following, we denote by $B=[\cdots]$ the bracket in \eqref{gap}.

\medskip

	\noindent
	\textit{Low-density limit.} First assume  the limit of $N/L$ is $0$. Then,  it follows from the taylor expansion of $\cos$ at $0$ and standard computations that  
	\begin{equation}\label{estimatepf}
	\rho=2N\left(1-\mathcal O\left(\frac{N^2}{L^2}\right)\right)
	\quad
	\text{and}
	\quad 
	B=	N\left(\frac{2\pi^2}{L^2}+\mathcal O\left(\frac{N^2}{L^4}\right)\right).
	\end{equation}
As a product, we get the desired upper bound \eqref{estimategapmessep}.

\medskip

\noindent
\textit{Hydrodynamic limit.}
Finally, when $N/L$ tends to $\alpha \in (0,1)$, standard estimates  yield  
\begin{equation}\label{eq:riemansum}
\rho =  2L\left(\frac{\sin\left(\frac{\pi N}{L}\right)}{\pi} + \mathcal{O}\left(\frac{1}{L^2}\right)\right)\quad \text{and}\quad B = \sin\left(\frac{\pi N}{L}\right)\frac{2\pi}{L} + \mathcal{O}\left(\frac{1}{L^2}\right).
\end{equation} 
and the upper bound \eqref{estimategapmessep2} follows easily.
\end{proof}

\label{sec:schur}

\begin{proof}[Proof of Proposition \ref{Schur}] We now make precise the correspondence between integer partitions and configurations illustrated in Figure~\ref{partition}. For $\xi\in\mathcal C_{L,N}$, let $F_\xi\subset\mathcal W_{\mathbb Z,L,N}$ denote its discrete fiber under the projection $p_2\circ p_1:\mathcal W_{\mathbb Z,L,N}\to\mathcal C_{L,N}$ (see~\eqref{diagram}). As before, we write $N=2p+1$ or $N=2p$, consistently with the definition of the compact configuration $c$ in~\eqref{compactconf}.

\begin{defin}[Minimal integer partition]  \label{def:minimal}
	For $\xi\in {\mathcal C_{L,N}}$, let $x=(x_1,\cdots, x_N) \in F_\xi$ be the unique element of $F_\xi$ for which $x_1 + p$ is non-negative and minimal. 	Then, the minimal integer partition $\lambda$ associated with the configuration $\xi$ is defined by
	\begin{equation}\label{minimalcompactconf}
	\lambda =\left\{\begin{array}{ll}
	(x_N - p,x_{N-1} - p +1, \cdots, x_2 + p - 1, x_1 + p), & \text{if $N=2p+1$},\\[10pt]
	(x_N - p + 1,x_{N-1} - p +1, \cdots, x_2 + p - 1, x_1 + p), & \text{if $N=2p$}.
	\end{array}\right.
	\end{equation}
\end{defin}

\begin{rem}\label{rem:minimalpartition}
	This partition describes the positive  displacement of the particles from the compact configuration $c$ (initially at $-p+N-1,\ldots,-p$ modulo $L$, see Remark~\ref{rem:compactrep}). The resulting configuration $\xi$ minimizes the displacement of the lowest particle (initially at $-p$). Consequently, the highest particle (initially at $-p+N-1$) can move at most $L-N$, so that $x_N-p\le L-N$.
\end{rem}

To go further, let $x\in F_\xi\subset \mathcal W_{\mathbb Z,L,N}$ be as in Definition \ref{def:minimal}.  One can see that  there exists a unique integer $k$ such that 
\begin{equation}\label{integerk}
p_1 ((x_N,\cdots, x_1))=(\dot \xi_{k},\cdots,\dot \xi_1,\dot \xi_N,\cdots,\dot \xi_{k+1})=\tau^{-k}\sigma\cdot \dot \xi,
\end{equation}
where $\sigma$ denotes the permuation that maps  $i$ to $N-i+1$ for any $1\leq i\leq N$ and $\tau$ is the cycle $(1\,2\, \cdots\, N)$. Let $\delta=(N-1,\cdots, 0)$ be as \eqref{eq:schurdet} and let $\lambda$ be the minimal integer partition corresponding to $\xi$.  Note that $(x_N+p,\cdots,x_1+p)=(\lambda_1+\delta_1,\cdots,\lambda_N+\delta_N)$.

It follows from \eqref{integerk}   that   
\begin{equation}\label{schureigenbasis}
{\rm det}\left(e^{\frac{2i\pi (\xi_m+\bgamma)\eta_j}{L}}\right)_{1\leq m,j\leq N}=\varepsilon(\sigma)\varepsilon(\tau)^{k}e^{\frac{2i\pi (\bgamma-p) \sum_{i=1}^N\eta_i}{L}}{\rm det}\left(e^{\frac{2i\pi (\lambda_m+\delta_m)\eta_j}{L}}\right)_{1\leq m,j\leq N}.
\end{equation}

Besides,  when $\xi=c$, one has $k=p+1$ or $k=p$ according with $N=2p+1$ or $N=2p$ respectively. Hence, noting that $\epsilon(\tau)=1$ if $N$ is odd, one can write by using  \eqref{eq:schurdef} that 
\begin{equation}\label{epsilon1}
\frac{\psi_\xi(\eta)}{\psi_{c}(\eta)}= \varepsilon(\tau)^{p-k} \, s_{\lambda}\left(e^{{2i\pi \eta_1}/{L}},\cdots,e^{{2i\pi \eta_N}/{L}}\right).
\end{equation}

Reciprocally, consider an integer partition $L-N\geq \lambda_1\geq \cdots\geq \lambda_N\geq 0$. Note that  
\begin{equation}
\forall 1\leq i<j\leq N,\quad L-1\geq \lambda_i+\delta_i>\lambda_j+\delta_j\geq 0.
\end{equation}
Then, by setting $\xi_i = \lambda_i + \delta_i$ for every $1 \leq i \leq N$, one can easily check that we obtain a configuration $\xi \in {\mathcal C_{L,N}}$ for which \eqref{epsilon1} holds. This ends the proof.\end{proof}

\section{Proofs of the results in Section \ref{sec:low}}
\label{sec:sec3proof}

We begin by rigorously defining the stochastic processes $(\bX(t))_{t\geq 0}$, $(\bTheta(t))_{t\geq 0}$, and $(\bXi(t))_{t\geq 0}$.

\subsection{The Unitary Dyson Brownian Motions (UDBMs)}
\label{sec:dyson}

Let $(B(t))_{t\geq 0} := (B_{1}(t), \cdots, B_{N}(t))_{t\geq 0}$ be a standard $N$-dimensional Brownian motion. We consider the interacting particle system in $\overline{\boldsymbol{\mathcal{W}}}_{\mathbb{R},2\pi,N}$ defined by
\begin{equation}\label{Dyson1}
d\boldsymbol{X}_i(t) = {\frac{2\pi}{\sqrt N}}\,dB_i(t) + \frac{2\pi^2}{N} \sum_{\substack{1 \leq j \leq N \\ j \neq i}} \cot\left( \frac{\boldsymbol{X}_i(t) - \boldsymbol{X}_j(t)}{2} \right) dt, \quad \boldsymbol{X}_i(0) = \boldsymbol{x}_i,
\end{equation}
for $1 \leq i \leq N$, where $\overline{\boldsymbol{\mathcal{W}}}_{\mathbb{R},2\pi,N}$ is defined in \eqref{eq:weylR}. Equivalently,
\begin{equation}\label{Dyson2}
d\boldsymbol{X}(t) = {\frac{2\pi}{\sqrt N}}\,dB(t) + \frac{(2\pi)^2}{N} \frac{\nabla \Psi(\boldsymbol{X}_t)}{\Psi(\boldsymbol{X}_t)} dt, \quad \boldsymbol{X}(0) = \boldsymbol{x},
\end{equation}
where $\Psi$ is given in \eqref{eq:Psi}. Let us set 
\begin{equation}
{\boldsymbol{\mathcal{W}}}_{\mathbb R,2\pi,N} := \left\{\bx \in \mathbb R^N : \bx_1 < \cdots < \bx_N < \bx_1 + 2\pi \right\},
\end{equation}
so that ${\boldsymbol{\mathcal{W}}}_{\mathbb R,2\pi,N}$ is an open set whose closure is $\overline{\boldsymbol {\mathcal W}}_{\mathbb R,2\pi,N}$ defined in \eqref{eq:weylR}. We say that a stochastic process  in $\overline{\boldsymbol {\mathcal W}}_{\mathbb R,2\pi,N}$ {collides} if, with probability one, it reaches the boundary $\partial {\boldsymbol{\mathcal{W}}}_{\mathbb R,N}$ at some time.

The following result can be found in \cite{Cepa95,CepaLepingle97,CepaLepingle01}.

\begin{prop}\label{cepa} 
	There exists a unique strong solution $\boldsymbol X$ to (\ref{Dyson2}) for every $\boldsymbol x\in \overline{\boldsymbol{\mathcal W}}_{\mathbb R,2\pi,N}$. Moreover, 
	\begin{equation}\label{estimation}
	\forall T\geq 0,\quad \sum_{1\leq i<j\leq N}\mathbb E_{\boldsymbol x}\left[\int_0^T \frac{{\rm d}t}{|\boldsymbol X_j(t)-\boldsymbol X_i(t)|}\right]<\infty.
	\end{equation}
	Furthermore, the process is continuous, strong Markov, never collides if $\boldsymbol x\notin\partial {{\boldsymbol{\mathcal W}}_{\mathbb R,N}}$, and leaves the boundary instantaneously otherwise.
\end{prop}

Mimicking Section \ref{sec:extend}, we introduce two continuous processes $(\boldsymbol \Theta(t))_{t\geq 0}$ and $(\boldsymbol \Xi(t))_{t\geq 0}$, analogues of $(\Theta(n))_{n\geq 0}$ and $(\Xi(n))_{n\geq 0}$.  The continuous setting requires additional topological care.

To this end, similarly to \eqref{eq:weyldiscretes}, set
\begin{multline}\label{eq:weyl1}
\bCW_{2\pi,N}^{\scriptscriptstyle \#} := \left\{ \btheta \in \mathbb U^N : \exists\, \bx \in \bCW_{\mathbb R,2\pi,N},\ \btheta_1 = e^{i \bx_1},\cdots,\btheta_N = e^{i \bx_N} \right\} \\
\text{and} \quad \bCW_{2\pi,N}  := \left\{ \btheta \in \mathbb R_{2\pi}^N : \exists\, \bx \in \bCW_{\mathbb R,2\pi,N},\ \btheta_1 = \overline{\bx}_1, \cdots, \theta_N = \overline{\bx}_N \right\},
\end{multline}
where $\overline{x}$ denotes the class of $x \in \mathbb R$ modulo $2\pi$.  Also,  consider
\begin{multline}
\bCC_{2\pi,N} := \left\{ \bxi\in [0,2\pi)^N : 0\leq \bxi_1<\cdots<\bxi_N<2\pi\right\} \label{eq:weyl2}\\
\text{and}\quad \bCC_{2\pi,N}^{\scriptscriptstyle \#} = \left\{ \btheta \in \mathbb U^N : \exists\, \bx\in \bCC_{2\pi,N},\; \btheta_1 = e^{i \bx_1}, \cdots, \btheta_N = e^{i \bx_N} \right\}. 
\end{multline}

As in (\ref{diagram}), one has canonical projections
\begin{equation}\label{eq:proj}
\bCW_{\mathbb R,2\pi,N} \overset{p_1}{\longtwoheadrightarrow} \bCW_{2\pi,N}\simeq \bCW_{2\pi,N}^{\scriptscriptstyle \#} \overset{p_2}{\longtwoheadrightarrow} \bCC_{2\pi,N}\simeq \bCC_{2\pi,N}^{\scriptscriptstyle \#},
\end{equation}
inherited from the standard projections $\mathbb R \overset{}{\longtwoheadrightarrow} \mathbb U \simeq \mathbb R/{2\pi\mathbb Z}$ and $\mathbb U \simeq \mathbb R/{2\pi\mathbb Z} \overset{}{\longtwoheadrightarrow} [0,2\pi)$. 

The spaces $\bCW_{2\pi,N}\simeq \bCW_{2\pi,N}^{\scriptscriptstyle \#}$ and $\bCC_{2\pi,N}\simeq \bCC_{2\pi,N}^{\scriptscriptstyle \#}$ are endowed with the quotient topologies induced by the above surjections and the standard topology of $\bCW_{\mathbb R,2\pi,N}\subset \mathbb R^N$. 

These topologies are metrizable. Indeed, for $\theta,\phi\in\mathbb R/{2\pi\mathbb Z}$ with representatives $x,y\in\mathbb R$ such that $\theta=\overline{x}$ and $\phi=\overline{y}$, set $d(\theta,\phi)=|e^{ix}-e^{iy}|$. When viewed in $\mathbb U$, set $d(\theta,\phi)=|\theta-\phi|$. A compatible metric on ${\bCW}_{2\pi,N}$ or ${\bCC}_{2\pi,N}$, as well as on their $\#$-versions, is given by
\begin{equation}
d_H(\xi, \eta) := \min_{\sigma \in \mathfrak{S}_N} \max_{1 \leq i \leq N} d(\xi_{\sigma(i)}, \eta_i),
\end{equation}
for $\xi,\eta$ in the corresponding space. Their completions are denoted by adding an overline. Note that, unlike $\overline{\bCW}_{\mathbb R,2\pi,N}$, the spaces $\overline{\bCW}_{2\pi,N}$ and $\overline{\bCC}_{2\pi,N}$ are compact.

Furthermore, the UDBM \eqref{Dyson2} admits continuous projections onto $\overline{\bCW}_{2\pi,N}$ (resp. $\overline{\bCW}_{2\pi,N}^{\scriptscriptstyle \#}$) and $\overline{\bCC}_{2\pi,N}$ (resp. $\overline{\bCC}_{2\pi,N}^{\scriptscriptstyle \#}$), denoted by $(\boldsymbol{\Theta}(t))_{t \geq 0}$ and $(\boldsymbol{\Xi}(t))_{t \geq 0}$.

\begin{defin}\label{def:merwcontinuous}
	The processes $(\boldsymbol{X}(t))_{t \geq 0}$, $(\boldsymbol{\Theta}(t))_{t \geq 0}$, and $(\boldsymbol{\Xi}(t))_{t \geq 0}$, evolving respectively in $\overline{\bCW}_{\mathbb{R},2\pi,N}$, $\overline{\bCW}_{2\pi,N}$, and $\overline{\bCC}_{2\pi,N}$ (or their $\#$-counterparts), are called UDBM.
\end{defin}

As in Proposition \ref{cepa}, all of these stochastic processes never collide when they do not start from the boundary  and instantaneously leave the boundary otherwise.

\begin{coro}\label{coro:scalinggen}
	Theorem \ref{scalingthm1} admits the straighforward counterparts \eqref{scalinglim-theta}  where  $\bx\in \overline{\bCW}_{\mathbb R,2\pi,N}$ is replaced by  $\btheta\in \overline{\bCW}_{2\pi,N}^{\scriptscriptstyle \#}$ (resp.  $\boldsymbol \xi\in \overline{\bCC}_{2\pi,N}^{\scriptscriptstyle \#}$).
\end{coro}

	Similarly to Remark~\ref{rem:lift1}, one can state the following remark.

\begin{rem}\label{rem:lift2}
 Any continuous path $(\xi(t))_{t\geq 0}$ in $\overline{\bCC}_{2\pi,N}$ that never collides and leaves the boundary instantaneously admits a unique continuous lift $(\theta(t))_{t\geq 0}$ in $\overline{\bCW}_{2\pi,N}$ with $\theta_k(0)=x$, whenever $p_2(x)=\xi_k(0)$. Any continuous path $(\theta(t))_{t\geq 0}$ in $\overline{\bCW}_{2\pi,N}$ with the same properties admits a unique continuous lift $(x(t))_{t\geq 0}$ in $\overline{\bCW}_{\mathbb{R},2\pi,N}$ with $x_k(0)=y$, whenever $p_1(y)=\theta_k(0)$.
\end{rem}

The following follows from Remark~\ref{rem:lift2} and Proposition \ref{cepa}.

\begin{coro}\label{rem:strongmarkov}
	All representations of the UDBM in Definition~\ref{def:merwcontinuous} are continuous strong Markov processes. Moreover, $(\boldsymbol{\Theta}(t))_{t \geq 0}$, seen in $\mathbb{U}^N$, satisfies a similar SDE to $(\boldsymbol{X}(t))_{t \geq 0}$.
\end{coro}

\subsection{Proofs of Theorems \ref{scalingthm1} and \ref{spectraldecompodyson}}

We prove both theorems together: first we establish tightness of the processes, then we determine the full spectrum of $\bP$, and finally we identify the law of any limit point.

\subsubsection{Tightness results}
\label{subsec:tight}

We first prove tightness of $(\Theta(L^2 t))_{t \geq 0}$ in $\mathbb{C}^N$ using  that elementary symmetric polynomials admit  expansion in the Schur basis and Lemma \ref{rouche}. By the continuous mapping theorem, this yields tightness of $(\Xi(L^2 t))_{t \geq 0}$ by projection. The tightness of $(2\pi X(L^2 t)/L)_{t \geq 0}$ then follows by lifting, which requires controlling the winding numbers (Lemma \ref{WN}).

\bigskip

\noindent
{\it Preliminaries.} Since $\mathbb{U}^N$ is a compact set, the Markov property of $(\Theta(n))_{n \geq 0}$ and the tightness criterion \cite[Theorem 8.6, p.~137]{EK} reduce the proof to showing
\begin{equation}\label{tightnesscriterion}
\lim_{\delta \to 0} \limsup_{L\geq 2}  \sup_{\theta \in {\bCW}_{L,N}^{\scriptscriptstyle \#}} \mathbb{E}_\theta\left[\sup_{0 \leq t \leq \delta}\left\|\Theta(L^2 t) - \theta \right\|_\infty\right] = 0.
\end{equation}


We view $\theta = (\theta_1, \cdots, \theta_N) = (e^{2i\pi x_1/L}, \cdots, e^{2i\pi x_N/L}) \in \bCW_{L,N}^{\scriptscriptstyle \#}$ as the set of roots of a unitary polynomial. Hence, up to a permutation $\sigma \in \mathfrak{S}_N$,  $\theta$ is characterized by
\begin{equation}
e_k(\theta_1, \cdots, \theta_N) = e_k(e^{2i\pi x_1/L}, \cdots, e^{2i\pi x_N/L}),\quad  1 \leq k \leq N,
\end{equation}
where $e_k$ denotes the $k$-th elementary symmetric polynomial. Controlling the differences between $e_k(z_1,\cdots,z_N)$ and $e_k(\theta_1,\cdots,\theta_N)$ thus controls the distance between the $z_i$ and the $\theta_i$, as in the classical continuity of polynomial roots with respect to their coefficients.

We use the following classical lemma, whose proof is given at the end  for completeness.

\begin{lem}\label{rouche}
	Let $K \subset \mathbb{C}^N$ be compact. For $\theta \in K$ and $\varepsilon > 0$, let $C_1(\theta,\varepsilon),\ldots,C_\ell(\theta,\varepsilon)$ be the connected components of the union of $B(\theta_k,\varepsilon)$, and $i_k$ the index of the component containing $\theta_k$. Then, there exists $\lambda$, depending only on $K$ and $N$, such that for all $z \in K^N$,
	\begin{equation}
	\max_{1 \leq k \leq N} |e_k(z) - e_k(\theta)| < \lambda\,{\varepsilon^N}
	\quad \Longrightarrow\quad
	\exists \sigma \in \mathfrak{S}_N,\, \forall\, 1 \leq k \leq N,\, z_{\sigma(k)} \in C_{i_k}(\theta,\varepsilon).
	\end{equation}
\end{lem}

This approach is fruitful since one has the classical identity
\begin{equation}
e_k(\theta_1, \cdots, \theta_N) = s_{(1^k)}(\theta_1, \cdots, \theta_N),
\end{equation}
where $(1^k) = (1, \cdots, 1, 0, \cdots, 0)$ is the partition with $k$ ones and $N-k$ zeros.

\begin{rem}\label{rem:ekxi}
	The eigenfunction $s_{(1^k)}$ corresponds to the MESSEP eigenfunction $\psi_\xi / \psi_{{c}}$, where $\xi \in \mathcal{C}_{L,N}$ is obtained from the compact configuration ${c}$ by adding $+1$ to its $k$ topmost particles (i.e., rotating them by $2\pi/L$ in the complex plane). See Remark~\ref{rem:corres} for the correspondence.
\end{rem}

	\begin{lem}\label{lem:gap-k}
	 The eigenvalue $r_k$ associated with the MESSEP eigenfunction $e_k$ satisfies: 
	 \begin{equation}\label{eq:gapk}
	 1 - r_k
	  = \frac{k}{2}\left(1 - \frac{k - 1}{N}\right)\left(\frac{2\pi}{L}\right)^2 + \mathcal{O}\left(\frac{1}{L^4}\right),
	 \end{equation}
	 	as $L \to \infty$, where the constant in the last $\mathcal{O}$-term depends only on $N$.
\end{lem}

		\begin{proof}[Proof of Lemma \ref{lem:gap-k}] The proof of the following lemma follows the same lines as the spectral gap estimates in Section~\ref{sec:gap}. We get from \eqref{asymptoticexp0} that, for any $x$ in a compact set $J \subset \mathbb{R}$,
		\begin{equation}\label{asymptoticexpsuite}
		\cos\left(\frac{2\pi (x+1)}{L}\right) = \cos\left(\frac{2\pi x}{L}\right) - \left(x+\frac{1}{2}\right)\left(\frac{2\pi}{L}\right)^2 + \mathcal{O}\left(\frac{1}{L^4}\right),
		\end{equation}
		as $L \to \infty$, where the constant in the $\mathcal{O}$ term depends only on $J$. Recall that $\rho_\xi$ is defined in \eqref{Spectre}, and that $r_k = \rho_\xi / \rho_c = \rho_\xi / \rho$ for some $\xi\in\bCC_{L,N}$ as in Remark~\ref{rem:ekxi}. Then, we get
		\begin{equation}
		1 - r_k
		= \frac{1}{\rho} \left[2\sum_{i = N - k + 1}^{N} \left(c_i+\bgamma + \frac{1}{2}\right) \left(\frac{2\pi}{L}\right)^2 + \mathcal{O}\left(\frac{1}{L^4}\right) \right],
		\end{equation}
		as $L \to \infty$, where the constant in the last $\mathcal{O}$-term depends only on $N$.
		
		Using the left-hand side of \eqref{estimatepf} and noting that $(c_i)_{1\leq i\leq N}$ is an arithmetic sequence with common difference $1$ satisfying $c_N + \bgamma = (N - 1)/2$, we deduce \eqref{eq:gapk}.\end{proof}

Furthermore, the discrete-time Itô's formula allows us to write, for all $n\geq 0$, 
\begin{equation}\label{ito}
e_k\left(\Theta(n)\right) =  e_k\left(\theta\right) + (r_k - 1) \sum_{i=0}^{n-1} e_k(\Theta(i)) + M_n^{(k)},
\end{equation}
where  $(M_n^{(k)})_{n \geq 0}$ is a square-integrable, complex-valued martingale such that
\begin{equation}\label{accroissementmart}
M_{n+1}^{(k)} - M_n^{(k)} = e_k(\Theta(n+1)) - e_k(\Theta(n)) + (1 - r_k)\,e_k(\Theta(n)).
\end{equation}

Hence, to prove \eqref{tightnesscriterion}, using Lemma~\ref{rouche} and \eqref{ito}, it suffices to control

\begin{align}
e_k\left(\Theta(L^2 t)\right) - e_k(\theta) 
&= \underbrace{e_k\left(\Theta(L^2 t)\right) - e_k\left(\Theta(\lfloor L^2 t \rfloor)\right)}_{=: \Delta_{\lfloor L^2 t \rfloor}} + e_k\left(\Theta(\lfloor L^2 t \rfloor)\right) - e_k(\theta) \nonumber \\[5pt]
&= \Delta_{\lfloor L^2 t \rfloor} + (r_k - 1) \sum_{i=0}^{{\lfloor L^2 t \rfloor}-1} e_k(\Theta(i)) + M_{\lfloor L^2 t \rfloor}^{(k)}. \label{threeterms}
\end{align}

\bigskip

\noindent
\textit{Bound on the first two terms in \eqref{threeterms}.} First, introduce
\begin{equation}\label{eq:alphabeta}
\alpha = \sup_{1 \leq k \leq N} \sup_{\theta \in \mathbb{U}^N} |e_k(\theta)| \quad \text{and} \quad  
\beta = \sup_{1 \leq k \leq N} \sup_{1 \leq j \leq N} \sup_{x \in \mathbb{R}^N} 
\left| \frac{\partial e_k(e^{2i\pi x_1/L}, \cdots, e^{2i\pi x_N/L})}{\partial x_j} \right|.
\end{equation}
Observe that 
\begin{equation}\label{eq:boundlow}
\forall k\in\{1,\cdots,N\},\quad  \frac{k}{2} \left(1 - \frac{k - 1}{N} \right) \leq \frac{(N+1)^2}{8N}.
\end{equation}

Then, it follows from the mean value theorem that, for all $1 \leq k \leq N$, $L \geq 2$ and $\delta\geq 0$, 
\begin{equation}\label{maj1}
\sup_{0\leq t\leq \delta }|\Delta_{\lfloor L^2 t \rfloor}| = \sup_{0\leq t\leq \delta } \left| e_k\left(\Theta(L^2 t)\right) - e_k\left(\Theta(\lfloor L^2 t \rfloor)\right) \right| \leq \frac{2\pi \beta}{L}.
\end{equation}
Moreover, combining  \eqref{eq:boundlow} and \eqref{eq:gapk}, we obtain
\begin{equation}\label{maj2}
\sup_{0\leq t\leq \delta } 
\left| (r_k - 1) \sum_{i=0}^{\lfloor L^2 t \rfloor -1} e_k(\Theta(i)) \right| 
\leq 
C_1 \delta,
\end{equation}
for some constant $C_1$ depending only on $N,\alpha,\beta$.

\bigskip

\noindent
\textit{Bound on the martingale term in \eqref{threeterms}.} Using \eqref{maj1}, \eqref{eq:boundlow} and \eqref{accroissementmart}, we obtain
\begin{equation}
\mathbb{E}_\theta\left[|M_{n+1}^{(k)} - M_n^{(k)}|^2\right] 
\leq \frac{C_2}{L^2}. \label{boundmart}
\end{equation}
for some constant $C_2$ depending only on $N,\alpha,\beta$. Applying the Doob’s maximal inequality, we deduce that, for all $1\leq k\leq N$, $L\geq 2$, $\delta\geq 0$ and  $\eta > 0$,
\begin{equation}\label{maj3}
\mathbb{P}_\theta\left( \max_{0 \leq t \leq \delta} |M_{\lfloor L^2 t \rfloor}^{(k)}| \geq {\eta} \right) =\mathbb{P}_\theta\left( \max_{0 \leq n \leq \lfloor L^2 \delta \rfloor} |M_n^{(k)}| \geq \eta  \right) 
\leq \frac{C_2 \delta}{\eta ^2}
\end{equation}

\bigskip

\noindent
{\it Tightness criterion \eqref{tightnesscriterion}.} Assume that $L$ is sufficiently large and $\delta$ is sufficiently small so that
\begin{equation}\label{eq:cond}
\frac{2\pi\beta}{L} + C_1\, \delta \leq 2\eta.
\end{equation}
It follows from \eqref{maj1}, \eqref{maj2}, \eqref{maj3}, and \eqref{threeterms}, that
\begin{equation}\label{deltasym}
\mathbb{P}_\theta\left(\max_{0 \leq t \leq \delta} \max_{1 \leq k \leq N} \left| e_k\left(\Theta(L^2 t)\right) - e_k(\theta) \right| \geq 3\eta \right)
\leq \frac{C_2 \delta}{\eta^2}.
\end{equation}

For any $\varepsilon>0$, introduce
\begin{equation}
\Lambda = \left\{ \max_{0 \leq t \leq \delta} \max_{1 \leq k \leq N} \left| e_k\left(\Theta(L^2 t)\right) - e_k(\theta) \right| < \lambda{\varepsilon^N} \right\},
\end{equation}
where $\lambda_\varepsilon$ is given by Lemma~\ref{rouche}. On the event $\Lambda_\varepsilon$, this lemma implies that for all $0 \leq t \leq \delta$, there exists $\sigma_t \in \mathfrak{S}_N$ such that 
\begin{equation}
\forall k\in\{1,\cdots,N\},\quad \Theta_{\sigma_t(k)}(L^2 t) \in C_{i_k}(\theta,\varepsilon) \subset B(\theta_k,(2N-1)\varepsilon).
\end{equation}
Since $t \mapsto \Theta(L^2 t)$ is continuous and equals $\theta$ at $t=0$, we must have $\sigma_t=\mathrm{id}$ for all $0 \leq t \leq \delta$. 

Finally, using \eqref{deltasym} with $3\eta = \lambda\varepsilon^N$, we can write
\begin{equation}\label{majtightness}
\mathbb{E}_\theta\left[ \sup_{0 \leq t \leq \delta}  \left\| \Theta(L^2 t) - \theta \right\|_\infty \right]
\leq (2N-1)\varepsilon + 2\mathbb{P}(\Lambda^c_\varepsilon) \leq (2N-1)\varepsilon + \frac{2C_2 \delta}{\eta^2}.
\end{equation}

In the latter inequality, $\eta$ (and thus $\varepsilon$) can be chosen arbitrarily small, we only need  to take $L$ sufficiently large and $\delta$ sufficiently small (depending on $\eta$) so that \eqref{eq:cond} is satisfied. Since the constants are uniform with respect to $\theta \in \bCW_{L,N}^{\scriptscriptstyle \#}$, letting $L \to \infty$ and then $\delta \to 0$  yield \eqref{tightnesscriterion}.

\bigskip

\noindent
{\it Tightness of $(2\pi X(L^2 t)/L)_{t \geq 0}$.} To apply \cite[Theorem 8.6, p.~137]{EK}, we need to prove 
\begin{equation}\label{tightnesscriterion2}
\lim_{\delta \to 0} \limsup_{L\geq 2} \sup_{x \in \mathbb{Z}} \mathbb{E}_x \left[ \sup_{0 \leq t \leq \delta} \left\| \frac{2\pi X(L^2t)}{L} - \frac{2\pi x}{L} \right\|_\infty \right] = 0,
\end{equation}
but also, since $\overline{\bCW}_{\mathbb R,2\pi,N}$ is not a compact set, the tightness of $\left({2\pi X(L^2t)}/{L}\right)_{L\geq 2}$ for each fixed $t \geq 0$.

The first step follows readily from the tightness of $\left(\Theta(L^2 t)\right)_{t \geq 0}$ established above, together with the related arguments and the inequalities
\begin{equation}\label{tightnesscriterion22}
\forall \kappa>0,\quad \mathbb{P}_x\left( \sup_{0 \leq t \leq \delta} \left\| \frac{2\pi X(L^2 t)}{L} - \frac{2\pi x}{L} \right\|_\infty > \kappa \right) 
\leq \frac{1}{\beta \kappa} \, \mathbb{E}_\theta\left[ \sup_{0 \leq t \leq \delta} \|\Theta(L^2 t) - \theta\|_\infty \right],
\end{equation}
which hold for some $\beta>0$, as a consequence of the equivalence between the geodesic distance on $\mathbb{U}$ and the Euclidean distance on $\mathbb{U}\subset \mathbb C$, combined with Markov’s inequality.

Secondly, since $|X_N(L^2t) - X_1(L^2t)| \leq L - 1$, tightness of $(2\pi X(L^2t)/L)_{L\geq 2}$ reduces to that of the center of mass
\begin{equation}\label{eq:defcentermass}
\frac{2\pi S(L^2t)}{L} = \frac{2\pi}{L} \sum_{k=1}^N X_k(L^2t).
\end{equation}

\begin{lem}\label{WN}
	Let $s \in \mathbb{Z}^{\mathbb{N}}$ be such that $s_{n+1} - s_n \in \{\pm 1\}$ for all $n \in \mathbb{N}$. Then, for any $L\geq 3$,
	\begin{equation}\label{eq:wind}
	\frac{2\pi (s_n - s_0)}{L} = \frac{2\pi/L}{\sin(2\pi/L)} \sum_{k=0}^{n-1} \Im\left(e^{\frac{2i\pi (s_{k+1} - s_k)}{L}}\right).
	\end{equation}
\end{lem}

\begin{rem}\label{rem:winding}
	The integer part of the left-hand side of \eqref{eq:wind} is nothing but the algebraic winding number of the sequence $(e^{2i\pi s_k / L})_{k \geq 0}$ up to time $n$.
\end{rem}

	\begin{proof}[Proof of Lemma~\ref{WN}]
	Assume $s_0 = 0$. Let $\delta^+$ and $\delta^-$ be the numbers of indices $0 \leq k \leq n-1$ such that $s_{k+1}-s_k=+1$ and $s_{k+1}-s_k=-1$. Then $\delta^+ + \delta^- = n$, $s_n=\delta^+-\delta^-$ and
	\[
	\Im\left( \sum_{k=0}^{n-1} e^{\frac{2i\pi (s_{k+1} - s_k)}{L}} \right)
	= \Im\left( \delta^+ e^{2i\pi/L} + \delta^- e^{-2i\pi/L} \right)
	= s_n \,\sin\left( \frac{2\pi}{L} \right),
	\]
	which proves the claim.
\end{proof}

To proceed, observe that
\begin{equation}\label{eq:centermass}
e^{\frac{2i\pi }{L}\sum_{k=1}^N x_k }=e_N\left(e^{{2i\pi x_1}/{L}},\cdots,e^{{2i\pi x_N}/{L}}\right)\in\mathbb U.
\end{equation}
We then deduce from Lemma~\ref{WN}, equality \eqref{eq:centermass}, and \eqref{accroissementmart} applied with $k = N$ and $n = k$ that
\begin{align}
\frac{2\pi (S(n) - S(0))}{L} & \frac{2\pi/L}{\sin(2\pi/L)} \sum_{k=0}^{n-1} \Im\left(e_N(\Theta(k+1))\overline{e_N(\Theta(k))}\right),\\
& =   \frac{2\pi/L}{\sin(2\pi/L)} \sum_{k=0}^{n-1} \Im\left( e^{-\frac{2i\pi S_k}{L}} (M_{k+1}^{(N)} - M_k^{(N)})\right),\label{WNappli}
\end{align}
where, in the last equality, we use that the MESSEP spectrum, and hence $r_k$, is real. Note that the summands inside the imaginary part in \eqref{WNappli} are centered and uncorrelated.

Therefore, applying \eqref{boundmart}, we obtain
\begin{equation}\label{WNtight}
\mathbb V\left(\frac{2\pi (S(\lfloor L^2t\rfloor)-S(0))}{L}\right) \leq  C_3 t,
\end{equation}
for some constant $C_3$ independent on $L$. This implies the tightness of \eqref{eq:defcentermass}, and hence of each marginal of $\left({2\pi X(L^2 t)}/{L}\right)_{L \geq 2}$, thereby completing the proof.\hfill\qedsymbol\\

	\begin{proof}[Proof of  Lemma \ref{rouche}]
Let $R>0$ be such that $K \subset \{z\in\mathbb C:\ |z|\leq R\}$, and set $\lambda^{-1} = 1 + R + \cdots + R^N$. 

Let us introduce
		\begin{multline}
		P(X)=\prod_{k=1}^N(X-\theta_k)=\sum_{k=0}^N (-1)^{N-k} e_{N-k}(\theta_1,\cdots,\theta_N)X^{k}\\
		\mbox{and}\quad  Q(X)=\prod_{k=1}^N(X-z_k)=\sum_{k=0}^N (-1)^{N-k} e_{N-k}(z_1,\cdots,z_N)X^{k}.
		\end{multline} 
	Since for  all  $z\in \bigsqcup_{i=1}^\ell\partial C_{i}(\theta,\varepsilon)$ and $1\leq k\leq N$ one has  $|z-\theta_k|\geq \varepsilon$,  we obtain  
		\begin{equation}
		\kappa=\min_{1\leq i\leq \ell}\min_{z\in\partial C_i(\theta,\varepsilon)} |P(z)|\geq \varepsilon^N.
		\end{equation}

Assume that $|e_{k}(z_{1},\cdots,z_N)-e_{k}(\theta_1,\cdots,\theta_N)|\leq \delta$, for all $1\leq k\leq N$, with $0<\delta< \lambda \varepsilon^N$. Then, for every $1\leq i\leq \ell$ and $z\in\partial C_i(\theta,\varepsilon)$,
		$$|P(z)-Q(z)|\leq \delta /\lambda < \kappa\leq P(z).$$
		
Applying Rouché’s theorem, $P$ and $Q$ have the same number of roots (counted with multiplicities) in each connected component $C_i(\theta,\varepsilon)$, $1 \leq i \leq \ell$. This completes the proof.
	\end{proof}

\subsubsection{Spectral decomposition}

\label{sec:subspectraldecompocontinu}

Before showing that $\boldsymbol{\mathcal{S}}_{2\pi,N}$ forms an orthonormal set of eigenfunctions of $(\bP_t)_{t \geq 0}$ in $\mathbb L^2(\bCC_{2\pi,N},\boldsymbol{\mu})$, we first prove Proposition~\ref{orthonormal}.

\begin{proof}[Proof of Proposition \ref{orthonormal}]

As in the discrete setting, one readily checks that for all $m, m' \in \mathfrak{W}_N$,

\begin{equation}\label{ortho3}
\int_{(0,2\pi)^N} \Psi_m(\theta)\, \overline{\Psi_{m'}(\theta)}\, \lambda(d\theta)
= N! \int_{\bCC_{2\pi,N}} \Psi_m(\theta)\, \overline{\Psi_{m'}(\theta)}\, \lambda(d\theta)
= N! \; \delta_{m=m'}.
\end{equation}
We deduce that $\{\Psi_m : m \in \mathfrak{W}_N\}$ is an orthonormal set in the Hilbert space $\mathbb{L}^2(\bCC_{2\pi,N}, \lambda)$. 

Let us prove that it forms a Hilbert basis. For any $m \in \mathbb{Z}^N$, define
\begin{equation}
e_m(x) = \frac{1}{(2\pi)^{N/2}} \prod_{i=1}^N e^{i m_i x_i}.
\end{equation}
The trigonometric system $\{e_m : m \in \mathbb{Z}^N\}$ forms a complete orthonormal system in $\mathbb{L}^2((0,2\pi)^N,\lambda)$ equipped with its standard inner product $\langle \cdot, \cdot \rangle$.

In particular, for any antisymmetric function $f \in \mathbb{L}^2((0,2\pi)^N,\lambda)$ and any  $\sigma \in \mathfrak{S}_N$, one has $\langle f, e_{\sigma \cdot m} \rangle = \varepsilon(\sigma)\langle f, e_m \rangle$, which yields
\begin{equation}
f(x) = \sum_{m \in \mathbb{Z}^N} \langle f, e_m \rangle\, e_m(x) 
= \sum_{m \in \mathfrak{W}_N} \langle f, e_m \rangle {\sum_{\sigma \in \mathfrak{S}_N} \varepsilon(\sigma) e_{\sigma \cdot m}}(x) = e^{i\bgamma\sum_{i=1}^N x_i}\sum_{m \in \mathfrak{W}_N} \langle f, e_m \rangle \Psi_m(x) .
\end{equation}

We deduce from the latter identity that the set $\{\Psi_m : m \in \mathfrak{W}_N\}$ is dense in the space of square-integrable antisymmetric functions on $(0,2\pi)^N$. Since any square-integrable function on $\bCC_{2\pi,N}$ lifts naturally to an antisymmetric function on $(0,2\pi)^N$, the first part of the Proposition \ref{orthonormal} follows.

Moreover, $\boldsymbol{\mathcal{S}}_{2\pi,N}$ consists of smooth functions on $\overline{\bCC}_{2\pi,N}$ (polynomials). Since $\Psi_{\bc}$ is proportional to $\Psi$, which is bounded and strictly positive on $\bCC_{2\pi,N}$, it follows easily that $\boldsymbol{\mathcal{S}}_{2\pi,N}$ is an orthonormal set in $\mathbb{L}^2(\bCC_{2\pi,N}, \boldsymbol{\mu})$. Regarding the density of ${\rm span}(\boldsymbol{\mathcal{S}}_{2\pi,N})$ in the uniform topology, we apply the Stone--Weierstrass theorem. It suffices to note that this family is a $C^*$-algebra.

To this end, Proposition~\ref{prop:schureigen} together with standard properties of Schur polynomials (see \eqref{eq:littlewood}) ensures stability under multiplication. The fact that ${\rm span}(\boldsymbol{\mathcal{S}}_{2\pi,N})$ is closed under complex conjugation follows from the following lemma.

\begin{lem}\label{lem:closeconj}
	Let $\lambda$ be an integer partition and $\ell \geq \lambda_1$. Then $\mu_i := \ell - \lambda_{N - i + 1}$ for $1 \leq i \leq N$ defines an integer partition such that
	\begin{equation}\label{schureigenconj}
	\overline{s_\lambda(e^{i x_1},\cdots,e^{i x_N})} = e^{-i\ell \sum_{k=1}^N x_k}\, s_\mu(e^{i x_1},\cdots,e^{i x_N}).
	\end{equation}	
\end{lem}

\begin{proof}[Proof of Lemma \ref{lem:closeconj}]
	Noting that $\delta = (N-1,\cdots,0)$ can be obtained from $-\delta$ by reversing its coordinates and adding $N-1$ to each component, identity \eqref{schureigenconj} follows by standard computations similar to those employed below and above Remark~\ref{rem:roots}.
\end{proof}

Consequently, standard approximation arguments  allow us to conclude that $\boldsymbol{\mathcal{S}}_{2\pi,N}$ is a Hilbert orthonormal basis of $\mathbb{L}^2(\bCC_{2\pi,N}, \boldsymbol{\mu})$, which completes the proof.
\end{proof}

We now show that $\boldsymbol{\mathcal{S}}_{2\pi,N}$ is precisely the set of orthonormal eigenfunctions of $(\bP_t)_{t \geq 0}$.

\bigskip

\noindent{\it Eigenfunctions and eigenvalues.} 
Let $\mathcal{L}$ be the infinitesimal generator of $(\boldsymbol{X}(t))_{t \geq 0}$ and $\mathcal{L}^*$ its adjoint, defined  for sufficiently smooth functions on $\mathbb{R}^N$ by
\begin{equation}
\mathcal{L}f = \frac{(2\pi)^2}{N} \left( \frac{1}{2} \Delta f + \frac{\nabla \Psi}{\Psi} \cdot \nabla f \right)
\quad\text{and}\quad
\mathcal{L}^* g = \frac{(2\pi)^2}{N} \left( \frac{1}{2} \Delta g - \operatorname{div} \left( g \frac{\nabla \Psi}{\Psi} \right) \right),
\end{equation}

Let $H^1(\bCC_{2\pi,N})$ and $H^1_0(\bCC_{2\pi,N})$ be the  Sobolev spaces defined as the closures in $\mathbb{L}^2(\bCC_{2\pi,N},\lambda)$ of the operator $f \mapsto \nabla f$, acting respectively on smooth functions on $\overline{\bCC}_{2\pi,N}$ and on smooth functions with compact support in $\bCC_{2\pi,N}$. For any $f \in H^1(\bCC_{2\pi,N})$ and $g/\Psi \in H^1_0(\bCC_{2\pi,N})$, both $(\mathcal{L}f)\,g$ and $f\,(\mathcal{L}^*g)$ are integrable on $\bCC_{2\pi,N}$, and
\begin{equation}\label{chainrule}
\int_{\bCC_{2\pi,N}} (\mathcal{L}f)\,\overline{g} \,d\lambda = \int_{\bCC_{2\pi,N}} f\,\overline{\mathcal{L}^*g} \,d\lambda.
\end{equation}
The latter identity follows first for smooth functions and then by density.

\begin{lem}\label{eigenfunctionsIG}
	For every $m\in{\mathfrak W}_N$, one has
	\begin{equation}\label{spectruminfgene1}
	\mathcal L^* (\Psi_{\bc}\Psi_{m})=\frac{2\pi^2}{N}\left[\Psi_{\bc}\Delta\Psi_m-\Psi_m\Delta\Psi_{\bc}\right]=-E_m\Psi_{\bc}\Psi_{m},
	\end{equation}
	and
	\begin{equation}\label{spectruminfgene2}
\mathcal L f_m=-E_m f_m,\quad\text{with}\quad f_m=\frac{\Psi_{m}}{\Psi_{\bc}}.
	\end{equation}
\end{lem}

\begin{proof}[Proof of Lemma \ref{eigenfunctionsIG}]
	Equation~\eqref{spectruminfgene1} follows from direct computations. Using \eqref{chainrule} and the orthogonality of $\{\Psi_m : m \in \mathfrak{W}_N\}$ in $\mathbb{L}^2(\bCC_{2\pi,N},\lambda)$, we obtain
	\begin{equation}
	\int_{\bCC_{2\pi,N}} \mathcal{L} \left( \frac{\Psi_m}{\Psi_{\bc}} \right) \overline{\frac{\Psi_{m'}}{\Psi_{\bc}}} \, d\boldsymbol{\mu}
	= \int_{\bCC_{2\pi,N}} \frac{\Psi_m}{\Psi_{\bc}} \, \overline{\mathcal{L}^* (\Psi_{m'} \Psi_{\bc})} \, d\lambda
	= -E_m \delta_{m = m'}.
	\end{equation}

Recall that $\Psi_{\bc}$ equals $\Psi$ up to a unit-modulus complex constant. Since $\mathcal{L}(\Psi_m / \Psi_{\bc}) \in \mathbb{L}^2(\bCC_{2\pi,N}, \boldsymbol{\mu})$ and $\boldsymbol{\mathcal{S}}_{2\pi,N}$ is an orthonormal basis of $\mathbb{L}^2(\bCC_{2\pi,N}, \boldsymbol{\mu})$ by Proposition~\ref{orthonormal}, we obtain \eqref{spectruminfgene2}.
\end{proof}

By Proposition~\ref{cepa}, for any function $f$ continuous on $\overline{\bCW}_{\mathbb{R},2\pi,N}$, twice differentiable on ${\bCW}_{\mathbb{R},2\pi,N}$, with derivatives admitting continuous extensions to the boundary, we can write
\begin{equation}\label{Ito2}
f(\boldsymbol{X}_t) = f(\bx) + \int_0^t \mathcal{L}f(\boldsymbol{X}_s)\, ds +{\frac{2\pi}{\sqrt N}} \int_0^t \nabla f(\boldsymbol{X}_s)\, dB_s.
\end{equation}
In particular, if we take $f=f_m$ in the latter identity, and then the expectation, we  obtain 
\begin{equation}\label{inv}
\mathbb{E}_{\boldsymbol{x}}\left[f_m(\boldsymbol{X}(t))\right] = \bP_t f_m(\boldsymbol{x})= e^{-E_m t} f_m(\boldsymbol{x}).
\end{equation}

\bigskip

\noindent
{\it The reversible probability measure.} It follows from \eqref{inv} that, for all $m, m' \in \mathfrak{W}_N$ and $t \geq 0$,
\begin{align}
\int_{\bCC_{2\pi,N}} (\bP_t f_{m}) \, \overline{f_{m'}} \, d\boldsymbol{\mu} 
&= e^{-E_m t} \int_{\bCC_{2\pi,N}} \Psi_m \, \overline{\Psi_{m'}} \, d\lambda \notag\\
&= e^{-E_m t} \delta_{m = m'} \notag\\
&= \int_{\bCC_{2\pi,N}} f_m \, \overline{\bP_t f_{m'}} \, d\boldsymbol{\mu}.\label{eq:rev}
\end{align}

The identities in \eqref{eq:rev} extend to ${\rm span}(\{f_m : m \in \mathfrak{W}_N\})$, and hence to $C(\overline{\bCC}_{2\pi,N})$ by density (see Proposition~\ref{orthonormal}). In particular, we obtain $\boldsymbol{\mu}\bP_t f = \boldsymbol{\mu}f$ for all continuous functions $f$ on $\overline{\bCC}_{2\pi,N}$. 

Applying Jensen's inequality, we get that $\bP_t$ maps $\mathbb{L}^2(\bCC_{2\pi,N}, \boldsymbol{\mu})$ into itself, and that the operator $f \mapsto \bP_t f$ defines a contraction on $\mathbb{L}^2(\bCC_{2\pi,N}, \boldsymbol{\mu})$. By density of ${\rm span}(\{f_m : m \in \mathfrak{W}_N\})$ in $\mathbb{L}^2(\bCC_{2\pi,N}, \boldsymbol{\mu})$, we deduce that \eqref{eq:rev} extends to  $\mathbb{L}^2(\bCC_{2\pi,N}, \boldsymbol{\mu})$, which is precisely the reversibility property we were aiming for.

To prove uniqueness, let ${\boldsymbol{\nu}}$ be another invariant probability measure. Then ${\boldsymbol{\nu}}(f_m)=0$ for all $m \neq \boldsymbol{c}$ by \eqref{inv}, that is,
$\langle \mathds{1}, f_m \rangle_{\boldsymbol{\nu}} = 0$ for all $m \neq \boldsymbol{c}$,
where $\langle \cdot, \cdot \rangle_{\boldsymbol{\nu}}$ is the inner product in $\mathbb{L}^2(\bCC_{2\pi,N},{\boldsymbol{\nu}})$. Since ${\rm span}(\{ f_m : m \in \mathfrak{W}_N \})$  is dense in $\mathbb{L}^2(\bCC_{2\pi,N}, \boldsymbol{\nu})$, one gets
\begin{equation}
\mathbb{L}^2(\bCC_{2\pi,N}, \boldsymbol{\nu}) = \mathbb{C} \cdot \mathds{1} \overset{\perp}{\oplus} \overline{{\rm span}(\{ f_m : m \neq \bc \})}^{(\boldsymbol{\nu})}.
\end{equation}

Therefore, any $f \in C(\overline{\bCC}_{2\pi,N})$ (thus square integrable with respect to $\boldsymbol{\mu}$ and $\boldsymbol{\nu}$) can be written as $f=\alpha \mathds{1}+g$, with $g$ orthogonal to $\mathbb{C}\cdot\mathds{1}$ in both $\mathbb{L}^2(\bCC_{2\pi,N},\boldsymbol{\nu})$ and $\mathbb{L}^2(\bCC_{2\pi,N},\boldsymbol{\mu})$. Consequently $\alpha=\boldsymbol{\mu}(f)=\boldsymbol{\nu}(f)$, and thus $\boldsymbol{\nu}=\boldsymbol{\mu}$.

\bigskip

\noindent
{\it Convergence of the series \eqref{decompo}.}
We prove the pointwise identity~\eqref{decompo} and the uniform convergence of all partial derivatives for $t>0$. At this stage, we have shown that \eqref{decompo} holds  in $\mathbb{L}^2(\bCC_{2\pi,N},\boldsymbol{\mu})$, since $f \mapsto \bP_t f$ is continuous on this space.

To go further, one can see that
\begin{equation}
C_m = \sup_{x \in \overline{\bCC}_{2\pi,N}} \left| f_m(x) \right|=\sup_{x \in \overline{\bCC}_{2\pi,N}} \left| \frac{\Psi_m(x)}{\Psi_{\bc}(x)} \right| = \sup_{\theta \in \mathbb{U}^N} \left| s_\lambda(\theta) \right|,
\end{equation}
where $\lambda$ is a partition satisfying $\lambda_i - \lambda_j + j - i = m_{N - i + 1} - m_{N - j + 1}$. We refer to Proposition~\ref{prop:schureigen} and Remark \ref{rem:mini} for the correspondence.

 Since $s_\lambda$ is a symmetric polynomial with non-negative coefficients, one has
\begin{equation}
C_m = s_\lambda(1,\cdots,1) = \prod_{1 \leq i < j \leq N} \frac{\lambda_i - \lambda_j + j - i}{j - i},
\end{equation}
where the last equality is a classical specialization formula \eqref{youngtab00} giving the number of semi-standard Young tableaux of shape $\lambda$ with entries in $\{1,\cdots,N\}$. 

We obtain
\begin{equation}
C_m \leq (m_N - m_1)^{\frac{N(N-1)}{2}}.
\end{equation}
Furthermore, one checks that there exists $d > 0$ such that, for all $m \in \mathfrak{W}_N$ and $\ell \geq 0$, if $E_m = \ell$, then both $m_1^2$ and $m_N^2$ are bounded  by $d (\ell + 1)$, leading to $m_N - m_1 \leq 2\sqrt{d (\ell + 1)}$ and 
\begin{equation}\label{eq:normalbound}
\mathrm{card}\left\{ m \in \mathfrak{W}_N : E_m = \ell \right\} \leq \left(2\sqrt{d (\ell + 1)} + 1\right)^{N}\quad\text{and}\quad C_m \leq {(2\sqrt{d (\ell + 1)})}^{\frac{N(N-1)}{2}}.
\end{equation}

Grouping in \eqref{decompo} the terms with $E_m=\ell$, where $\ell=2\pi^2 n/N$ with $n\geq 0$, it follows from \eqref{eq:normalbound} that the series converges normally in the supremum norm on $\overline{\bCC}_{2\pi,N}$.

Regarding the partial derivatives, the representation \eqref{schureigen} and the multivariate Leibniz rule give, allow us to write, 0for any multi-index $\alpha = (\alpha_1, \cdots, \alpha_N)$,
\begin{equation}\label{eq:leibnitzmulti}
\partial^{\alpha} f_m(x) = \sum_{\beta \leq \alpha} \binom{\alpha}{\beta} (i\delta)^{|\alpha - \beta|} e^{i\delta \sum_{j=1}^N x_j} \, \partial^{\beta} s_\lambda(e^{ix_1},\cdots,e^{i x_N}).
\end{equation}
Since the integers $t_j$ in the combinatorial Schur-identity~\eqref{youngtab00}  are bounded by $|\lambda| = \lambda_1 + \cdots + \lambda_N$, we get that   $|\partial^{\beta} s_\lambda(e^{i x_1}, \cdots, e^{i x_N})| \leq |\lambda|^{|\beta|} s_\lambda(1, \cdots, 1)$,
and then  
\begin{equation}
C_m^{\alpha}:=\sup_{x\in\overline{\bCC}_{2\pi,N}}|\partial^{\alpha} f_m(x)| \leq (|\delta|+|\lambda|)^{|\alpha|} s_\lambda(1, \cdots, 1).
\end{equation}
One can choose $\lambda$ with $\lambda_N = 0$. Then necessarily $\delta = m_1 + p$, where $N = 2p$ or $N = 2p+1$, and each $\lambda_j$ is bounded by $m_N - m_1$. We refer to Proposition~\ref{prop:schureigen} and Remark~\ref{rem:mini} for details. We deduce that $C_m^{\alpha}$ grows at most polynomially in $E_m$, and the argument proceeds as previously.

Finally, the series~\eqref{decompo}, as well as all its partial derivatives, is normally convergent, this completes the proof of Theorem~\ref{spectraldecompodyson}. 	The first part of Corollary~\ref{coro:convergenceL2} follows directly from these convergence properties. The bound \eqref{speed} is  a straightforward consequence of~\eqref{decompo}.\hfill\qed

\subsubsection{Uniqueness of the limit point}

To identify the limit, we use the spectral decomposition established in the previous section.

Given $m \in \mathfrak W_N$ and $L \geq m_N - m_1 + 1$, let $j \in \mathbb Z$ be such that $m_N + jL < L$ and $m_N + jL$ is maximal, and set $q = m + j(1,\cdots,1)$. Necessarily, there exists $d$ such that $-L < q_1 < \cdots < q_{d-1} < 0$ and $0 \leq q_d < \cdots < q_N < L$, in such a way that 
$$(\xi_{1},\cdots \xi_N)=(q_d,\cdots,q_N,q_1+L,\cdots,q_{d-1}+L)=\tau_L^{d-1}\cdot (q_1,\cdots,q_N)\in \mathcal C_{L,N},$$  
where the operator $\tau_L$ is defined in \eqref{tauchap} and $\tau$ denotes the cycle $(1\,\cdots\, N)$. Similarly, one checks that $c = \tau_L^p \cdot \bc$, where $N = 2p$ or $N = 2p+1$ (see Remark~\ref{rem:compactrep}).

This yields, for all $x \in \mathbb Z^N$,
\begin{equation}
f_m\left(\frac{2\pi x}{L}\right)
= \varepsilon(\tau)^{p+d-1}\frac{\psi_\xi(x)}{\psi_c(x)},
\end{equation}
Since the latter coincides with the MESSEP eigenfunction given in \eqref{eigenmessep}, we may write
\begin{equation}\label{eq:approxmessep}
\mathbb{E}_{x}\left[f_m\left(\frac{2\pi X(n)}{L}\right)\right] = \left(\frac{\rho_\xi}{\rho_c}\right)^n f_m\left(\frac{2\pi x}{L}\right).
\end{equation}

We shall  prove that
\begin{equation}\label{energy}
\lim_{L \to \infty} \left(\frac{\rho_\xi}{\rho_c}\right)^{\lfloor L^2 t \rfloor} = e^{-E_m t}.
\end{equation}
First, by periodicity,
\begin{equation}
\rho_\xi = 2 \sum_{i=1}^N \cos\left(\frac{2\pi (m_i + \bgamma)}{L}\right)
\quad \text{and} \quad
\rho_c = 2 \sum_{i=1}^N \cos\left(\frac{2\pi (\bc_i + \bgamma)}{L}\right).
\end{equation}
Moreover, generalizing \eqref{asymptoticexpsuite}, one uniformly for $x, h$ in compact sets:
\begin{equation}\label{estimatecoslow}
\cos\left(\frac{2\pi (x + h)}{L}\right) = \cos\left(\frac{2\pi x}{L}\right) - \frac{1}{2} \left((x + h)^2 - x^2\right) \left(\frac{2\pi}{L}\right)^2 + \mathcal{O}\left(\frac{1}{L^4}\right),
\end{equation}

 Applying \eqref{estimatecoslow} with $x = \boldsymbol{c}_k + \bgamma$ and $h = m_k - \boldsymbol{c}_k$ for all $1 \leq k \leq N$, and using the asymptotic expansion of $\rho_c = \rho$ given in \eqref{estimatepf}, we obtain
 \begin{equation}
 \frac{\rho_\xi}{\rho_c} = 1 - \frac{2\sum_{i=1}^N \left( \cos\left( \frac{2\pi (\boldsymbol{c}_i + \bgamma)}{L} \right) - \cos\left( \frac{2\pi (m_i + \bgamma)}{L} \right) \right)}{\rho_c}
 = 1 - \frac{2\pi^2}{L^2} E_m + \mathcal{O}\left( \frac{1}{L^4} \right),
 \end{equation}
 from which \eqref{energy} follows.

Assuming \eqref{eq:convstaringpoint}, we deduce from the above estimates that, for all $f \in \boldsymbol{\mathcal S}_{2\pi,N}$ and all fixed $t \geq 0$,
\begin{equation}\label{eq:convmarginal}
\lim_{L\to\infty}\mathbb{E}_{x}\left[f\left(\Xi(L^2 t)\right)\right] = \mathbb E_{\bx}[f(\bXi(t))].
\end{equation}
Since ${\rm span}(\boldsymbol{\mathcal S}_{2\pi,N})$ is dense in $C(\overline{\bCC}_{2\pi,N})$ (see Proposition~\ref{orthonormal}), one can extend \eqref{eq:convmarginal} to $C(\overline{\bCC}_{2\pi,N})$, thereby showing the convergence of every one-dimensional marginal distribution. 

In order to obtain the functional convergence of $(\Xi(L^2 t))_{t \geq 0}$ as $L \to \infty$, it remains to show that every finite-dimensional marginal distribution converges to those of $(\bXi(t))_{t \geq 0}$.

This follows from Theorem~2.5 in \cite[p.~167]{EK} (see Equations~(2.25) and~(2.26)), once we verify that $(\bP_t)_{t \geq 0}$ is a Feller semigroup on $C(\overline{\bCC}_{2\pi,N})$. The latter follows from Theorem~\ref{spectraldecompodyson} and Corollary~\ref{coro:convergenceL2}.
This ends the proof of the functional convergence for $\Xi$. 

Let $(\boldsymbol{X}(t))_{t \geq 0}$ be a limit point of the sequence $\left({2\pi X(L^2 t)}/{L}\right)_{t \geq 0}$, whose existence is guaranteed by the tightness result in Section~\ref{subsec:tight}. Necessarily, the projection of $(\boldsymbol{X}(t))_{t \geq 0}$ coincides in law with $(\bXi(t))_{t \geq 0}$. In other words, $(\boldsymbol{X}(t))_{t \geq 0}$ is a lift of $(\bXi(t))_{t \geq 0}$, and by Remark~\ref{rem:lift2}, it follows that $(\boldsymbol{X}(t))_{t \geq 0}$ is itself a UDBM on $\overline{\bCW}_{\mathbb{R},2\pi,N}$, thereby completing the proof of Theorem~\ref{scalingthm1}.\hfill\qedsymbol

\section{Proofs of the results in Section \ref{sec:hydro}}

\label{sec:proofhydro}

We now give the detailed proofs of Theorems \ref{thm:hydro}, \ref{prop:regularityPDE} and \ref{thm:pack}, as outlined in Section~\ref{sec:sketch}.

\subsection{Proofs of character Lemmas~\ref{lem:characteridentities} and~\ref{lem:variancehook}}

\begin{proof}[Proof of Lemma~\ref{lem:characteridentities}]
	Denote by $S_{n,j}(\pi)$ the sum in the left-hand side of~\eqref{lem:character}. It follows from  the identities $\chi^{\scriptscriptstyle \{n|k\}}_{\scriptscriptstyle(n)} = (-1)^k$ and the orthogonality  the  characters  that
	\begin{equation}\label{eq:inititcharacter}
	S_{n,0}(\pi)=
	\begin{cases}
	n, & \text{if } \pi=(n),\\
	0, & \text{otherwise.}
	\end{cases}
	\end{equation}

To establish the formula \eqref{lem:character} for arbitrary $j \geq 0$, we proceed by induction, relying on the recursive form of the Murnaghan--Nakayama rule, which may be stated as follows:
\begin{equation}\label{murnaghamnakyama}
\chi_\pi^\lambda
=
\sum_{\mu \in \mathrm{BS}(\lambda,\pi_1)}
(-1)^{\mathrm{ht}(\mu)}\,
\chi_{\pi\setminus \pi_1}^{\lambda\setminus \mu},
\end{equation}
where the sum ranges over the set $\mathrm{BS}(\lambda,\pi_1)$ of border strips of size $\pi_1$ in the Young diagram of shape $\lambda$. A border strip $\mu$ is a connected skew shape containing no $2\times 2$ square, whose removal from $\lambda$ yields another Young diagram. The height $\mathrm{ht}(\mu)$ is the number of rows it spans minus one.
	
In the present setting, when $\lambda$ is the hook shape $\{n|k\}$, there are at most two border strips of size $\pi_1$ that can be removed. Consequently, the general formula~\eqref{murnaghamnakyama} simplifies to
\begin{equation}\label{MNrecnew}
\chi^{\{n|k\}}_\pi
=
(-1)^{\pi_1-1}\,\chi^{\{n-\pi_1|k-\pi_1\}}_{\pi\setminus \pi_1}
+\chi^{\{n-\pi_1|k\}}_{\pi\setminus \pi_1}.
\end{equation}
The first term in \eqref{MNrecnew} vanishes when $k < \pi_1 < n$, and the second when $n - k - 1 < \pi_1 < n$.

Performing a simple change of variables, one obtains the identities:
	\begin{equation}
	\sum_{k=0}^{n-1} (-1)^k \, \chi^{\{n - \pi_1 | k - \pi_1\}}_{\pi \setminus \pi_1} \, (n - 2k - 1)^j
	= \sum_{k=0}^{n - \pi_1 - 1} (-1)^{k+\pi_1} \, \chi^{\{n - \pi_1 | k\}}_{\pi \setminus \pi_1} \, ((n - \pi_1 - 2k - 1) - \pi_1)^j,
	\end{equation}
	and
	\begin{equation}
	\sum_{k=0}^{n-1} (-1)^k \, \chi^{\{n - \pi_1 | k\}}_{\pi \setminus \pi_1} \, (n - 2k - 1)^j
	= \sum_{k=0}^{n - \pi_1 - 1} (-1)^k \, \chi^{\{n - \pi_1 | k\}}_{\pi \setminus \pi_1} \, ((n - \pi_1 - 2k - 1) + \pi_1)^j.
	\end{equation}
Combining these two identities, \eqref{MNrecnew}, and  \eqref{murnaghamnakyama}, we obtain
	\begin{equation}\label{iterate}
	S_{n,j}(\pi) = \sum_{k=0}^{n - \pi_1 - 1} (-1)^k \, \chi^{\{n - \pi_1 | k\}}_{\pi \setminus \pi_1} 
	\left[ ((n - \pi_1 - 2k - 1) + \pi_1)^j - ((n - \pi_1 - 2k - 1) - \pi_1)^j \right].
	\end{equation}
Expanding the difference of powers yields
	\begin{equation}\label{eq:recrelation}
	S_{n,j}(\pi) = \sum_{i_1=0}^{j-1} \binom{j}{i_1} \left( 1 - (-1)^{j - i_1} \right) \pi_1^{j - i_1} \, S_{n - \pi_1, i_1}(\pi \setminus \pi_1).
	\end{equation}

Assume that $0\leq j\leq \ell(\pi)-2$. Noting that $\ell(\pi\setminus \pi_1)=\ell(\pi)-1$ and $i_1\leq j-1$ in  \eqref{eq:recrelation}, iterate this recurrence relation allow to express $S_{n,j}(\pi)$ as a linear combination of terms of the form $S_{m,0}(\sigma)$ with $m\geq 2$ and $\ell(\sigma)\geq 2$. All such terms vanish by~\eqref{eq:inititcharacter}, and therefore $S_{n,j}(\pi)=0$.

It remains to consider the case $j=\ell(\pi)-1$. Here, the only sequence of indices $i_1>i_2>\cdots$ produced by the recursive application of~\eqref{eq:recrelation} that yields a non-zero terminal term $S_{m,0}(\sigma)$ is the strictly decreasing chain $j,j-1,\ldots,1,0$, with $\sigma=(m)$. This gives
\begin{equation}
\binom{j}{j-1} 2\pi_1 \, \binom{j-1}{j-2} 2\pi_2 \cdots \binom{1}{0} 2\pi_{j} \, S_{\pi_\ell,0}((\pi_\ell)).
\end{equation}
which coincides with the right-hand side of~\eqref{lem:character}. This completes the proof.
\end{proof}

\begin{proof}[Proof of Lemma~\ref{lem:variancehook}] We prove the four statements successively.

\medskip

\noindent
\textit{Point 1.} The identity~\eqref{eq:conjhook} follows from the correspondence between configurations $\xi$ and partitions $\lambda$ described in the proof of Proposition~\ref{Schur}. Set $\lambda=\{n|k\}$ and $\overline{\lambda}=\{L-n|N-1-k\}$ and let $\bc+\bgamma=(-(p-\bgamma),\ldots,(p-\bgamma))$ be the symmetric compact configuration (see Remark~\ref{rem:compactrep}).
	
\begin{figure}[!h]
	\centering
	\includegraphics[scale=0.7]{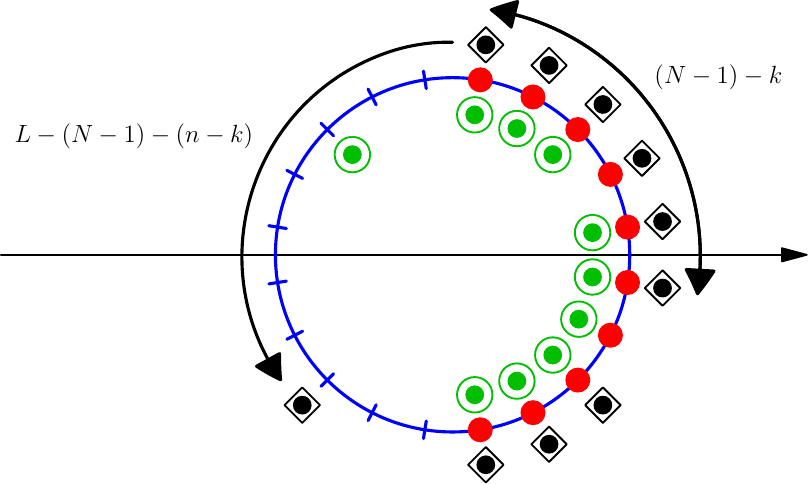}
	\captionsetup{width=13cm}
	\caption{\small The dots surrounded by a circle correspond to the configuration associated with the hook partition $\{n|k\}$ (here $n=6$ and $k=3$), whereas the dots surrounded by a square correspond to the hook partition $\{L-n|N-k-1\}$. They can be deduced from each other by conjugation in the complex plane.}
	\label{fig:hookschurconj}
\end{figure}

In the space of $\mathbb U$-configurations, those associated with $\overline{\lambda}$ and $\lambda$ are obtained from one another by conjugation (up to a permutation accounting for the order, see Figure~\ref{fig:hookschurconj}). Furthermore, one has

 \begin{align}
	\overline{s_{\{n|k\}}\left(e^{2i\pi x_1/L},\cdots,e^{2i\pi x_N/L}\right)}
	&=
	\frac{\displaystyle
		\det(e^{\frac{2i\pi}{L}(-\lambda_j-\mathbf{c}_{N-j+1}-\boldsymbol{\gamma})\,x_l})}
	{\displaystyle
		\det(e^{\frac{2i\pi}{L}(-\mathbf{c}_{N-j+1}-\boldsymbol{\gamma})x_l})} \label{eq:firste} \\
	&= \frac{\displaystyle
		\varepsilon(\overline \sigma)\det(e^{\frac{2i\pi}{L}(\overline \lambda_j+\bc_{N-j+1}+\bgamma) x_l})}
	{\displaystyle
		\varepsilon(\sigma)\det(e^{\frac{2i\pi}{L}(\mathbf{c}_{N-j+1}+\boldsymbol{\gamma})x_l})},
	\end{align}
	where $\sigma(j)=N-j+1$ for all $1\leq j\leq N$ and 
	\begin{equation}
	\overline\sigma \;=\;
	\begin{pmatrix}
	1 & 2 & 3 & \cdots & N-1 & N\\
	1 & N & N-1 & \cdots & 3   & 2
	\end{pmatrix}.
	\end{equation}
	
The first equality in \eqref{eq:firste} follows from $\bc+\bgamma=\sigma\cdot\delta-(p-\bgamma,\ldots,p-\bgamma)$, where $\delta=(N-1,\ldots,0)$, together with the definition of $s_{\{n|k\}}$ via the Vandermonde determinant~\eqref{eq:schurdef}. 

For the second one,  observe that $\sigma\cdot(\bc+\bgamma)=-(\bc+\bgamma)$,  $p-\bgamma-(N-1)=-(p-\bgamma)$, and 
\begin{equation}
\overline{\sigma}\cdot\bigl(\overline{\lambda}+\sigma\cdot(\bc+\bgamma)\bigr)
= -\lambda-\sigma\cdot\bc-\bgamma+(L,0,\ldots,0).
\end{equation}

Since $\varepsilon(\sigma)=\varepsilon(\overline{\sigma})=(-1)^p$ when $N=2p+1$, while $\varepsilon(\sigma)=(-1)^p$ and $\varepsilon(\overline{\sigma})=(-1)^{p-1}$ when $N=2p$, we obtain the desired result.

\medskip

\noindent	
{\it Point 2.} 	 By using the Frobenius formula \eqref{eq:powersumhook}, the non-recursive version of the Murnaghan--Nakayama rule \eqref{murnaghamnakyama}, and the first point \eqref{eq:conjhook}, one can write
\begin{align}
p_n(z) \, p_{-n}(z)
&= p_n(z) \sum_{l=0}^{n-1} (-1)^l\, \overline{s_{\{n|l\}}(z)} \nonumber \\
&= (-1)^{2\bgamma}\sum_{l=0}^{n-1} (-1)^l\, p_n(z)\, s_{\,\overline \lambda_l}(z)\\
&= (-1)^{2\bgamma}\sum_{l=0}^{n-1} (-1)^l\, \sum_{\mu\in{\rm BS}_n(\overline\lambda_l )}(-1)^{{\rm ht}(\mu\setminus \overline\lambda_l)} s_{\mu}(z),\label{eq:powerprod}
\end{align}
where we denote by ${\rm BS}_n(\lambda)$  the set of partitions $\mu$ such that $\mu\setminus\lambda$ is a border strip of size $n$ and $\overline{\lambda}_{l}=\{L-n|N-1-l\}$. The set ${\rm BS}_n(\overline{\lambda}_l)$ admits the following geometric description:
\begin{enumerate}
	\item 
	It contains the double-hook partition $\{n|k,l\}$, obtained by attaching the border strip $\{n|k\}$ of height $k$ at the inner corner of the Young diagram of $\overline{\lambda}_l$.
	\item 
	It also contains two additional diagrams: 
	\begin{enumerate}
	\item[i.] One obtained by attaching the border strip $(1^n)$ along the leg of $\overline{\lambda}_l$, producing a border strip of height $n-1$.
	\item[ii.] An another obtained by attaching the border strip $(n)$ along the arm of $\overline{\lambda}_l$, giving a border strip of  height $0$.
	\end{enumerate}
\end{enumerate}

 In the  case  \text{2.i.}, one has $\mu=\{L\,|\,N-1-l+n\}$. Since $l<n$, the number of rows of  $\mu$ exceeds $N$, and therefore the corresponding Schur polynomial in $N$ variables vanishes (see Remark \ref{specialization}).
 
 In the  case  \text{2.ii.},  one obtains $\mu=\{L\,|\,N-1-l\}$. This partition corresponds to a configuration on the circle in which all particles initially located in the compact configuration  return to that configuration, up to a permutation. Consequently, to compute $s_\mu$ -- which is necessarily a constant equal to $\pm1$ -- it suffices to determine the signature of this permutation.
 
  Observe that
$\mu_1+\delta_1=\delta_{N-l}$, $\mu_2+\delta_2=\delta_1,\ldots,\mu_{N-l}+\delta_{N-l}=\delta_{N-l-1}$, and $\mu_{N-l+1}+\delta_{N-l+1}=\delta_{N-l+1},\ldots,\mu_N+\delta_N=\delta_N$. Hence, the corresponding permutation is the cycle given by
\begin{equation}
\begin{pmatrix}
1 & 2 & \cdots & N-l & N-l+1 & \cdots & N\\
N-l & 1 & \cdots & N-l-1 & N-l+1 & \cdots & N
\end{pmatrix},
\end{equation}
whose signature is equal to $(-1)^{N-1-l}$.

We deduce from \eqref{eq:powerprod} that
\begin{equation}\label{eq:prodsignature}
p_n(z) \, p_{-n}(z)= (-1)^{2\bgamma}\sum_{l=0}^{n-1}\sum_{k=0}^{n-1} (-1)^{l+k}\, s_{\{n|k,l\}}(z)+\underbrace{(-1)^{N-1+2\bgamma}}_{=1} n.
\end{equation}

\medskip

\noindent
\textit{Point 3.}
We rely crucially on results of Koike~\cite{Koike} and Stembridge~\cite{Stembridge} on rational Schur functions. For the reader’s convenience, we briefly recall the key points.

First, the rational Schur function $s_\lambda$ can be defined as in~\eqref{eq:schurdef}, allowing $\lambda_1\geq\cdots\geq\lambda_N$ to be non-negative integers (see also Remark \ref{rem:mini}). Such a sequence $\lambda$ is called a {signature}. In the latter two references, it is also referred to as a {staircase} or a {mixed partition}. In this case, $s_\lambda(z)$ is a rational function of the indeterminates $z_1,\ldots,z_N$. A combinatorial generating description of $s_\lambda$, analogous to~\eqref{youngtab00}, is given in~\cite[Def.~2.3]{Stembridge}, although it is less straightforward to state. In addition, Koike~\cite{Koike} provides a Jacobi--Trudi type formula for $s_\lambda$, that is  a determinantal expression in terms of elementary symmetric functions (see Definition~2.1 in the corresponding paper).

To proceed further,  denote by $\lambda=[\alpha,\beta]$ (or $[\alpha,\beta]_N$ to indicate its length), the signature
\begin{equation}
\lambda=(\alpha_1,\ldots,\alpha_{\ell(\alpha)},0,\ldots,0,-\beta_{\ell(\beta)},\ldots,-\beta_1),
\end{equation}
 where $\alpha$ and $\beta$ are ordinary partitions of lengths $\ell(\alpha)$ and $\ell(\beta)$.
For an $N$-tuple $z$ of non-zero complex numbers,  set $z^{-1}=(z_1^{-1},\ldots,z_N^{-1})$. As shown at the beginning of the proof of Proposition~2.6 (page~74) in~\cite{Koike}, the Jacobi--Trudi type formula mentioned above is employed to establish 
\begin{equation}
s_{[\alpha,\beta]}(z)
=
\sum_{\tau}(-1)^{|\tau|}\,
s_{\alpha\setminus\tau}(z)\,
s_{\beta\setminus\tau^\prime}(z^{-1}).
\end{equation}

We now relate this to the double-hook partition $\{n|k,l\}$ and its associated Schur function. Consider the signature of length $N$ defined by
\begin{equation}
\lambda=(n-k,1,\ldots,1,0,\ldots,0,-1,\ldots,-1,-(n-l)),
\end{equation}
where $1$ appears $k$ times, $-1$ appears $l$ times, and $0$ appears $N-2n$ times. This signature parametrizes the configuration on the circle corresponding to the double hook $\{n|k,l\}$. The first $n$ entries describe the displacements, in the trigonometric direction, of the first $n$ particles from the compact configuration, while the last $n$ entries describe the displacements, in the opposite direction, of the last $n$ particles, the remaining particles stay fixed. See Figure~\ref{fig:doublehook}.

In particular, we obtain $s_{\{n|k,l\}}(z)=\pm s_{\lambda}(z)$. To determine the sign, it suffices to compute the signature of the induced permutation (see also the proof of Point~2.\@ above). One readily verifies that the relevant cycle is $(1,N,N-1,\ldots,2)$, which concludes the proof.

\medskip

\noindent
\textit{Point 4.}
The proof is a direct application of the Littlewood--Richardson rule~\eqref{eq:littlewood}. It amounts to identifying the admissible partitions $\nu$ and counting the skew semistandard Young tableaux of shape $\lambda\setminus\mu$ with content $\nu$, when $\lambda$ and $\mu$ are hook partitions.

As a matter of fact, there are at most two such admissible contents $\nu$, each giving rise to a unique skew semistandard Young tableau. See Figure~\ref{fig:littlewood}..

\begin{figure}[!h]
	\centering
	\includegraphics[width=14cm]{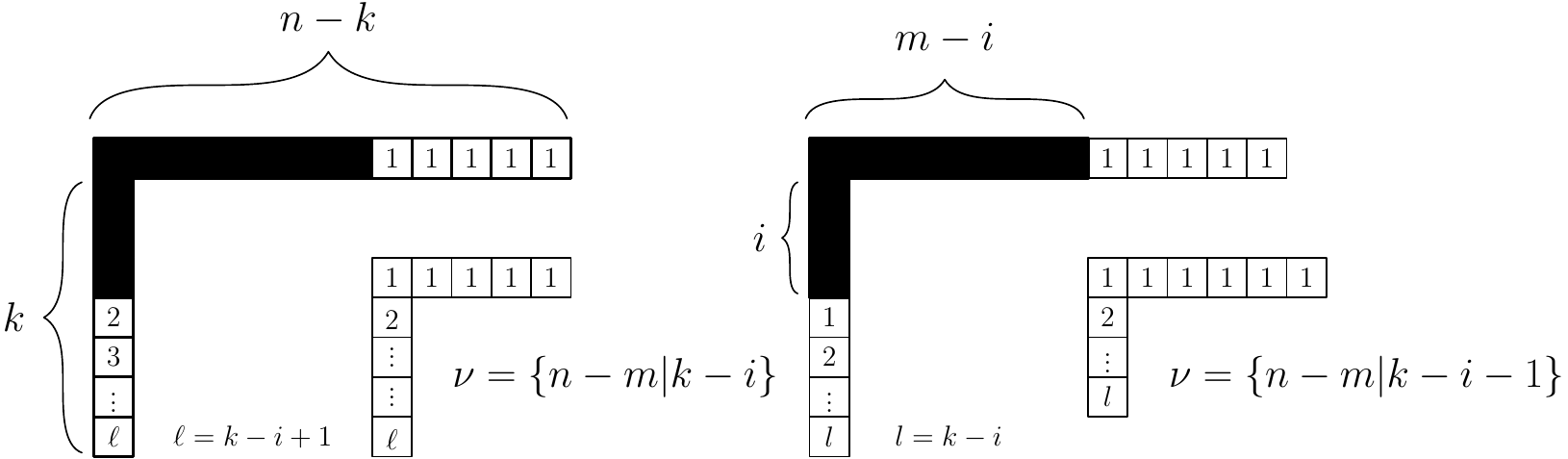}
	\captionsetup{width=16cm}
\caption{\small The two skew semistandard Young tableaux of shape $\lambda\setminus\mu$ with contents $\nu$. Here $\lambda=\{n|k\}$, $\mu=\{m|i\}$ (black squares), and $\nu$ is either $\{n-m|k-i\}$ (left) or $\{n-m|k-i-1\}$ (right).}
\label{fig:littlewood}
\end{figure}
\end{proof}

\subsection{Proofs of the technical Lemmas~\ref{lem:eigenhookdevasym}, \ref{lem:eigendoublehookdevasym} and \ref{lem:finallemvar}}

\begin{proof}[Proof of Lemma \ref{lem:eigenhookdevasym}] Let us set 
	\begin{equation}\label{eq:hookeigenvalueprecisefunction}
	u_{L,N}^{(n)}(x)
	=
	L^2\,\frac{2 \sin\!\left(\frac{\pi}{L}\right)\sin\!\left(\frac{\pi n}{L}\right)\sin\!\left(\frac{\pi N}{L}+x\right)}{\sin\!\left(\frac{\pi N}{L}\right)}\quad\text{and}\quad 	h_{t,L,N}^{(n)}(x)
	=
	\left(1-\frac{u_{L,N}^{(n)}(x)}{L^2}\right)^{\lfloor L^2t\rfloor }.	
	\end{equation}
Observe that for every $j\geq 0$, the following convergence holds uniformly on $\mathbb R$:
\begin{equation}
\lim_{L,N\to \infty} \frac{\partial^j\,u_{L,N}^{(n)}(x)}{\partial x^j}=-2\pi^2 n\frac{\partial^j  u_\alpha(x)}{\partial x^j},\quad\text{with }\;u_\alpha(x)=\frac{\sin\left(\pi\alpha +x\right)}{\sin(\pi\alpha)}.
\end{equation}
Furthermore, one checks that
\begin{equation}
\left|\frac{\partial^j }{\partial z^j}\left(1+\frac{z}{L^2}\right)^{\lfloor L^2 t\rfloor}- t^j e^{tz}\right|=\mathcal O\left(\frac{1}{L^2}\right),
\end{equation}
uniformly for $0\leq t\leq T$, $0 \leq j\leq d$, and $z\in K\subset \mathbb C$ a compact set.
	
	Hence, it follows from the Faà di Bruno formula that
	\begin{equation}
	\lim_{L,N\to\infty} \frac{\partial^j h_{t,L,N}^{(n)}(x)}{\partial x^j}=\frac{\partial^j h^n_t(x)}{\partial x^j},
	\end{equation}
	uniformly on $\mathbb R$ and uniformly for $0\leq t\leq T$ and $0\leq j\leq d$. Applying the Taylor's theorem, we deduce easily \eqref{eq:expansiorhoschurhook} and  \eqref{eq:deriventh0} with
	\begin{equation}
	a_{n,j}^{L,N}(t)=\left.\frac{\partial^j h_{t,L,N}^{(n)}(x)}{\partial x^j}\right|_{x=0}.
		\end{equation}
\end{proof}

\begin{proof}[Proof of Lemma \ref{lem:eigendoublehookdevasym}]
	The argument parallels that of Lemma \ref{lem:eigenhookdevasym}. We use the double--Schur eigenvalue formula \eqref{eq:doublehookeigenvalueprecise} and and apply the multidimensional Taylor theorem to
	\begin{equation}
	h_{t,L,N}^{(n)}(x,y)=\left(1-\frac{2\sin\left(\frac{\pi}{L}\right)\sin\left(\frac{\pi n}{L}\right)\left(\sin\left(\frac{\pi N}{L}+x\right)+\sin\left(\frac{\pi N}{L}+y\right)\right)}{\sin\left(\frac{\pi N}{L}\right)}\right)^{\lfloor L^2t\rfloor}.
	\end{equation}
\end{proof}

\begin{proof}[Proof of Lemma \ref{lem:finallemvar}]
	We shall see that \eqref{eq:techlem1} follows from Lemma~\ref{lem:eigendoublehookdevasym} together with  Lemma~\ref{lem:characteridentities}.

Similarly to \eqref{eq:avoidmistake}, the left-hand side of~\eqref{eq:techlem1} can be rewritten as
		
		\begin{equation}\label{eq:varpart}
		\mathbf V=\frac{1}{N^2}\sum_{\pi,\mu \vdash n}\frac{1}{z_\pi z_\mu} \sum_{k=0}^{n-1} \sum_{l=0}^{n-1} (-1)^{k+l} \, \br_{\{n|k,l\}}^{\lfloor L^2 t \rfloor} \, \chi^{\{n|k\}}_\pi\chi^{\{n|l\}}_\mu {p_\pi \overline{p_\mu}}.
		\end{equation}
		It then follows  that
		\begin{align}
		\mathbf V &= \sum_{\pi,\mu \vdash n}\frac{1}{z_\pi z_\mu} \sum_{j_1+j_2\leq d_{\pi,\mu}} \frac{a_{n,j_1,j_2}^{L,N}(t)\pi^{j_1}\pi^{j_2}}{j_1!j_2!}
		\; \boxed{\mathfrak X_{\pi}^{(j_1)} \mathfrak X_{\mu}^{(j_2)}}\;
		\frac{p_\pi}{N L^{j_1}}\frac{\overline{p_\mu}}{N L^{j_2}} +\mathcal O\left(\frac{1}{L}\right)\label{eq:mult0}\\
		&=\sum_{\pi,\mu \vdash n}\frac{1}{z_\pi z_\mu} \sum_{j_1+j_2\leq d_{\pi,\mu}} \frac{a_{n,j_1,j_2}^{L,N}(t)\pi^{j_1}\pi^{j_2}}{j_1!j_2!} \mathfrak X_{\pi}^{(j_1)} \mathfrak X_{\mu}^{(j_2)}\frac{p_\pi}{N L^{j_1}}\frac{\overline{p_\mu}}{N L^{j_2}} +\mathcal O\left(\frac{1}{L}\right) \\
		&= \sum_{\pi,\mu \vdash n}\frac{1}{z_\pi z_\mu} \frac{a_{n,\ell_\pi^-,\ell_\mu^-}^{L,N}(t)\pi^{\ell_\pi^-}\pi^{\ell_\mu^-}}{\ell_\pi^-!\ell_\mu^-!} \mathfrak X_{\pi}^{(\ell_\pi^-)} \mathfrak X_{\mu}^{(\ell_\mu^-)}\frac{p_\pi}{N L^{\ell_\pi^-}}\frac{\overline{p_\mu}}{N L^{\ell_\mu^-}} +\mathcal O\left(\frac{1}{L}\right),\label{eq:mult}
		\end{align}
where $\ell_\pi^-:=\ell(\pi)-1$, $\ell_\mu^-:=\ell(\mu)-1$, and $d_{\pi,\mu}:=\ell_\pi^-+\ell_\mu^-$. Above,  we use  \eqref{eq:boundppi} and therefore 
\begin{equation}
p_\pi\overline{p_\mu}=\mathcal O(N^2 L^{d_{\pi,\mu}}),
\end{equation}
together with $\mathfrak X_{\pi}^{(j_1)}\mathfrak X_{\mu}^{(j_2)}=0$ whenever $j_1<\ell_\pi^-$ or $j_2<\ell_\mu^-$.

We refer to Lemma~\ref{lem:characteridentities} for the expression of each term in this product, boxed in~\eqref{eq:mult0} for ease of reference and to facilitate the proof of~\eqref{eq:techlem2}. The convergence of~\eqref{eq:mult} to the desired limit~\eqref{eq:techlem1} then follows directly from~\eqref{eq:deriventhdouble} together with~\eqref{eq:momentasymptotic}.

 \medskip

To show \eqref{eq:techlem2}, we use the skew hook Schur decompositions~\eqref{eq:skewhooklittle1} and~\eqref{eq:skewhooklittle2}, yielding
	\begin{multline}\label{eq:fourdecompo}
	s_{\{n|k\} \setminus \{m|i\}}\, \overline{s_{\{n|l\} \setminus \{m|i\}^\prime}}= \underbrace{s_{\{n-m|k-i\}}\overline{s_{\{n-m|l-(m-i-1)\}}}}_{[0,0]}+ \underbrace{s_{\{n-m|k-i-\mathbf{1}\}}\overline{s_{\{n-m|l-(m-i-1)\}}}}_{[\mathbf{1},0]}\\
	+\underbrace{s_{\{n-m|k-i\}}\overline{s_{\{n-m|l-(m-i-1)-\mathbf{1}\}}}}_{[0,\mathbf{1}]}+ \underbrace{s_{\{n-m|k-i-\mathbf{1}\}}\overline{s_{\{n-m|l-(m-i-1)-\mathbf{1}\}}}}_{[\mathbf{1},\mathbf{1}]}.
	\end{multline}
Plugging this decomposition into~\eqref{eq:techlem2}, the sum can be rewritten as
\begin{equation}\label{eq:Cbf}
\mathbf C=\sum_{m=1}^n\sum_{i=0}^{m-1}\sum_{a,b\in\{0,1\}}\mathbf C_{[a,b]}^{\{m|i\}},
\end{equation}
where the subscript $[a,b]$ refers in a natural way to the individual terms in \eqref{eq:fourdecompo}. Each of these terms further decomposes in a manner analogous to~\eqref{eq:mult0}, namely:\begin{equation}\label{eq:boxed}
\mathbf C_{[a,b]}^{\{m|i\}}= \sum_{\pi,\mu \vdash n-m}\frac{1}{z_\pi z_\mu} \sum_{j_1+j_2\leq d_{\pi,\mu}} \frac{a_{n,j_1,j_2}^{L,N}(t)\pi^{j_1}\pi^{j_2}}{j_1!j_2!}
\;\boxed{{\mathcal X}_{\pi,a}^{(j_1,i)} \widetilde{\mathcal X}_{\mu,b}^{(j_2,i)}}\;
\frac{p_\pi}{N L^{j_1}}\frac{\overline{p_\mu}}{N L^{j_2}} +\mathcal O\left(\frac{1}{L}\right).
\end{equation}

However, due to the presence of shifts, the boxed terms in~\eqref{eq:boxed} are more delicate to handle than those in~\eqref{eq:mult0}. Nevertheless, their explicit form can still be determined using Lemma~\ref{lem:characteridentities} and the techniques developed in its proof. 

Indeed, one checks that for $a\in\{0,1\}$ one has
 	\begin{align}
{\mathcal X}_{\pi,a}^{(j_1,i)} 
&= \sum_{k=0}^{n-1} (-1)^k \chi_\pi^{\{n-m|k-i-a\}} (n-2k-1)^{j_1}\\
	&= \sum_{k=i+a}^{n-m+i-1+a} (-1)^k \chi_\pi^{\{n-m|k-i-a\}} (n-2k-1)^{j_1}\label{eq:jaiplusdidees}\\
	& = (-1)^{i+a}\sum_{k=0}^{n-m-1} (-1)^{k} \chi_\pi^{\{n-m|k\}} \big((n-m-2k-1) + (m-2(i+a))\big)^{j_1}\\
	& = (-1)^{i+a}\sum_{i_1=0}^{j_1}\binom{j_1}{i_1}(m-2(i+a))^{j_1-i_1} \,
	\mathfrak X_{\pi}^{(i_1)}\\
	& = \begin{cases}
	 0, & \text{if $j_1<\ell_\pi^-$,}\\[5pt]
	 (-1)^{i+a} \,\mathfrak X_{\pi}^{(\ell_\pi^-)}, & \text{if $j_1=\ell_\pi^-$}. 
	\end{cases} 
	\end{align}
Similarly, one has for $b\in\{0,1\}$, 
\begin{align}
\widetilde{\mathcal X}_{\mu,b}^{(j_2,i)}
& = \sum_{l=m-i-1+b}^{n-i-2+b} (-1)^l \chi_\mu^{\{n-m|l-(m-i-1)-b\}} (n-2l-1)^{j_2}\label{eq:jaiplusdideeencore}\\
& =\begin{cases}
 0, & \text{if $j_2<\ell_\mu^-$,}\\[5pt]
(-1)^{m-i-1+b} \,\mathfrak X_{\mu}^{(\ell_\mu^-)} , & \text{if $j_2=\ell_\mu^-$}. 
\end{cases} 
\end{align}

Then, the key observation is that ${\mathcal X}_{\pi,1}^{(j_1,i)}=-{\mathcal X}_{\pi,0}^{(j_1,i)}$ and $\widetilde{\mathcal X}_{\mu,1}^{(j_2,i)}=-\widetilde{\mathcal X}_{\mu,0}^{(j_2,i)}$, which  imply
\begin{equation}
\sum_{a,b\in\{0,1\}}{\mathcal X}_{\pi,a}^{(j_1,i)} \widetilde{\mathcal X}_{\mu,b}^{(j_2,i)} = 0.
\end{equation}
We deduce that $\mathbf C$ in~\eqref{eq:Cbf}, and therefore the corresponding summand in~\eqref{eq:techlem2}, is of order $\mathcal O(1/L)$, which completes the proof of Lemma~\ref{lem:finallemvar}.

\begin{rem}
	For completeness, we note a minor boundary effect. When $m=n$, the upper index in the sums  in~\eqref{eq:jaiplusdidees} and ~\eqref{eq:jaiplusdideeencore} become smaller than the lower one. This situation can  be handled separately in~\eqref{eq:techlem2} by observing that $s_{\emptyset}\equiv 1$,  yielding again contributions of order $\mathcal O(1/L)$.
\end{rem}
\end{proof}

\subsection{Proofs of Theorems \ref{thm:hydro} and \ref{prop:regularityPDE}}

\label{sec:PDEid}

At this stage, assuming the initial empirical measure converges in probability to a random probability measure with density $f_0(dx)$, the $n$th complex moment of $\upsilon_{L,N}(t,dx)$ converges in probability to
\begin{equation}\label{eq:momentasymptoticsuite}
\mathfrak m_n(t) = \sum_{\pi\vdash n}
\frac{(2\pi\alpha)^{\ell(\pi)-1}}{\prod_{i}l_i!}\,
\left.\frac{\partial^{\ell(\pi)-1} h^n_t(z)}{\partial z^{\ell(\pi)-1}}
\right|_{z=0} \,m_\pi.
\end{equation}
We recall that
\begin{equation}\label{eq:defhetu}
h_t^n(z)=e^{-2\pi^2 n t \, u_\alpha(z)},\quad\text{with}\quad u_\alpha(z)=\frac{\sin(\pi\alpha + z)}{\sin(\pi\alpha)}.
\end{equation}

It remains to identify the limiting PDEs satisfied by the corresponding density and its moment generating function, and to establish their regularity and stationarity properties.

\subsubsection{Identification of the PDEs}

\label{sec:pdet0}

For convenience,  introduce the following analyticity assumption on $f_0$, formulated in terms of the generating function $g_0$ of its complex moments.

\begin{ass}\label{ass:smooth}
	The generating function $g_0$ has a radius of convergence $r_0>1$.
\end{ass}

\begin{prop}\label{eq:pdeinit}
	Under Assumption~\ref{ass:smooth}, the PDE \eqref{eq:derivative0gbis000000} satisfied by $g(t,z)$ holds at $t=0$ for $z$ in a neighbourhood of the unit disk $\mathbb D$, while the PDE \eqref{PDE} satisfied by the density $f(t,x)$ holds in the strong sense at $t=0$.
\end{prop}

\begin{proof}[Proof of Proposition \ref{eq:pdeinit}]
 We begin with investigate the regularity of the moment generating function $g(t,z)$ and show that it satisfies the PDE \eqref{eq:derivative0gbis000000} at $t=0$ from which \eqref{PDE} follows.

\begin{lem}\label{lem:edptransportt0}
	Assume that the power series $\sum_{n=1}^\infty\mathfrak m_n(t) z^n$ and $\sum_{n=1}^\infty {\partial_t \mathfrak m_n(t)} z^n$ are normally convergent on some neighbourhood $[0,\epsilon)\times B(0,\delta)$. Then, for all $z\in B(0,\delta)$, one has  
	\begin{equation}\label{eq:derivative0g}
	\left.\frac{\partial g(t,z)}{\partial t}\right|_{t=0}=
	-{2\pi^2   u_\alpha(2\pi\alpha\, g_0(z)) z g_0^\prime(z)}. 
	\end{equation}
\end{lem}

\begin{proof}[Proof of Lemma \ref{lem:edptransportt0}] 
	
	To begin with, one can rewrite \eqref{eq:momentasymptoticsuite} in the form
	\begin{equation}\label{eq:momentasymptoticsuitefaadi}
	\mathfrak m_n(t) = \sum_{k=1}^n \frac{B_{n,k}(1! m_1,2!m_2,\cdots )}{n!} \left.\frac{\partial^{k-1} h^n_t(z)}{\partial z^{k-1}} \right|_{z=0} (2\pi\alpha)^{k-1},
	\end{equation}
	where $B_{n,k}(x_1,\cdots, x_{n-k+1})$ denote the partial (or incomplete) exponential Bell polynomials. Thus, applying the Faà di Bruno formula, we obtain
	\begin{equation}\label{eq:momentasymptoticsuitefaadi2}
	\mathfrak m_n(t) = \frac{1}{2\pi\alpha  \,n!}\left.\frac{\partial^{n} H_{t,n}(2\pi\alpha \, g_0(z))}{\partial z^{n}}\right|_{z=0}
	= \frac{1}{n!}\left.\frac{\partial^{n-1}g_0^\prime(z)\, h_t^n(2\pi\alpha \, g_0(z))}{\partial z^{n-1}}\right|_{z=0},
	\end{equation}
	where $H_{t,n}$ denotes a primitive of $h_t^n$. 
	
	Moreover, one can also write
	\begin{equation}\label{eq:momentasymptoticderive}
	\frac{\partial \mathfrak m_n(t)}{\partial t} = -2\pi^2n \sum_{k=1}^n \frac{B_{n,k}(1! m_1,2!m_2,\cdots )}{n!} \left.\frac{\partial^{k-1} u_\alpha (z)h^n_t(z)}{\partial z^{k-1}} \right|_{z=0} (2\pi\alpha)^{k-1}.
	\end{equation}
Applying the Faà di Bruno formula once again, we similarly obtain
	\begin{equation}\label{eq:momentasymptoticderiveen00}
	\frac{\partial \mathfrak m_n(t)}{\partial t} = -2\pi^2\frac{1}{(n-1)!}\left.\frac{\partial^{n-1}g_0^\prime(z)  u_\alpha(2\pi\alpha \, g_0(z)) h_t^n(2\pi\alpha \, g_0(z))}{\partial z^{n-1}}\right|_{z=0}.
	\end{equation}
	
Thereafter, by interchanging differentiation and summation, one gets
\begin{equation}\label{eq:derivative0gexplain}
\left.\frac{\partial g(t,z)}{\partial t}\right|_{t=0}
=-2\pi^2 z \sum_{n=0}^{\infty}\frac{1}{n!}
\left.\frac{\partial^{n}\!\left(g_0^\prime(z)u_\alpha(2\pi\alpha g_0(z))\right)}{\partial z^{n}}\right|_{z=0} z^{n},
\end{equation}
from which \eqref{eq:derivative0g} follows directly.
\end{proof}

\begin{lem}\label{lem:cvu}
	Under Assumption \ref{ass:smooth}, there exist $\delta>1$ and $\varepsilon>0$ such that $\sum_{n=1}^\infty\mathfrak m_n(t) z^n$ and $\sum_{n=1}^\infty {\partial_t \mathfrak m_n(t)} z^n$  are normally convergent on $[0,\varepsilon)\times B(0,\delta)$.
\end{lem}

The proof of this lemma will be given below. It implies, together with Lemma \ref{lem:edptransportt0}, that under Assumption~\ref{ass:smooth} the moment generating function $g(t,z)$ satisfies \eqref{eq:derivative0g} in a neighbourhood of the unit disk $\mathbb D$. The corresponding PDE for the density  is then derived as follows. 

First note that $\mathfrak m_0(t)=1$ and that $\mathfrak m_{-n}(t)$ is the complex conjugate of $\mathfrak m_n(t)$. Thus
\begin{equation}\label{eq:density}
2\pi f(t,x)=\operatorname{Re}\left(1+2 g(t,e^{ix})\right)=\sum_{n\in\mathbb Z} \mathfrak m_n(t) e^{inx}.
\end{equation}
To go further, one can check that
\begin{equation}\label{eq:hilberimaginary}
2\pi \mathcal H f(t,x)=\operatorname{Im}\left(1+2 g(t,e^{ix})\right)=-i\sum_{n\in\mathbb Z} \mathrm{sgn}(n) 
\mathfrak m_n(t) e^{in x},
\end{equation}
and we get  from \eqref{eq:hilberimaginary} and \eqref{eq:density} that   
\begin{equation}\label{eq:hardy}
1+2g(t,e^{i x})=2\pi\left(f(t,x)+ i\, \mathcal Hf(t,x)\right).
\end{equation}

In addition, noting that $e^{ix}g_0^\prime(e^{ix})=-i\,\partial_x g_0(e^{ix})$, it follows from \eqref{eq:hardy} evaluated at $t=0$, the smoothness of $f_0$, and the standard properties of the Hilbert transform, that
\begin{equation}\label{eq:g}
e^{ix} g_0^\prime(e^{ix})= \pi \left(\mathcal H f_0^\prime(x)-i f_0^\prime(x)\right).
\end{equation}

Finally, combining \eqref{eq:g}, \eqref{eq:hardy} evaluated at $t=0$, and \eqref{eq:derivative0g}, 
we obtain
\begin{align}\label{eq:derivative0f}
\left.\frac{\partial f}{\partial t}\right|_{t=0}  = \frac{1}{\pi}\left.\frac{\partial \operatorname{Re}(g)}{\partial t}\right|_{t=0} & = -\frac{2\pi^2}{\sin(\pi\alpha)}\operatorname{Re}\left( \left(\mathcal H f_0^\prime -i f_0^\prime\right) \sin(2\pi^2\alpha\,(f_0+i\mathcal H f_0))\right)\\
& =-\frac{1}{\alpha\sin(\pi\alpha)}\frac{\partial \sin(2\pi^2\alpha \, f_0)\sinh(2\pi^2\alpha \,\mathcal H f_0)}{\partial x}. 
\end{align}
which completes the proof of Proposition~\ref{eq:pdeinit}.
\end{proof}

\begin{proof}[Proof of Lemma \ref{lem:cvu}]
The proof consists in expressing \eqref{eq:momentasymptoticsuitefaadi2} as the contour integral
\begin{equation}\label{eq:momentasymptoticsuitefaadi2contour}
\mathfrak m_n(t) = \frac{1}{2 i \pi n}\int_{0}^{2\pi}\frac{i g_0^\prime(re^{ix})\, h_t^n(2\pi\alpha \, g_0(r e^{ix}))}{r^{n-1}e^{i(n-1)x}}\,dx,\quad\text{with}\quad 0<r<r_0.
\end{equation}

Thereafter, a simple estimate based on the mean value inequality shows that, for all $1<\kappa<r_0$, there exists $c_\kappa>0$ such that, for all $r\in [1,\kappa]$, one has
\begin{equation}\label{eq:meanvalue}
\sup_{x\in \mathbb R}\left|\operatorname{Re}\sin(\pi\alpha\,(1+g_0(r e^{ix})))-\operatorname{Re}\sin(\pi\alpha\,(1+g_0(e^{ix})))\right|\leq \sin(\pi\alpha) c_\kappa (r-1).
\end{equation}

By using \eqref{eq:hardy} evaluated at $t=0$, together with the bounds  \eqref{eq:bounddesnity} on $f_0$, 
one has
\begin{equation}\label{eq:minorationsin}
\operatorname{Re}\sin(\pi\alpha\,(1+g_0(e^{ix})))=\sin(2\pi^2\alpha \,f_0(x))\cosh(2\pi^2\alpha\, \mathcal H f_0(x))\geq 0.
\end{equation}
In particular, it comes from \eqref{eq:minorationsin} and \eqref{eq:meanvalue}  that
\begin{equation}\label{eq:majht}
|h_t^n(2\pi\alpha g_0(r e^{ix}))|=e^{-2\pi^2 n t \, \operatorname{Re} u_\alpha(2\pi\alpha g_0(r e^{ix}))}\leq e^{{2\pi^2 c_\kappa  n t (r-1) }}.
\end{equation}

Now, let $\kappa$ sufficiently close to $1$ so that $\ln(r)\geq (r-1)/2$ holds for $1\leq r\leq \kappa$, and choose $\varepsilon>0$ such that $2\pi^2 c_\kappa \varepsilon < 1/2-\eta$ for some $0<\eta<1/2$. It follows from \eqref{eq:majht} that
\begin{equation}
\sup_{t\in[0,\varepsilon]}|\mathfrak m_n(t)|=\mathcal O\!\left(e^{-\eta(\kappa-1) n}\right).
\end{equation}

We deduce that $g(t,z)$ is normally convergent on $[0,\varepsilon)\times B(0,\delta)$ for some $\delta>1$. The proof that the generating function of the moment derivatives  is normally convergent on a similar domain follows exactly the same lines, which completes the proof
\end{proof}

\label{sec:markovmoment}

\begin{prop}\label{eq:pdeanyt}
	Let $t\ge 0$ be such that $z\mapsto g(t,z)$ has radius of convergence greater than $1$. Then there exist $\varepsilon>0$ and $\delta>1$ such that the PDE \eqref{eq:derivative0gbis000000} holds for $g(t,z)$ on $[t,t+\varepsilon)\times B(0,\delta)$, while the PDE \eqref{PDE} holds in the strong sense for the density $f(t,x)$ on $[t,t+\varepsilon)$.
\end{prop}


\begin{proof}[Proof of Proposition \ref{eq:pdeanyt}] For any $s \geq 0$, let $q$ denote the largest integer such that $q \leq L^2 s$, and define $s_\ast$ and $s^\ast$ by $q=L^2 s_\ast$ and $q+1=L^2 s^\ast$. Thereafter, set
\begin{equation}
\vartheta =
\begin{cases}
\Theta(L^2 s_\ast), & \text{if } s = s_\ast,\\[4pt]
\Theta(L^2 s^\ast), & \text{if } s > s_\ast.
\end{cases}
\end{equation}

Note that $s_\ast\leq s<s^\ast=s_\ast+1/L^2$. Moreover, since $|\Theta(L^2 s) - \vartheta| = \mathcal{O}(1/L)$, the empirical measure of $\vartheta$ converges in probability to $f(s,\theta)\,d\theta$, exactly as does the empirical measure of $\Theta(L^2 s)$. 

Let $t > s$ be fixed and observe that $L^2 t > L^2 s + 1$ for $L$ sufficiently large and thus $L^2 t>L^2 s^\ast$. Then, by the Markov property, the conditional law of $\Theta(L^2 t)$ given $\Theta(L^2 s)$ satisfies
\begin{equation}\label{eq:condi}
\mathcal{L}\big(\Theta(L^2 t)\,|\,\Theta(L^2 s)\big)=
\begin{cases}
\mathcal{L}\big(\Theta(L^2(t - s_\ast))\,|\,\Theta(0)=\vartheta\big), & \text{if } s=s_\ast,\\[5pt]
\mathcal{L}\big(\Theta(L^2(t - s^\ast))\,|\,\Theta(0)=\vartheta\big), & \text{if } s>s^\ast.
\end{cases}
\end{equation}

Since $|\Theta(L^2(t - u)) - \Theta(L^2(t - s))| = \mathcal{O}(1/L)$ for $u\in\{s_\ast,s^\ast\}$, it follows from \eqref{eq:condi}, together with the convergence in probability of the empirical measure of $\vartheta$ and of the latter results about the convergence of the empirical moments, that the empirical measure of $\Theta(L^2 t)$, conditionally on $\Theta(L^2 s)$, converges in probability to that of $\Theta(L^2 (t - s))$, with initial configuration distributed according to $f(s,\theta)\,d\theta$. Hence,  one has $f(t,\theta) = Q_t f_0(\theta)$ for some semigroup kernel $(Q_t)_{t \geq 0}$, and    
	\begin{equation}\label{eq:momentasymptoticbisbis}
	\forall s,t\geq 0,\quad \mathfrak m_n(s+t)=\sum_{\pi\vdash n}
	\frac{(2\pi\alpha)^{\ell(\pi)-1}}{\prod_{i}l_i!}\,
	\left.\frac{\partial^{\ell(\pi)-1} h^n_t(x)}{\partial x^{\ell(\pi)-1}}
	\right|_{x=0} \,\mathfrak m_\pi(s).
	\end{equation}
	
Using Proposition~\ref{eq:pdeinit}, for any time $t_0\ge 0$ such that $z\mapsto g(t_0,z)$ has radius of convergence greater than $1$, the PDE \eqref{eq:derivative0gbis000000} holds at $t=t_0$ for $z\in B(0,\delta)$, for some $\delta>1$, and the PDE \eqref{PDE} holds in the strong sense at $t=t_0$. Furthermore, Lemmas~\ref{lem:cvu} and~\ref{lem:edptransportt0}, together with their proofs, show that \eqref{eq:derivative0gbis000000} holds on $[t_0,t_0+\varepsilon)\times B(0,\delta)$ for some $\varepsilon>0$ and $\delta>1$, and hence that \eqref{PDE} holds on $[t_0,t_0+\varepsilon)$. This completes the proof of Proposition~\ref{eq:pdeanyt}.

\begin{rem}
	In general, even an entire initial density $f_0$ does not remain smooth for all $t\geq 0$. We refer to Section~\ref{sec:charc} and to the example detailed in Section~\ref{sec:periodicdens}.
\end{rem}
\end{proof}

\subsubsection{Characteristic representation of the solutions}
\label{sec:charc}

 The standard method of characteristics shows that any local smooth solution $g(t,z)$ remains constant along the trajectories $t\mapsto (t,Z_t)$ of the ODE
\begin{equation}\label{eq:ODE0}
{\partial}_t Z_t = A_0(w) Z_t, \quad Z_0 = w,\quad \text{with}\quad A_0(w)=2\pi^2 
\frac{\sin(\pi\alpha(1+2g_0(w)))}{\sin(\pi\alpha)}.
\end{equation}

The main interest of this observation is that it allows one to solve, whenever possible, the PDE \eqref{eq:derivative0gbis000000} by setting $g(t,z):=g_0(w(t,z))$, where $w(t,z)$  satisfies 
\begin{equation}\label{eq:flowimplicit}
\Phi_t(w(t,z))=z,\quad \text{with}\quad \Phi_t(w)=w e^{t A_0(w)}.
\end{equation}
Here $\Phi_t(w)$ denotes the flow of the ODE \eqref{eq:ODE0}. Note that necessarily  $w(0,z)=z$.

The goal of this section is to define the branch $w(t,z)$ for $z\in\mathbb D$, taking values in  $V_t\subset\mathbb D$ (see Lemma~\ref{lem:tech1}), and to establish the representation $g(t,z)=g_0(w(t,z))$ (see Proposition~\ref{prop:representationsolution}).

We start with a simple yet crucial observation: the flow $\Phi_t$ is a dilation inside  $\mathbb{D}$.

\begin{lem}\label{lem:dilationrealpart}
For all $w\in \mathbb D$, one has $\operatorname{Re}A_0(w)>0$.
\end{lem}

\begin{proof}[Proof of Lemma \ref{lem:dilationrealpart}]
This result can be proved by generalizing \eqref{eq:hardy} together with the lower bound \eqref{eq:minorationsin}, which can be viewed as the boundary version of this lemma. 

To this end, one checks that, for all $0\leq r<1$ and all $x\in\mathbb R$, one has
\begin{equation}\label{eq:poisson}
1 + 2g_0(re^{ix}) = 2 \pi\left(P_r f_0(x) + i\, P_r \mathcal{H}f_0(x)\right),
\end{equation}
where $(P_r)_{0\leq r<1}$ denotes the usual Poisson convolution kernels. Recall that they form a family of smooth mollifiers as $r\uparrow 1$ and act on Fourier series according to
\begin{equation}
P_r\left( \sum_{n \in \mathbb{Z}} c_n e^{inx} \right) = \sum_{n \in \mathbb{Z}} c_n\, r^{|n|} e^{inx}.
\end{equation}

In particular, there exists $\eta_r>0$ such that the smooth density $h_r:=P_r  f_0$ satisfies 
	\begin{equation}
	\eta_r \leq \min_{x\in\mathbb R/2\pi\mathbb Z}h_r(x)\leq \max_{x\in\mathbb R/2\pi\mathbb Z}h_r(x)\leq \frac{1}{2\pi\alpha}-\eta_r.
	\end{equation}
	Straightforward computations then yield
	\begin{equation}\label{eq:minorationsinsuite}
	\operatorname{Re} \sin(\pi\alpha (1+2g_0(re^{ix})))
	= \sin(2\pi^2\alpha h_r(x)) \cosh(2\pi^2\alpha \mathcal{H}h_r(x))
	\geq \sin(2\pi^2\alpha \eta_r)>0,
	\end{equation}
	from which Lemma~\ref{lem:dilationrealpart} follows by taking $w=re^{ix}$.
\end{proof}

\begin{rem}\label{rem:herglotz}
It follows from  \eqref{eq:poisson} and \eqref{eq:minorationsinsuite} that the functions $1+2g_0(w)$ and $A_0(w)$ belong to the Carathéodory--Herglotz class $\mathcal{HC}$, i.e., holomorphic functions on $\mathbb D$ with positive real part. They admit non-tangential boundary limits and the Herglotz representation
	\begin{equation}
	\int_{0}^{2\pi}\frac{e^{ix}+w}{e^{ix}-w}\,\mu(dx),
	\end{equation}
	for some finite positive Borel measure $\mu$. Besides, since $1+2g_0(w)$ is in  the Hardy space $H^\infty$, the associated $\mu$ is absolutely continuous with bounded density. 
\end{rem}

\begin{lem}\label{lem:anlytic1bis}
	The function $g(t,z)$ is analytic on $[0,\infty)\times \mathbb D$ and satisfies \eqref{eq:derivative0gbis000000} on this domain.
\end{lem}

\begin{proof}[Proof of Lemma \ref{lem:anlytic1bis}]
	First, we extend \eqref{eq:momentasymptoticderiveen00} and write
	\begin{equation}\label{eq:momentasymptoticderiveen000}
	\frac{\partial^{k} \mathfrak m_n(t)}{\partial t^k} = \frac{(-2\pi^2)^k\, n^{k-1}}{(n-1)!}\left.\frac{\partial^{n-1}g_0^\prime(z)\, u_\alpha^k(2\pi\alpha \, g_0(z))\, h_t^n(2\pi\alpha \, g_0(z))}{\partial z^{n-1}}\right|_{z=0}.
	\end{equation}
In particular, for any $0<r<1$, one has 
	\begin{equation}\label{eq:momentasymptoticsuitefaadi2contour1}
	\frac{\partial^{k} \mathfrak m_n(t)}{\partial t^k}
	= \frac{(-2\pi^2)^k\, n^{k-1}}{2 i \pi }\int_{0}^{2\pi}\frac{i g_0^\prime(re^{ix})\, u_\alpha^k(2\pi\alpha \, g_0(re^{ix}))\, h_t^n(2\pi\alpha \, g_0(r e^{ix}))}{r^{n-1}e^{i(n-1)x}}\,dx.
	\end{equation}
To proceed further, Lemma~\ref{lem:dilationrealpart} implies that, for all $w\in\mathbb D$,
\begin{equation}
|h_t^n(2\pi\alpha\, g_0(w))|=e^{-2\pi^2 n t\,\operatorname{Re} u_\alpha(2\pi\alpha\, g_0(w))}=e^{-nt \operatorname{Re}A_0(w)}\leq 1.
\end{equation}

Letting $r\to 1$ in \eqref{eq:momentasymptoticsuitefaadi2contour1}, we deduce that for any $\varepsilon>0$, there exist positive constants $C$ and $M$ such that, for all $n,k,t\ge 0$,
\begin{equation}\label{eq:momentasymptoticsuitefaadi2contour1bis}
\left|\frac{\partial^{k} \mathfrak m_n(t)}{\partial t^k}\right| \leq C\, n^k M^k e^{\varepsilon n}.
\end{equation}

Thereafter, let $\delta>0$ be such that $M\delta<\varepsilon$. It follows from \eqref{eq:momentasymptoticsuitefaadi2contour1bis} that the power series
\begin{equation}
\sum_{n=1}^{\infty} \sum_{k= 0}^\infty \frac{1}{k!}\left.\frac{\partial^{k} \mathfrak m_n(t)}{\partial t^k}\right|_{t=0} t^k z^n,
\end{equation}
is absolutely convergent for all $0\le t<\delta$ and $|z|<e^{-2\varepsilon}$. Since $\varepsilon$ is arbitrary small, we deduce that $g(t,z)$ is analytic in a neighbourhood of any point $(0,z)$ with $|z|<1$.

Arguments similar to those of Section~\ref{sec:pdet0} (see in particular Lemma~\ref{lem:edptransportt0}) show that the PDE \eqref{eq:derivative0g} holds at $t=0$ for $z\in\mathbb D$. Its extension to any time $t>0$ on $\mathbb D$ then follows from the Markov property (see \eqref{eq:momentasymptoticbisbis}), which concludes the proof.
\end{proof}

\begin{lem}\label{lem:tech1}
	For every $t\geq 0$ and $z\in \mathbb D$, there exists a unique $w(t,z)\in \mathbb D$ solving \eqref{eq:flowimplicit}. Moreover, $|w(t,z)|<|z|$ for all $t>0$, and the following properties hold:
	\begin{enumerate}
		\item For each $t\geq 0$, the map $w(t,\cdot)$ is a biholomorphism of $\mathbb D$ onto
		$V_t:=\Phi_t^{-1}(\mathbb D)\cap \mathbb D$.
		\item The family $(V_t)_{t\geq 0}$ is decreasing: for all $0\leq s\leq t$, one has $V_t\subset V_s$.
		\item The map $(t,z)\mapsto w(t,z)$ is analytic on $[0,\infty)\times \mathbb D$.
		\item One has $\Phi_t(\partial V_t)=\mathbb D$.
	\end{enumerate}
\end{lem}

\begin{proof}[Proof of Lemma \ref{lem:tech1}]
Using Lemma~\ref{lem:dilationrealpart}, one has for all $t>0$, $z\in \mathbb D$, and $w\in \mathbb D$ with $|w|\ge |z|$:
\begin{equation}
\left|(w e^{tA_0(w)}-z)-w e^{tA_0(w)}\right|<\left|w e^{tA_0(w)}\right|.
\end{equation}
Rouché’s Theorem (see the proof of Lemma~\ref{rouche}) then ensures the existence of a unique solution $w(t,z)\in \mathbb D$ to \eqref{eq:flowimplicit}, satisfying $|w(t,z)|<|z|$ for all $t>0$.

The restriction of $\Phi_t$ to $V_t$ thus defines a bijective analytic map from $V_t$ onto $\mathbb D$, and hence a biholomorphism. In particular, for all $w\in V_t$, one has
\begin{equation}\label{eq:implicitanal}
\Phi_t'(w) = (1+t w A_0'(w))e^{t A_0(w)} \neq 0.
\end{equation}
Since the map $(t,w)\mapsto \Phi_t(w)$ is analytic, this non-degeneracy condition allows us to apply the implicit function theorem, yielding the analyticity of $w(t,z)$.

The inclusion $\Phi_t(\partial V_t)\subset \partial \mathbb D$ is classical and follows from the fact that $\Phi_t^{-1}(K)$ is compact for every compact set $K\subset \mathbb D$. Indeed, if $\Phi_t(w)\in \mathbb D$ for some $w\in \partial V_t$, choose a sequence $(w_n)\subset V_t$ such that $w_n\to w$, and set $K=\{\Phi_t(w_n):n\geq 0\}\cup\{\Phi_t(w)\}$. Then $K$ is compact in $\mathbb D$, and consequently $\Phi_t^{-1}(K)\subset V_t$ is compact as well, contradicting the fact that $w\in \partial V_t$.

Now let $z\in \partial \mathbb D$ and choose a sequence $(z_n)\subset \mathbb D$ such that $z_n\to z$. For each $n$, pick $w_n\in V_t$ with $\Phi_t(w_n)=z_n$. By relative compactness of $V_t$, there exists a subsequence, still denoted $(w_n)$, converging to some $w\in \overline{V_t}$. By continuity of $\Phi_t$, one has $\Phi_t(w)=z$. Since $z\in \partial \mathbb D$, it follows that $w\in \partial V_t$, proving that $\Phi_t$ maps $\partial V_t$ onto $\partial \mathbb D$.

It remains to prove that the family $(V_t)_{t\geq 0}$ is decreasing. Let $w_0\in \mathbb D\cap V_s^c$. By definition of $V_s$, one has $|\Phi_s(w_0)|\geq 1$, and Lemma~\ref{lem:dilationrealpart} implies that $|\Phi_t(w_0)|\geq 1$ for all $t\geq s$. In other words, one has $w_0\in \mathbb D\cap V_t^c$ for all $t\geq s$, showing that $(V_t)_{t\geq 0}$ is decreasing and completing the proof.
\end{proof}

At this stage, for an arbitrary initial density profile, one verifies that $(t,z)\mapsto g_0(w(t,z))$ solves the PDE \eqref{eq:derivative0gbis000000} on $[0,\infty)\times \mathbb D$. It suffices to observe that
\begin{equation}\label{eq:preedp0}
\partial_t w(t,z) = -\frac{w(t,z)\, A_0(w(t,z))}{1+t\, w(t,z)\, A_0'(w(t,z))}\quad\text{and}\quad
\partial_z w(t,z) = \frac{e^{-t A_0(w(t,z))}}{1+t\, w(t,z)\, A_0'(w(t,z))},
\end{equation}
which imply
\begin{equation}\label{eq:lowner}
\partial_t w(t,z)
= -A_0(w(t,z))\, z\, \partial_z w(t,z)
= -2\pi^2 u_\alpha(2\pi\alpha\, g_0(w(t,z)))\, z\, \partial_z w(t,z).
\end{equation}
The identity \eqref{eq:derivative0gbis000000} then follows.

\begin{rem} \label{rem:loewner}
	The transport equation \eqref{eq:lowner} is closely related to Loewner--Kufarev type PDEs. From this perspective, the family $(w(t,\cdot))_{t\geq 0}$ can be interpreted as a non-increasing Loewner-type chain on the unit disk. For recent references, see \cite{Bracci,Contreras}.
\end{rem}

Moreover, since $u_\alpha$ is entire, the Cauchy--Kovalevski theorem (see \cite{Petrov}) ensures that \eqref{eq:derivative0gbis000000} admits a unique local analytic solution $(t,z)\mapsto g(t,z)$ in a neighbourhood of any  $(t_0,w_0)$ such that $z\mapsto g(t_0,z)$ is analytic at $w_0$. In particular, Lemma~\ref{lem:anlytic1bis} yields the following proposition.

\begin{prop}\label{prop:representationsolution}
	For all $t\geq 0$ and all $z\in \mathbb D$, one has $g(t,z)=g_0(w(t,z))$.

\end{prop}

\subsubsection{Regularity and asymptotic properties}

\label{sec:regularprop}

Representing the density $f(t,x)$ via the method of characteristics requires extending the conformal map $z\mapsto w(t,z)$ from $\mathbb D$ to $V_t$ continuously to the unit circle $\mathbb U=\partial\mathbb D$, and determining when this extension admits an analytic or smooth continuation across the boundary.

These conditions depend on the regularity of $\partial V_t$. We refer to the continuity theorem and Carathéodory’s theorem in \cite[Chap.~2]{Pommerenke}, which ensure a continuous extension or a homeomorphism between the boundaries under the assumption that $\partial V_t$ is locally connected or a Jordan curve. See also \cite[Theorem~4, Chap.~6]{Ahlfors1966} and \cite[Chap.~3]{Pommerenke}, where local smoothness of $\partial V_t$ guarantees a local analytic extension across the boundary.

\begin{lem}\label{lem:regulaboundary}
	For all $t>0$, the set $\partial V_t\cap \mathbb D$ is an analytic curve. Moreover, one has
	\begin{equation}\label{eq:Z0}
	\partial V_t\cap \mathbb U\subset \mathcal W_0=\left\{w_0\in\mathbb U : \liminf_{w\to w_0} \operatorname{Re} A_0(w)=0\right\}.
	\end{equation}
\end{lem}

\begin{proof}[Proof of Lemma \ref{lem:regulaboundary}]
Consider $w_0\in \partial V_t\cap \mathbb D$ and set $z_0=\Phi_t(w_0)\in\mathbb U$. We claim that $w_0$ is not a critical point of $\Phi_t$, that is, $\Phi_t'(w_0)\neq 0$. Otherwise, by standard results in complex analysis, the equation $\Phi_t(w)=z$ would admit more than one solution $w\in \mathbb D$ in a neighbourhood of $w_0$ for any $z\in \mathbb D$ in some neighbourhood of $z_0$, contradicting Lemma~\ref{lem:tech1}. It then follows from the implicit function theorem that $\partial V_t$ is an analytic curve in a neighbourhood of $w_0$.
	
	Finally, assume that $w_0\in \partial V_t\cap \mathbb U$ and let $(w_n)_{n\geq 1}\subset V_t$ be a sequence such that $w_n\to w_0$. Since $|w_n|\to 1$, $w_n e^{t A_0(w_n)}\in \mathbb D$ and $\operatorname{Re} A(w_n)>0$ (see Lemma~\ref{lem:dilationrealpart}), one necessarily has $\operatorname{Re} A(w_n)\to 0$, which completes the proof of the lemma.
\end{proof}

\begin{rem}\label{rem:extremal}
	Assume that $A_0$ is continuous at some $w_0\in \overline{\mathbb D}$. Using \eqref{eq:minorationsin} and writing $w_0=e^{ix_0}$, one sees that the condition $\operatorname{Re} A_0(w_0)=0$ is equivalent to $f_0(x_0)$ attaining an extremal value, namely
	\begin{equation}
	\operatorname{Re} A_0(w_0)=0\quad \Longleftrightarrow \quad f_0(x_0)=0\quad\text{or}\quad f_0(x_0)=\frac{1}{2\pi\alpha}.
	\end{equation}
	We set
	\begin{equation}\label{eq:z0}
	\mathcal Z_0=\left\{x\in\mathbb R/2\pi\mathbb Z : f_0(x)=0\quad \text{or}\quad f_0(x)=\frac{1}{2\pi\alpha}\right\}.
	\end{equation}
\end{rem}

In view of Lemma~\ref{lem:regulaboundary} and Remark~\ref{rem:extremal}, it appears that the regularity properties  $f(t,x)$ are closely related to the set $\mathcal Z_0$.


\begin{prop}\label{prop:anal}
	Define
	\begin{equation}\label{eq:It}
	\mathcal I_t =\{x\in\mathbb R/2\pi\mathbb Z : e^{ix}\in \Phi_t(\partial V_t\cap \mathbb D)\}\quad\text{and}\quad 
	\mathcal I =\bigsqcup_{t>0} \{t\}\times \mathcal I_t.
	\end{equation}
	Then the map $(t,x)\mapsto f(t,x)$ is analytic on the open set $\mathcal I\subset (0,\infty)\times \mathbb R/2\pi\mathbb Z$ and consequently satisfies \eqref{PDE} in the strong sense on this domain.
\end{prop}

\begin{proof}[Proof of Proposition \ref{prop:anal}]
	The proof follows from the references given above Lemma~\ref{lem:regulaboundary} and from the arguments in its proof. Recall that $\Phi_t'(w)\neq 0$ for all $w\in \partial V_t\cap \mathbb D$. The implicit function theorem  ensures that $w(t,z)$ extends to a one-to-one analytic function in a neighbourhood of $(t,e^{ix})$ for $x\in \mathcal I_t$, with values in $\mathbb D$, which imply that $f(t,x)$ si analytic around such point.
\end{proof}

As a consequence of Proposition~\ref{prop:anal} and Lemma~\ref{lem:regulaboundary}, if $\mathcal W_0$ is empty -- for instance, if $f_0$ satisfies the bounds \eqref{eq:upperlowerf0} -- then the map $(t,x)\mapsto f(t,x)$ is analytic on $(0,\infty)\times \mathbb R/2\pi\mathbb Z$ and satisfies \eqref{PDE} in the strong sense on this domain. This can also be proved by standard estimates, as in Lemmas~\ref{lem:cvu}, showing that $\mathfrak m_n(t)$ decays geometrically as $n\to\infty$.

There is always a spontaneous transition to smoothness and convergence to the equilibrium.

\begin{prop}\label{prop:spontaneous}
	For any initial density profile, there exists $t^\ast\geq 0$ such that $(t,x)\mapsto f(t,x)$ is analytic on $(t^\ast,\infty)\times \mathbb R/2\pi\mathbb Z$. Besides, there exists $C,\lambda>0$ such that for all $t\geq 0$, 
	\begin{equation}\label{eq:equilibrium}
	\sup_{x\in\mathbb R}\left|f(t,x)-\frac{1}{2\pi}\right|
	 \leq C e^{-\lambda t}.
	\end{equation}
\end{prop}

\begin{proof}[Proof of Proposition \ref{prop:spontaneous}]
We begin by estimating the derivative term in \eqref{eq:momentasymptoticsuitefaadi} via the following Cauchy integral representation, with $0<r<1$:
	\begin{equation}\label{eq:newcontour}
	\left.\frac{\partial^{k-1} h^n_t(z)}{\partial z^{k-1}} \right|_{z=0}
	= \frac{(k-1)!}{2i\pi} \int_{0}^{2\pi} \frac{ i e^{-2\pi^2 n t\, u_\alpha(r e^{ix})}}{r^{k-1}e^{i(k-1)x}}\, dx.
	\end{equation}
	
	Since $u_\alpha(0) = 1$, one can choose $r$, arbitrarily small, such that $|e^{-2\pi^2 n t\, u_\alpha(r e^{ix})}| \leq e^{-\pi^2 n t}$ uniformly in $x$, $n$, and $t$. In particular, we obtain for such $r$:
	\begin{equation}
	\left|\left.\frac{\partial^{k-1} h^n_t(z)}{\partial z^{k-1}} \right|_{z=0}\right| \leq \frac{(k-1)!\, e^{-\pi^2 n t}}{r^{k-1}}.
	\end{equation}
Next, regarding the exponential partial Bell polynomials in \eqref{eq:momentasymptoticsuitefaadi}, one can write
	\begin{equation}
	|B_{n,k}(1!m_1,2!m_2,\cdots)| \leq B_{n,k}(1!|m_1|,2!|m_2|,\cdots).
	\end{equation}

Then, using  $B_{n,k}(x_1 a^1,x_2 a^2,\cdots) = B_{n,k}(x_1,x_2,\cdots)a^n$ and the Faà di Bruno formula, we get
	\begin{align}
	|\mathfrak m_n(t)| 
	&\leq \frac{1}{n!} \sum_{k=1}^n B_{n,k}\big(1!|m_1|e^{-\pi^2t\times 1},2!|m_2|e^{-\pi^2t\times 2},\cdots\big) (k-1)!\frac{(2\pi\alpha)^{k-1}}{r^{k-1}} \nonumber\\
	&\leq \frac{1}{n!} 
	\left.\frac{\partial^{n} S_r \circ g_0^+(z e^{-\pi^2 t})}{\partial z^n} 
	\right|_{z=0},
	\end{align}
	where we set
	\begin{equation}
	S_r(z) = \sum_{k=1}^\infty \frac{1}{k}\frac{(2\pi\alpha)^{k-1}}{r^{k-1}}z^k = -\frac{r}{2\pi\alpha}\ln\left(1-\frac{2\pi\alpha z}{r}\right)
	\quad\text{and}\quad 
	g_0^+(z) = \sum_{n=1}^\infty |m_n| z^n.
	\end{equation}
One then obtains the identity
	\begin{equation}\label{eq:powerseriesmnradius}
	\sum_{n=1}^\infty |\mathfrak m_n(t)|\, |z|^n \leq  -\frac{r}{2\pi\alpha}\ln\left(1-\frac{2\pi\alpha\, g_0^+(|z|e^{-\pi^2 t})}{r}\right).
	\end{equation}

Let $R>1$ be arbitrary, and choose $t^\ast>0$ such that $2\pi\alpha\, g_0^+(R e^{-\pi^2 t^\ast})< r$. 
Then the power series in \eqref{eq:powerseriesmnradius} is absolutely convergent for all $t\ge t^\ast$ and $|z|\le R$. It follows from  Section~\ref{sec:pdet0} that $f(t,x)$ is smooth on $[t^\ast,\infty)\times\mathbb R$ and satisfies the PDE~\eqref{PDE} in the strong sense on this domain.
	
It remains to study the convergence of $f(t,x)$ as $t \to \infty$. 
	By using the upper bound in \eqref{eq:powerseriesmnradius} and noting that $g_0^+(0)=0$, we obtain 
	\begin{equation}
	\lim_{t\to\infty} \sum_{n=1}^\infty |\mathfrak m_n(t)| = 0.
	\end{equation}
	Since $\mathfrak m_0(t)=1$ and $\mathfrak m_{-n}(t)$ is the complexe conjugate of ${\mathfrak m_{n}(t)}$, we deduce \eqref{eq:equilibrium} from
	\begin{equation}
\sup_{x\in\mathbb R}\left|f(t,x)-\frac{1}{2\pi}\right|
	\leq \frac{1}{\pi} \sum_{n=1}^\infty |\mathfrak m_n(t)| = \mathcal O\left(g_0^{+}(e^{-\pi^2 t})\right).
	\end{equation}
\end{proof}

To proceed, recall that the disk algebra $\mathcal A(\mathbb D)$ consists of analytic functions on $\mathbb D$ that extend continuously to $\overline{\mathbb D}$. The next result is a consequence of the correspondence between the regularity of the boundary $\partial V_t$ and that of $f(t,x)$. Its proof is left to the reader.

\begin{prop}\label{prop:cont}
	Assume that $A_0\in \mathcal A(\mathbb D)$. If  $\mathcal Z_0$ 
	is finite, then $(t,x)\mapsto f(t,x)$ is continuous on $[0,\infty)\times \mathbb R/2\pi\mathbb Z$. Moreover, for all $t>0$, the map $x\mapsto f(t,x)$ is analytic, except possibly on 
	\begin{equation}\label{eq:Ic}
	\mathcal I_t^c =\{x\in\mathbb R/2\pi\mathbb Z : e^{ix}\in \Phi_t(\partial V_t\cap  \mathbb U)\}= \mathcal Z_t,
	\end{equation}
	where
	\begin{equation}
	\mathcal Z_t=\left\{x\in\mathbb R/2\pi\mathbb Z : f(t,x)=0\quad \text{or}\quad f(t,x)=\frac{1}{2\pi\alpha}\right\}.
	\end{equation}
\end{prop}

Under the assumptions of Proposition~\ref{prop:cont}, the monotonicity of $(V_t)_{t\geq 0}$ yields $|\mathcal Z_t|\leq |\mathcal Z_0|$ for all $t\geq 0$, together with
\begin{equation}\label{eq:carac0}
x\in\mathcal Z_t
\quad\Longleftrightarrow\quad
\exists\,x_0\in \mathcal Z_0 \; \text{such that} \; 
e^{ix_0}\in \partial V_{t}\; \text{and} \;  e^{ix}=e^{i\left(x_0 + t\,\operatorname{Im} A_0(e^{ix_0})\right)},
\end{equation}
and
\begin{equation}
f(t,x)=0 \;\left(\text{resp.} \; \frac{1}{2\pi\alpha}\right) \quad \Longrightarrow\quad  \forall\, 0\leq s\leq t,\; f(s,x)=0 \;\left(\text{resp.} \; \frac{1}{2\pi\alpha}\right).
\end{equation}

It is natural to ask under which conditions a zero $w_0=e^{ix_0}\in\mathcal W_0$ of $\operatorname{Re} A_0$ -- equivalently, an extremal point $x_0\in \mathcal Z_0$ of $f_0$ (see Remark~\ref{rem:extremal}) -- may generate an irregularity in $f(t,x)$. The next proposition shows, roughly speaking, that macroscopic voids or congestions propagate at linear speed until they reach a critical point, where a loss of regularity in the derivative occurs.

\begin{prop}\label{prop:deriv}
	Assume that both $A_0$ and $A_0'$ belong to $\mathcal A(\mathbb D)$. For any $x_0\in \mathcal Z_0$, set $w_0=e^{ix_0}$ and $v_0=\operatorname{Im} A_0(w_0)$. Then the following properties hold:
	\begin{enumerate}
		\item For sufficiently small $t\geq 0$, the boundary $\partial V_t$ is a $\mathcal C^1$ closed curve  tangent to the unit circle at $w_0$. Moreover, one has $x_0+t v_0\in\mathcal Z_t$.
		\item Let $t_0^\ast=\inf\{t>0 : w_0\notin \partial V_t\}$. Then $t_0^\ast<\infty$, and necessarily
		$1+t_0^\ast\, w_0 A_0'(w_0)=0$.
	\end{enumerate}
\end{prop}

\begin{proof}[Proof of Proposition \ref{prop:deriv}]
	The first claim follows from the implicit function theorem, after applying Whitney’s extension theorem to work locally on an open set. The tangency property is a consequence of the differentiability of $\partial V_t$ at $w_0$, since the boundary must remain inside the unit disk. 
	
	From \eqref{eq:carac0} and local uniqueness at noncritical points, we deduce that $x_0+t v_0\in\mathcal Z_t$ for sufficiently small $t\geq 0$. Proposition~\ref{prop:spontaneous} ensures that $t_0^\ast<\infty$. Arguing by contradiction, $w_0$ must be a critical point at time $t_0^\ast$. Otherwise, for small $s>0$, one would have $w_0\in\partial V_{t_0^\ast+s}$.
\end{proof}

\begin{rem}\label{rem:extend0}
	The global assumptions on $A_0$ in Propositions~\ref{prop:cont} and~\ref{prop:deriv} can be localized, yielding corresponding local regularity results. In addition, higher regularity assumptions on $A_0$  entail enhanced boundary regularity.
\end{rem}

\subsection{Single-mode periodic density profile with saturation}

\label{sec:periodicdens}

The following example complements Proposition~\ref{prop:deriv} above  and shows that even an entire analytic initial density profile may generate irregularities. 

For $\alpha=1/2$, consider the following initial density profile, together with its moment generating function and  characteristic coefficient:
\begin{equation}
f_0(x)=\frac{1+\cos(px)}{2\pi},\quad g_0(z)=\frac{z^p}{2},\quad A_0(w)=2\pi^2\cos\left(\frac{\pi}{2}w^p\right),
\end{equation}
with $p\geq1$. Note that $\mathcal W_0=\mathbb U_p\sqcup \zeta\,\mathbb U_p$, where $\mathbb U_p=\{\omega_1,\cdots,\omega_p\}$ denotes the set of $p$th roots of unity and $\zeta^p=-1$. Consequently,
\begin{equation}
\mathcal Z_0=\left\{\frac{2k\pi}{p}:0\leq k\leq p-1\right\}\sqcup\left\{\frac{(2k+1)\pi}{p}:0\leq k\leq p-1\right\}.
\end{equation}

Note that $\mathcal W_0$ is invariant under complex conjugation and sign change. Moreover, $A_0(w_0)=0$ for all $w_0\in\mathcal W_0$, so the extremal points of $f(t,x)$ are locally stationary.

We claim that for any $w_0\in\mathcal W_0$,
\begin{equation}
t^\ast=\inf\{t>0 : w_0\notin \partial V_t\}=\frac{1}{p\pi^3}.
\end{equation}

Indeed, observe  that $\Phi_t'(w_0)=0$ if and only if $t=1/(p\pi^3)$. By using Proposition \ref{prop:deriv}, we have necessarily $w_0\in \partial V_t$ for all $t\leq t^\ast$. Furthermore, one can check  that $w_0\notin\partial V_t$ for all $t>t^\ast$. To this end, we study the map $\lambda\mapsto \Phi_t(\lambda w_0)\overline{w_0}$ on $[0,1]$ and we  note that there exists a unique $\lambda_t\in(0,1)$ such that $\lambda_t w_0\in\partial V_t$. By symmetry, one easily checks that $\lambda w_0\in V_t$ for all $\lambda\in(-\lambda_t,\lambda_t)$, whereas $\lambda w_0\in V_t^c$ for all $\lambda\in(\lambda_t,1)\cup(-1,-\lambda_t)$. Note also that the points $\lambda_t w_0$ are regular points of $\Phi_t$, as are the points $w_0$. As a consequence, none of the points $w_0$ can belong to the boundary, since this would contradict one of the preceding properties.

More precisely, the geometry of the curve $\partial V_t$, invariant under complex conjugation, sign change, and rotation by an angle $2\pi/p$, depends on the relative position of $t$ with respect to $t^\ast$:
\begin{enumerate}
	\item If $0<t<t^\ast$, the curve $\partial V_t$ is analytic and connects $w_0$ to $e^{2i\pi/p}w_0$ for any $w_0\in\mathcal W_0$.
	\item If $t=t^\ast$, the curve $\partial V_{t^\ast}$ is continuous and analytic except at the points $w_0$, and has the same geometric configuration as in the previous case.
	\item If $t>t^\ast$, the curve $\partial V_t$ is analytic and connects $\lambda_t w_0$ to $\lambda_t e^{2i\pi/p}w_0$ for any $w_0\in\mathcal W_0$.
\end{enumerate}

\begin{figure}[!h]
	\centering
	\includegraphics[width=0.425\textwidth]{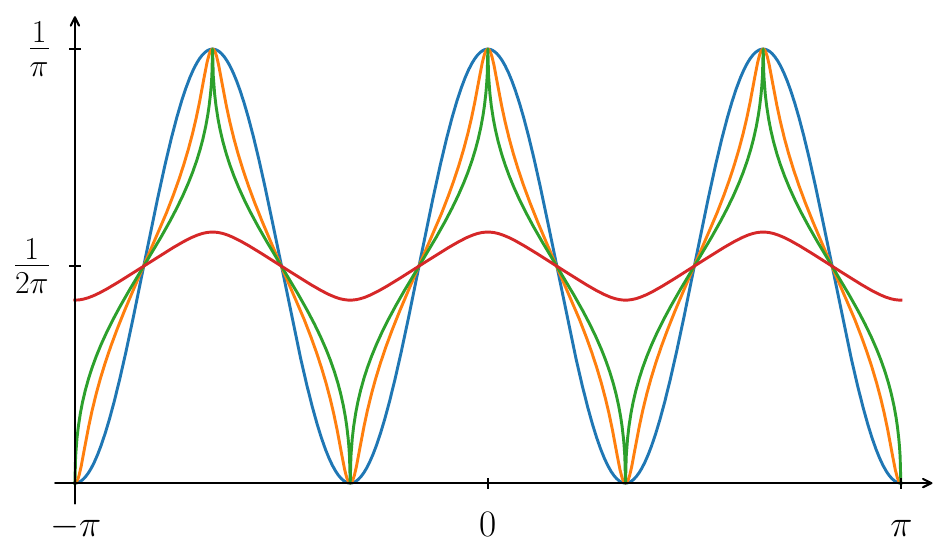}
	\includegraphics[width=0.325\textwidth]{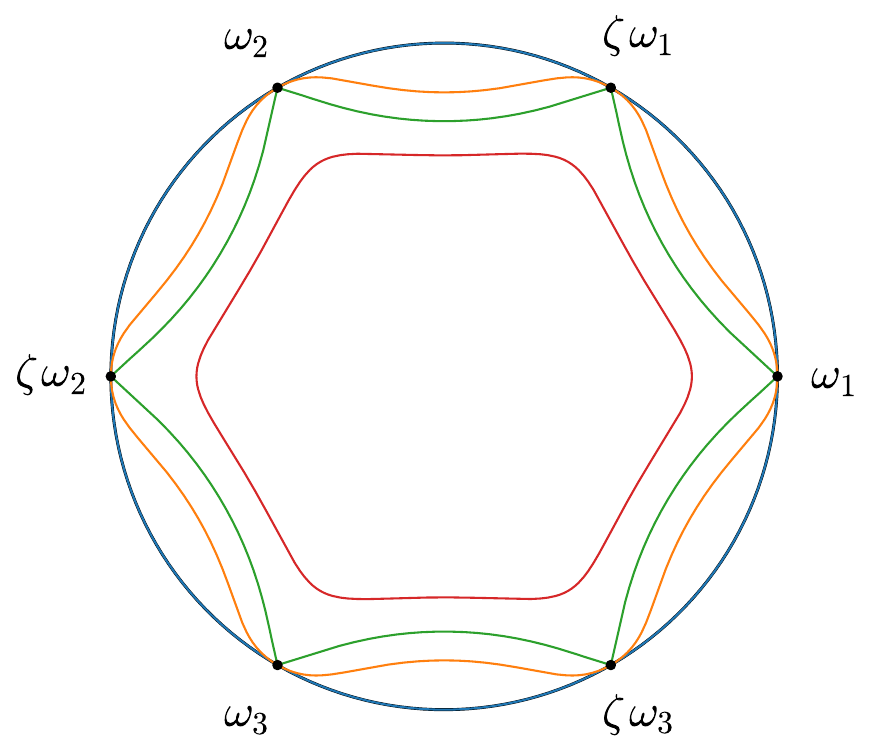}
	\captionsetup{width=15.1cm}
	\caption{\small  Evolution of the density $f(t,x)$ associated with a periodic harmonic profile with $p=3$,  at four different times $0<t_1<t^\ast<t_3$, and corresponding evolution of the boundary $\partial V_t$.
	}
	\label{fig:levels}
\end{figure}

We conclude that $f(t,x)$ is analytic except at $t=t^\ast$ and $x\in\mathcal Z_0$. The function remains continuous at these points, but does not admit a spatial derivative there (see  Figure~\ref{fig:levels}).

To go further, the order of the resulting irregularity can be determined as follows. First, the Lagrange--Bürmann formula (see, for instance, \cite{Flajolet09}) allows us to write around the origin:
\begin{equation}
w(t,z)=\sum_{n=1}^\infty \mathfrak a_n(t)\, z^n,
\quad\text{with}\quad 
\mathfrak a_n(t)=\frac{1}{n}\,[w^{n-1}]\,e^{-n t A_0(w)}.
\end{equation}
Here, as usual, $[w^{k}]S(w)$ denotes the coefficient of $w^k$ in the power series expansion of $S(w)$.

 Equivalently, one has
\begin{equation}\label{eq:contoura}
\mathfrak a_n(t)=\frac{1}{2\pi i\, n}\oint_{\Gamma}\frac{e^{-n t A_0(w)}}{w^n}\,dw
=\frac{1}{2\pi i\, n}\oint_{\Gamma}e^{-n\{t A_0(w)+\log(w)\}}\,dw,
\end{equation}
where $\Gamma$ is a positively oriented simple closed contour enclosing $0$.

To derive the precise asymptotics of $\mathfrak a_n(t)$, we apply the classical method of steepest descent to~\eqref{eq:contoura}, also known as the saddle-point, Laplace, or stationary phase method (see, e.g., \cite{Flajolet09,Bleistein86,Wong}). In the present setting, we take $\Gamma=\partial V_t$. Along this contour, the integrand in~\eqref{eq:contoura} has modulus one, so the phase function
\begin{equation}\label{eq:fgot}
\mathfrak f_t(w)=tA_0(w)+\log(w),
\end{equation}
is purely imaginary, since $\partial V_t$ lies in the level set $\operatorname{Re} A_0(w)=0$. Moreover, as $\Phi_t$, $\mathfrak f_t$, and $\operatorname{Im}\mathfrak f_t$ have the same critical points, the stationary phase method is the natural tool for the asymptotic analysis when $t=t^\ast$, where the critical points $w_0$ lie on $\partial V_{t^\ast}$.

In addition, one verifies that
\begin{equation}
\mathfrak f_{t^\ast}(w_0)=i\frac{(2 k_0+\epsilon_0)\pi}{p}\quad\text{and}\quad  \mathfrak f_{t^\ast}''(w_0)=-p\,\overline w_0^{\,2},
\end{equation}
for some $0\leq k_0\leq p-1$ and $\epsilon_0\in\{0,1\}$. 

Consequently, at each non-degenerate critical point $w_0$, the curve $\partial V_{t^\ast}$ meets the unit circle with angle $\pi/4$, and the two branches issuing from $w_0$ are orthogonal. 

Combining straightforward computations with geometric considerations, one finds that the contribution of $w_0$ (from the two sides of the branch) to the integral in~\eqref{eq:contoura} is asymptotically
\begin{equation}
i\,\overline w_0^{\,n-1}\sqrt{\frac{\pi}{2np}}+\mathcal O\left(\frac{1}{n}\right).
\end{equation}

Summing all these contributions and dividing by $2i\pi n$ yields
\begin{equation}
\mathfrak a_n(t^\ast)=
\begin{cases}
\displaystyle \frac{1}{n\sqrt{2\pi np}}+\mathcal O\!\left(\frac{1}{n^{2}}\right), & \text{if } 2p\mid(n-1),\\[15pt]
\displaystyle \mathcal O\!\left(\frac{1}{n^{2}}\right), & \text{otherwise}.
\end{cases}
\end{equation}

Standard estimates 
then show that $x\mapsto f(t^\ast,x)$ exhibits a square-root behaviour in a neighbourhood of each extremal point $x\in\mathcal Z_0$, just as  $x\mapsto w(t^\ast,e^{ix})$ does.

\subsection{Proof of Theorem \ref{thm:pack}}

\label{sec:packedproof}

We now consider the case where the initial density $f_0$ is given by~\eqref{eq:stepprofile}. Unlike the previous example, $f_0$ is discontinuous and its saturation set $\mathcal Z_0$ coincides with the whole domain. The associated analytic structure exhibits new features, requiring a separate analysis.

\subsubsection{The critical points} 

First, it comes from \eqref{eq:g0A0} that $\Phi_t(w)=we^{t A_0(w)}$ is well defined on $\mathcal U_{\alpha}=\mathbb C\setminus\{\omega_\alpha,\overline{\omega_\alpha}\}$. We shall give a precise description of  $\mathcal C_{\alpha,t}=\{w\in\mathcal U_\alpha : 1+ t w A_0'(w)=0\}$.

To this end, one readily checks that
\begin{equation}
A_0'(w)=2\pi^2 \times 
\left\{\begin{array}{ll}
\displaystyle  \frac{-4w}{(w-i)^2(w+i)^2}, &\text{ if $\alpha=1/2$,}\\[15pt] 
\displaystyle 2 \cos(\pi\alpha) \frac{(w-\beta)(w-1/\beta)}{(w-\omega_\alpha)^2(w-\overline{\omega_\alpha})^2}, & \text{ if $\alpha\neq 1/2$,}
\end{array}\right.
\end{equation}
where
\begin{equation}
\beta=\frac{1-\sin(\pi\alpha)}{\cos{\pi\alpha}}
\quad\text{and}\quad
1/\beta=\frac{1+\sin(\pi\alpha)}{\cos(\pi\alpha)}.
\end{equation}

We then obtain
\begin{equation}\label{eq:critical}
1+t w A_0^\prime(w)  
=\frac{P_{\alpha,t}(w)}{(w-\omega_\alpha)^2(w-\overline{ \omega_\alpha})^2},
\end{equation}
where
\begin{equation}\label{eq:polrootsaddle}
P_{\alpha,t}(w)=w^4+a w^3 + b w^2 + a w + 1, 
\quad \text{with}\quad 
\begin{cases}
a = 4\cos(\pi\alpha)(\pi^2 t-1), \\
b = 4\cos^2(\pi\alpha) + 2 - 8\pi^2 t.
\end{cases}
\end{equation}

As a consequence,  $\mathcal C_{\alpha,t}$ coincides with the set of roots of $P_{\alpha,t}$. Since this polynomial has real coefficients and is palindromic (that is, reciprocal), it follows that if $w\in\mathcal C_{\alpha,t}$, then both $\overline w$ and $1/w$ also belong to $\mathcal C_{\alpha,t}$. Accordingly, only four distinct configurations may occur:
\begin{equation}\label{eq:roots}
\mathcal{C}_{\alpha,t} =
\begin{cases}
\{\zeta_1,\overline{\zeta}_1,\zeta_2,\overline{\zeta}_2\}, & \text{with } \zeta_1, \zeta_2 \in \mathbb{U}, \\[5pt]
\{x_1,1/x_1,\zeta_2,\overline{\zeta}_2\}, & \text{with } x_1 \in \mathbb{R}^\ast,\, \zeta_2 \in \mathbb{U}, \\[5pt]
\{x_1,1/x_1,x_2,1/x_2\}, & \text{with } x_1, x_2 \in \mathbb{R}^\ast, \\[5pt]
\{w,1/w,\overline{w},1/\overline{w}\}, & \text{with } w \in \mathbb{C}^\ast.
\end{cases}
\end{equation}

As a matter of fact, only the first three cases in~\eqref{eq:roots} are possible. More precisely, these configurations arise according to the position of $t$ relative to $t_\ast$ and $t^\ast$, given by
\begin{equation}\label{eq:tempscritiques}
t_\ast = \min\left(\frac{1-\cos(\pi\alpha)}{2\pi^2},\, \frac{1+\cos(\pi\alpha)}{2\pi^2}\right),
\;\;\; 
t^\ast = \max\left(\frac{1-\cos(\pi\alpha)}{2\pi^2},\, \frac{1+\cos(\pi\alpha)}{2\pi^2}\right).
\end{equation}

\begin{rem}\label{rem:alphasym}
	One can check that $\mathcal C_{1-\alpha,t} = -\mathcal C_{\alpha,t}$. This symmetry reflects the duality between the evolution of the particle system and that of the empty sites: both processes follow the same maximal entropy dynamics, with initial densities easily deduced from one another by exchanging $\alpha$ with $1-\alpha$. 
\end{rem}

Hence,   one may assume without loss of generality that $\alpha \leq 1/2$.

\begin{lem}\label{lem:parametric}
When $\alpha\leq 1/2$,	the set of critical points $\mathcal{C}_{\alpha,t}$ has the following structure:
	\begin{enumerate}
		\item For all $0<t<t_\ast$, there exist four distinct critical points $\zeta_1,\overline{\zeta}_1,\zeta_2,\overline{\zeta}_2$ on the unit circle, with $\arg(\zeta_1)\in(0,\pi\alpha)$ and $\arg(\zeta_2)\in(\pi\alpha,\pi)$.
		
		\item When $t=t_\ast$ and $\alpha< 1/2$, there exist two distinct critical points $\zeta_2,\overline{\zeta}_2$ on the unit circle, with $\arg(\zeta_2)\in(\pi\alpha,\pi)$, and $1$ is a critical point of multiplicity two.
		
		\item For all $t_\ast<t<t^\ast$, there exist two distinct critical points $\zeta_2,\overline{\zeta}_2$ on the unit circle, with $\arg(\zeta_2)\in(\pi\alpha,\pi)$, and two distinct positive critical points $x_1$ and $1/x_1$, with $x_1<1$.
		
		\item When $t=t^\ast$ and $\alpha< 1/2$, one has $-1$ is a critical point of multiplicity two and there exist two distinct positive critical points $x_1$ and $1/x_1$ with $x_1<1$. 
		
		\item For all $t>t^\ast$, there exist four distinct real critical points $x_1,1/x_1,x_2,1/x_2$, where $0<x_1<1$ and $-1<x_2<0$.
		
		\item When $\alpha=1/2$ and $t=t_\ast=t^\ast$, both $-1$ and $1$ are critical points of multiplicity two.
	\end{enumerate}
\end{lem}

\begin{proof}[Proof of Lemma \ref{lem:parametric}]
Using $P_{\alpha,t}$ is reciprocal, one can see that $P_{\alpha,t}(z)=0$ if and only if
	$$
	w^2 + \frac{a}{2}w + \frac{b - 2}{4} = 0,\quad \text{with} \quad w = \frac{1}{2}\left(z + \frac{1}{z}\right).\label{eq:wpol}
	$$
	
	Besides, since $z^2 - 2wz + 1 = 0$, the roots of $P_{\alpha,t}(z)$ can be obtained from those of the associated polynomial in $w$ by setting $z = w \pm \sqrt{w^2 - 1}$. Note that the discriminant is given by
	\begin{equation}
	\Delta = \frac{a^2 - 4(b - 2)}{4} = 4\pi^2 t \left(\pi^2 \cos^2(\pi\alpha)\, t + 2\sin^2(\pi\alpha)\right).
	\end{equation}
	
	Hence, the corresponding roots are given by
	\begin{equation}\label{eq:eqwt}
	w^{\pm}(t) = \cos(\pi\alpha)(1 - \pi^2 t) \pm \sqrt{\pi^2 t\left(\pi^2 \cos^2(\pi\alpha)\, t + 2\sin^2(\pi\alpha)\right)}.
	\end{equation}

Straightforward computations show that $w^+(t)$ increases, with $w^+(0) = \cos(\pi\alpha)$ and $w^+(t_\ast) = 1$, whereas $w^-(t)$ decreases, with $w^-(0) = \cos(\pi\alpha)$ and $w^-(t^\ast) = -1$. The structure of $\mathcal C_{\alpha,t}$ then follows immediately from these monotonicity properties.
\end{proof}

\subsubsection{The boundary $\partial V_t$} 

Let us introduce  $\mathfrak f_t(w)$ as in \eqref{eq:fgot}, so that $\Phi_t(w)=e^{\mathfrak f_t(w)}$, and the level set
\begin{equation}\label{eq:levelset}
\mathcal R_t = \left\{w \in \mathcal U_\alpha\setminus\{0\} : \operatorname{Re}\mathfrak f_t(w) = 0\right\}.
\end{equation}

First note that $\mathcal R_t$ is locally a one-dimensional smooth curve, except at the critical points of $\operatorname{Re}\mathfrak f_t$, which coincide with those of $\mathfrak f_t$ and  $\Phi_t$. Moreover, observe that $\mathbb U\setminus \{\omega_\alpha,\overline{\omega_\alpha}\}\subset \mathcal R_t$, since for all $\theta\neq  \pm\pi\alpha$ modulo  $2\pi$, one has 
\begin{equation}\label{eq:unitcirclecurve}
A_0(e^{i\theta}) = \frac{2\pi^2 \sin(\theta)\, i}{\cos(\pi\alpha)-\cos(\theta)}.
\end{equation}

We describe the geometry of $\mathcal R_t$ as $t>0$ varies, and in particular show that it contains a simple closed curve lying in $\mathcal R_t\cap \overline{\mathbb D}$. We will later justify that this curve coincides with $\partial V_t$.

Moreover, following Remark~\ref{rem:alphasym}, we still assume without loss of generality that $\alpha\leq 1/2$. We refer to Figure~\ref{fig:levelsetstep}, which illustrates the evolution of the boundary, and to Lemmas~\ref{lem:path1}, \ref{lem:path1bis}, \ref{lem:path1bisbis}, and \ref{lem:path1bisbisfin} for the precise mathematical description.

\begin{figure}[!h]
	\centering
	\includegraphics[width=0.35\textwidth]{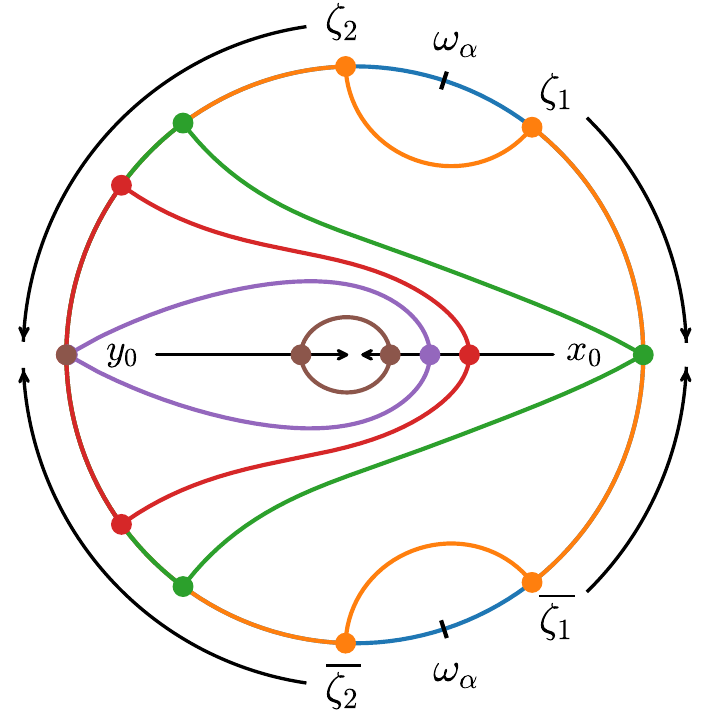}
		\captionsetup{width=14.75cm}
	\caption{\small  Evolution of the boundary $\partial V_t$ as $t$ varies. The critical points $\zeta_1$ and $\overline{\zeta_1}$ collide at $1$ when $t=t_\ast$ and give rise to a regular point $x_0>0$ moving toward $0$, while $\zeta_2$ and $\overline{\zeta_2}$ collide at $t=t^\ast$ and produce a negative regular point $y_0<0$ also converging to $0$.}
	\label{fig:levelsetstep}
\end{figure}

\medskip

\noindent
{\it Case $0<t\leq t_\ast$.} Let $w_c$ be a critical point on the unit circle. It follows from \eqref{eq:critical} that $w_c$ is a non-degenerate critical point of $\Phi_t$ if and only if $w_c$ is a simple root of  $P_{\alpha,t}$. The only situation in which this fails occurs is when $w_c=\pm 1$ (See Lemma~\ref{lem:parametric} and its proof). In this case, the critical point is of order two, that is, $w_c$ is a root of multiplicity two of $P_{\alpha,t}$.

Elementary results from complex analysis show that the level set $\mathcal R_t$ has a simple local structure near its critical points. Near a non-degenerate critical point, $\mathcal R_t$ consists locally of two smooth branches intersecting transversely, one of which coincides with the unit circle. In contrast, near a degenerate critical point of order two, $\mathcal R_t$ consists locally of three smooth branches meeting at equal angles of $\pi/3$, which are mapped onto each other by rotations of angle $\pi/3$, with one of these branches again coinciding with the unit circle.

\begin{lem}[case $t < t_\ast$]\label{lem:path1}
	Following the level set $\mathcal R_0$ starting from $\zeta_1$ (resp.\ $\overline{\zeta_1}$) and moving inward or outward across the unit circle, the path reaches $\zeta_2$ (resp.\ $\overline{\zeta_2}$). Moreover, each such path remains strictly above (resp.\ below) the real axis.
\end{lem}

\begin{proof}[Proof of Lemma~\ref{lem:path1}] By symmetry, it suffices to establish this result for the non-conjugate critical points. The fact that the level set $\mathcal R_0$ cannot cross the real axis, except at the points $\pm1$ which are excluded here, can be readily deduced from the variation of $\operatorname{Re} \mathfrak f_t(x)=\mathfrak f_t(x)$ on $\mathbb R^\ast$. 

Based on this observation, starting from $\zeta_1$ and following the level set orthogonally to the unit circle and inward to it, there are only two possible scenarios: either the level set reaches $\zeta_2$, or it ends at $\omega_\alpha$. Hence, it suffices to show that the latter case cannot occur.

	To this end, observe that $\mathfrak g_t(w) = (\mathfrak f_t(w))^{-1}$ extends as a holomorphic function in a neighbourhood of $\omega_\alpha$, and satisfies
	\begin{equation}
	\mathfrak g_t(\omega_\alpha)=0
	\quad\text{and}\quad
	\mathfrak g_t'(\omega_\alpha)=-\frac{\sin(\pi\alpha)}{\pi^2 t}\, i \neq 0.
	\end{equation}
It follows that $\mathfrak g_t$ is biholomorphic in some neighbourhood $\mathcal B$ of $\omega_\alpha$.

Moreover, since
	\begin{equation}
	\mathcal R_t \cap \mathcal B
	=
	\{\operatorname{Re}\mathfrak g_t(w)=0\}\cap \mathcal B \setminus \{\omega_\alpha\},
	\end{equation}
	we deduce that $(\mathcal R_t \cap \mathcal B)\sqcup \{\omega_\alpha\}$ is a smooth curve. 
	
	By local uniqueness, it must  coincide with the circular arc $\mathbb U \cap \mathcal B$, which shows that the branch issuing from $\zeta_1$ and pointing orthogonally inward to the unit circle cannot terminate at $\omega_\alpha$.\end{proof}

The proof of the following lemma follows exactly the same arguments as in Lemma~\ref{lem:path1} above, its proof is omitted.

\begin{lem}[case $t = t_\ast$]\label{lem:path1bis}
	Following the level set $\mathcal R_t$ from $1$ and moving inward or outward across the unit circle in the upper (resp.\ lower) direction, the path reaches $\zeta_2$ (resp.\ $\overline{\zeta_2}$). Moreover, it remains strictly above (resp.\ below) the real axis, except at the starting point $1$ and, when $\alpha=1/2$ (equivalently $t_\ast=t^\ast$), at the endpoint $-1$.
\end{lem}

\medskip

\noindent
{\it Case $t_\ast<t\leq t^\ast$.} 
Studying the variation of $\operatorname{Re}\mathfrak f_t(x)=\mathfrak f_t(x)$ on $\mathbb R^\ast$ shows that there exists a unique point $x_0\in(0,1)$ such that $\mathfrak f_t(x_0)=0$. More precisely, one finds that $0<x_0<1/x_1$, where $1/x_1$ denotes the unique real critical point in $(0,1)$ introduced in Lemma~\ref{lem:parametric}. In particular, $x_0$ is a regular point of the level set, and it can be  shown that the latter is orthogonal to the real axis at this point.

\begin{lem}\label{lem:path1bisbis}
	Following the level set $\mathcal R_t$ from $x_0$ and moving in the upper (resp.\ lower) direction, the path reaches $\zeta_2$ (resp.\ $\overline{\zeta_2}$). Moreover, it  remains strictly above (resp.\ below) the real axis, except at the starting point $x_0$ and, when $t=t^\ast$, at the endpoint $-1$.
\end{lem}

The proof of the latter lemma follows exactly the same arguments as in Lemma~\ref{lem:path1} above, as Lemma \ref{lem:path1bis}, its proof is omitted.

\medskip

\noindent
{\it Case $t>t^\ast$.} 
Here, the variation of $\operatorname{Re}\mathfrak f_t(x)=\mathfrak f_t(x)$ shows that, in addition to $x_0$, there exists a second point $y_0\in(0,1)$ such that $\mathfrak f_t(y_0)=0$. These are the only zeros  on $(-1,1)$. Furthermore, the same arguments as in the previous lemmas yield the following result.

\begin{lem}\label{lem:path1bisbisfin}
	Following the level set $\mathcal R_t$ from $x_0$ and moving in the upper (resp.\ lower) direction, the path reaches $y_0$. Moreover, it remains strictly above (resp.\ below) the real axis, except at the starting point $x_0$ and at the endpoint $y_0$.
\end{lem}

It remains to show that the simple closed curve described above, whose interior we denote by $V_t$ by abuse of notation, indeed coincides with the boundary $\partial V_t$ characterized in Lemma~\ref{lem:tech1}. In fact, it suffices to prove the following result.

\begin{lem}\label{lem:homeo}
	$\Phi_t$ induces an orientation-preserving homeomorphism from $\partial V_t$ onto $\partial \mathbb D$.
\end{lem}

Assuming this lemma, the argument principle shows that
\begin{equation}
\#\{w\in V_t : \Phi_t(w)=z\}
=\frac{1}{2i\pi}\int_{\partial V_t}\frac{\Phi_t'(w)}{\Phi_t(w)-z}\,dw
=\frac{1}{2i\pi}\int_{\partial \mathbb D}\frac{dw}{w-z}
=1,
\end{equation}
for all $z\in\mathbb D$ and all $t>0$, which is precisely the characterisation of the open set $V_t$.

\begin{proof}[Proof of Lemma~\ref{lem:homeo}]
By analyzing the variation of the function
\begin{equation}\label{eq:varim}
\theta\longmapsto \arg(\Phi_t(e^{i\theta}))
=\theta+\frac{2 \pi^2 t\,\sin(\theta)}{\cos(\pi\alpha)-\cos(\theta)},
\end{equation}
whose critical points coincide with the arguments of the critical points of $\Phi_t$ lying on the unit circle, one observes that, for $t\le t_\ast$, following the portion of $\partial V_t$ from $-\arg(\zeta_1)$ to $\arg(\zeta_1)$ in the trigonometric direction, its image under $\Phi_t$ traces positively the corresponding arc of $\partial\mathbb D$, from $\xi_1=\Phi_t(\overline{\zeta_1})$ to $\overline{\xi_1}=\Phi_t(\zeta_1)$. The same reasoning applies to the other portions of $\partial V_t$ lying on the unit circle, for instance the arc between $\zeta_2$ and $\overline{\zeta_2}$.

	We now consider an arc of $\partial V_t$ not contained in the unit circle, for instance the arc joining $\zeta_1$ to $\zeta_2$ inside the disk when $t\leq t_\ast$. Let $\gamma$ be a one-to-one parametrization of this arc, and set $z(s)=\Phi_t(\gamma(s))$. Then $z(s)$ runs along $\partial\mathbb D$ from $\xi_1$ to $\xi_2$. Since the level sets of the real and imaginary parts of a holomorphic function are orthogonal away from critical points, the function $s\mapsto \operatorname{Im}\mathfrak f_t(\gamma(s))$ is monotone. The correct orientation follows from a local Taylor expansion of $\mathfrak f_t$ at $\zeta_1$ (of order two when $\zeta_1\neq1$, and of order three otherwise). The same reasoning applies to all remaining arcs of $\partial V_t$.
	
	Altogether, this yields a positively oriented continuous parametrization $s\in[0,1]\mapsto\gamma(s)$ of $\partial V_t$ such that $z(s)=\Phi_t(\gamma(s))$ parametrizes $\partial\mathbb D$ positively. To conclude, it suffices to check that this parametrization is one-to-one. This follows from  the computation of the winding number
	\begin{equation}
	\frac{1}{2i\pi}\int_0^1\frac{z'(s)}{z(s)}\,ds
	=\frac{1}{2i\pi}\int_0^1\frac{\gamma'(s)}{\gamma(s)}\,ds
	+\frac{1}{2i\pi}\int_0^1 t\,A_0'(\gamma(s))\,ds
	=1.
	\end{equation}

This proves that $\Phi_t$ induces an orientation-preserving homeomorphism from $\partial V_t$ onto $\partial\mathbb D$.
\end{proof}

\subsubsection{Regularity and saturation properties}

\label{sec:final}

Whenever it is well defined, we set
\begin{equation}
\xi_i(t)=\Phi_t(\zeta_i)\in\mathbb D
\quad\text{and}\quad
\phi_i(t)=\arg(\xi_i(t))\in(-\pi,\pi),
\quad i\in\{1,2\}.
\end{equation} 

One can check that $\phi_1$ decreases smoothly from $\pi\alpha$ to $0$ on $(0,t_\ast)$, while $\phi_2$ increases smoothly from $\pi\alpha$ to $\pi$ on $(0,t^\ast)$. We recall that we assume $\alpha\leq 1/2$.

It then follows from the method of characteristics that  $x\mapsto f(t,x)$ is continuous for all $t>0$, and analytic except at the points $x\in\{\pm\phi_1(t),\pm\phi_2(t)\}$. These points are precisely the boundaries of the saturation sets. 

The density attains the maximal  value $1/(2\pi\alpha)$ on the interval $(-\phi_1(t),\phi_1(t))$, for all $0\leq t\leq t_\ast$, and vanishes on the complement of $(-\phi_2(t),\phi_2(t))$, for all $0\leq t\leq t^\ast$. For all $t>t^\ast$, the density becomes smooth and remains uniformly bounded away from these extremal values.

Regarding the order of the irregularities, the same method as in Section~\ref{sec:periodicdens} can be applied to determine the asymptotic behaviour of the complex moments of $x\mapsto f(t,x)$. More precisely, one can show that  $\mathfrak m_n(t)$ satisfy the following asymptotic expansions, depending on the regime:
\begin{equation}\label{eq:asymp-cases}
\mathfrak m_n(t)
=
\begin{cases}
\dfrac{\sum_{i\in\{1,2\},\,\pm}\lambda_{i,\pm}\,e^{\pm i n \phi_i(t)}}{n^{3/2}}
+\mathcal O\!\left(\dfrac{1}{n^{2}}\right),
& \quad\text{if } t<t_\ast, \\[15pt]
\dfrac{\lambda}{n^{4/3}}
+\mathcal O\!\left(\dfrac{1}{n^{5/3}}\right),
& \quad\text{if } t=t_\ast \ \text{or}\ t=t^\ast, \\[15pt]
\dfrac{\sum_{\pm}\lambda_{2,\pm}\,e^{\pm i n \phi_2(t)}}{n^{3/2}}
+\mathcal O\!\left(\dfrac{1}{n^{2}}\right),
& \quad\text{if } t_\ast<t<t^\ast, \\[15pt]
\dfrac{\lambda\,q(t)^n}{n^{3/2}}
+\mathcal O\!\left(\dfrac{q(t)^n}{n^{2}}\right),
& \quad\text{if } t>t^\ast.
\end{cases}
\end{equation}

Above, one has $q(t)=\Phi_t(x_1)$, where $x_1$ denotes the unique critical point of $\Phi_t$ in $(0,1)$, and the coefficients $\lambda_{i,\pm}$ and $\lambda$ are some nonzero complex constants depending on $t$. 

Since the proof relies on standard but rather technical asymptotic arguments and yields no further insight, it is omitted and left to the interested reader. 

Finally, these  asymptotics show that the singularities of the density at these points are of square-root type, except at $0\!\!\mod \pi$, when $t\in\{t_\ast, t^\ast\}$, where they display cubic-root behaviour.

\bigskip

\noindent
{\bf Acknowledgements.} The author thanks Kilian Raschel for drawing his attention to reference~\cite{GLT2023}.

\bibliographystyle{alpha}
\bibliography{bibliomesep}

\newcommand{\etalchar}[1]{$^{#1}$}
\begin{thebibliography}{CDMG10}

\bibitem[Aba06]{Abanov2}
Alexander~G. Abanov.
\newblock Hydrodynamics of correlated systems.
\newblock In {\em Applications of random matrices in physics}, volume 221 of
  {\em NATO Sci. Ser. II Math. Phys. Chem.}, pages 139--161. Springer,
  Dordrecht, 2006.

\bibitem[AGS08]{Ambrosio}
Luigi Ambrosio, Nicola Gigli, and Giuseppe Savar\'e.
\newblock {\em Gradient flows in metric spaces and in the space of probability
  measures}.
\newblock Lectures in Mathematics ETH Z\"urich. Birkh\"auser Verlag, Basel,
  second edition, 2008.

\bibitem[Ahl]{Ahlfors1966}
L.~V. Ahlfors.
\newblock {\em Complex Analysis}.
\newblock McGraw-Hill Book Company, 2 edition.

\bibitem[AK16]{Katori}
Sergio Andraus and Makoto Katori.
\newblock Characterizations of the hydrodynamic limit of the {D}yson model.
\newblock In {\em Stochastic analysis on large scale interacting systems},
  volume B59 of {\em RIMS K\^oky\^uroku Bessatsu}, pages 157--173. Res. Inst.
  Math. Sci. (RIMS), Kyoto, 2016.

\bibitem[BBMP14]{Majumdar}
J.~Bun, J.~P. Bouchaud, S.~N. Majumdar, and M.~Potters.
\newblock Instanton approach to large $n$ harish-chandra-itzykson-zuber
  integrals.
\newblock {\em Phys. Rev. Lett.}, 113:070201, Aug 2014.

\bibitem[BCDM12]{Bracci}
Filippo Bracci, Manuel~D. Contreras, and Santiago D\'iaz-Madrigal.
\newblock Evolution families and the {L}oewner equation {I}: the unit disc.
\newblock {\em J. Reine Angew. Math.}, 672:1--37, 2012.

\bibitem[BDL{\etalchar{+}}20]{SimonS}
Oriane Blondel, Aurelia Deshayes, Cyril Labb\'e, Laure Mar\^ech\'e, and
  Marielle Simon.
\newblock Dynamics of interacting particle systems.
\newblock In {\em Journ\'ees {MAS} 2018---sampling and processes}, volume~68 of
  {\em ESAIM Proc. Surveys}, pages 52--72. EDP Sci., Les Ulis, 2020.

\bibitem[BDLL22]{Bertucci1-1}
Charles Bertucci, M\'erouane Debbah, Jean-Michel Lasry, and Pierre-Louis Lions.
\newblock A spectral dominance approach to large random matrices.
\newblock {\em J. Math. Pures Appl. (9)}, 164:27--56, 2022.

\bibitem[BDLW10]{BurdaReview}
Z.~Burda, J.~Duda, J.~M. Luck, and B.~Waclaw.
\newblock The various facets of random walk entropy.
\newblock {\em Acta Physica Polonica B}, 41(5):949--987, 2010.
\newblock Presented by Z. Burda at the XXII Marian Smoluchowski Symposium on
  Statistical Physics.

\bibitem[BH86]{Bleistein86}
Norman Bleistein and Richard~A. Handelsman.
\newblock {\em Asymptotic expansions of integrals}.
\newblock Dover Publications, Inc., New York, second edition, 1986.

\bibitem[Bia97a]{BianeFree}
Philippe Biane.
\newblock Free {B}rownian motion, free stochastic calculus and random matrices.
\newblock In {\em Free probability theory ({W}aterloo, {ON}, 1995)}, volume~12
  of {\em Fields Inst. Commun.}, pages 1--19. Amer. Math. Soc., Providence, RI,
  1997.

\bibitem[Bia97b]{BianeAnalogue}
Philippe Biane.
\newblock Segal-{B}argmann transform, functional calculus on matrix spaces and
  the theory of semi-circular and circular systems.
\newblock {\em J. Funct. Anal.}, 144(1):232--286, 1997.

\bibitem[BKM10]{BilerCrystal}
Piotr Biler, Grzegorz Karch, and R\'egis Monneau.
\newblock Nonlinear diffusion of dislocation density and self-similar
  solutions.
\newblock {\em Comm. Math. Phys.}, 294(1):145--168, 2010.

\bibitem[BLL24]{Bertucci1-2}
Charles Bertucci, Jean-Michel Lasry, and Pierre-Louis Lions.
\newblock A spectral dominance approach to large random matrices: {P}art {II}.
\newblock {\em J. Math. Pures Appl. (9)}, 192:Paper No. 103630, 30, 2024.

\bibitem[Bor11]{Boro}
Alexei Borodin.
\newblock Schur dynamics of the {S}chur processes.
\newblock {\em Adv. Math.}, 228(4):2268--2291, 2011.

\bibitem[BP14]{BorodimPetrov}
Alexei Borodin and Leonid Petrov.
\newblock Integrable probability: from representation theory to {M}acdonald
  processes.
\newblock {\em Probab. Surv.}, 11:1--58, 2014.

\bibitem[BP25]{Bertucci2hal}
Charles Bertucci and Valentin Pesce.
\newblock {A new approach for the unitary Dyson Brownian motion through the
  theory of viscosity solutions}.
\newblock Working Paper or Preprint, 2025.

\bibitem[C\'95]{Cepa95}
Emmanuel C\'{e}pa.
\newblock \'{E}quations diff\'{e}rentielles stochastiques multivoques.
\newblock In {\em S\'{e}minaire de {P}robabilit\'{e}s, {XXIX}}, volume 1613 of
  {\em Lecture Notes in Math.}, pages 86--107. Springer, Berlin, 1995.

\bibitem[CDMG10]{Contreras}
Manuel~D. Contreras, Santiago D\'iaz-Madrigal, and Pavel Gumenyuk.
\newblock Loewner chains in the unit disk.
\newblock {\em Rev. Mat. Iberoam.}, 26(3):975--1012, 2010.

\bibitem[CFP12]{Ferreira2}
Jos\'e{}~A. Carrillo, Lucas C.~F. Ferreira, and Juliana~C. Precioso.
\newblock A mass-transportation approach to a one dimensional fluid mechanics
  model with nonlocal velocity.
\newblock {\em Adv. Math.}, 231(1):306--327, 2012.

\bibitem[CL97]{CepaLepingle97}
Emmanuel C\'{e}pa and Dominique L\'{e}pingle.
\newblock Diffusing particles with electrostatic repulsion.
\newblock {\em Probab. Theory Related Fields}, 107(4):429--449, 1997.

\bibitem[CL01]{CepaLepingle01}
Emmanuel C\'{e}pa and Dominique L\'{e}pingle.
\newblock Brownian particles with electrostatic repulsion on the circle:
  {D}yson's model for unitary random matrices revisited.
\newblock {\em ESAIM Probab. Statist.}, 5:203--224, 2001.

\bibitem[Col03]{Collins}
Beno\^it Collins.
\newblock Moments and cumulants of polynomial random variables on unitary
  groups, the {I}tzykson-{Z}uber integral, and free probability.
\newblock {\em Int. Math. Res. Not.}, (17):953--982, 2003.

\bibitem[DH12]{Nizar}
Nizar Demni and Taoufik Hmidi.
\newblock Spectral distribution of the free unitary {B}rownian motion: another
  approach.
\newblock In {\em S\'eminaire de {P}robabilit\'es {XLIV}}, volume 2046 of {\em
  Lecture Notes in Math.}, pages 191--206. Springer, Heidelberg, 2012.

\bibitem[DKM24]{Mallick}
Rahul Dandekar, P.~L. Krapivsky, and Kirone Mallick.
\newblock Current fluctuations in the {D}yson gas.
\newblock {\em Phys. Rev. E}, 110(6):Paper No. 064153, 16, 2024.

\bibitem[DO24]{Duboux}
Thibaut Duboux and Yoann Offret.
\newblock {Maximum Entropy Random Walks: the Infinite Setting and the Example
  of Spider Networks with their Scaling Limits}.
\newblock Working Paper or Preprint, 2024.

\bibitem[DO25]{Dovgal}
Sergey Dovgal and Yoann Offret.
\newblock Maximal entropy random walks and central {M}arkov chains.
\newblock Preprint available at \url{https://arxiv.org/abs/2503.08172}, 2025.

\bibitem[Dud12]{DudaPhd}
Jarosław Duda.
\newblock {\em Extended Maximal Entropy Random Walk. Uniwersytet Jagielloński.
  Instytut Fizyki}.
\newblock PhD thesis, Kraków, Nov 2012.

\bibitem[Dys62]{Dyson}
Freeman~J. Dyson.
\newblock A {B}rownian-motion model for the eigenvalues of a random matrix.
\newblock {\em J. Mathematical Phys.}, 3:1191--1198, 1962.

\bibitem[EK86]{EK}
Stewart~N. Ethier and Thomas~G. Kurtz.
\newblock {\em Markov processes}.
\newblock Wiley Series in Probability and Mathematical Statistics: Probability
  and Mathematical Statistics. John Wiley \& Sons, Inc., New York, 1986.
\newblock Characterization and convergence.

\bibitem[FGcN13]{Goncalvez}
Tertuliano Franco, Patr\'icia Gon\c~calves, and Adriana Neumann.
\newblock Hydrodynamical behavior of symmetric exclusion with slow bonds.
\newblock {\em Ann. Inst. Henri Poincar\'e{} Probab. Stat.}, 49(2):402--427,
  2013.

\bibitem[For98]{ForestCalo}
Peter~J. Forrester.
\newblock Random matrices, log-gases and the {C}alogero-{S}utherland model.
\newblock In {\em Quantum many-body problems and representation theory},
  volume~1 of {\em MSJ Mem.}, pages 97--181. Math. Soc. Japan, Tokyo, 1998.

\bibitem[For10]{Forest}
P.~J. Forrester.
\newblock {\em Log-gases and random matrices}, volume~34 of {\em London
  Mathematical Society Monographs Series}.
\newblock Princeton University Press, Princeton, NJ, 2010.

\bibitem[FS09]{Flajolet09}
Philippe Flajolet and Robert Sedgewick.
\newblock {\em Analytic Combinatorics}.
\newblock Cambridge University Press, 2009.

\bibitem[FS16]{SimonF}
Max Fathi and Marielle Simon.
\newblock The gradient flow approach to hydrodynamic limits for the simple
  exclusion process.
\newblock In {\em From particle systems to partial differential equations.
  {III}}, volume 162 of {\em Springer Proc. Math. Stat.}, pages 167--184.
  Springer, [Cham], 2016.

\bibitem[Ful97]{fulton1997young}
William Fulton.
\newblock {\em Young Tableaux: With Applications to Representation Theory and
  Geometry}, volume~35 of {\em London Mathematical Society Student Texts}.
\newblock Cambridge University Press, Cambridge, 1997.

\bibitem[FVG16]{Ferreira1}
Lucas C.~F. Ferreira and Julio~C. Valencia-Guevara.
\newblock Periodic solutions for a 1{D}-model with nonlocal velocity via mass
  transport.
\newblock {\em J. Differential Equations}, 260(10):7093--7114, 2016.

\bibitem[GGPN14]{Goulden}
I.~P. Goulden, Mathieu Guay-Paquet, and Jonathan Novak.
\newblock Monotone {H}urwitz numbers and the {HCIZ} integral.
\newblock {\em Ann. Math. Blaise Pascal}, 21(1):71--89, 2014.

\bibitem[GLBM23]{Guillin}
Arnaud Guillin, Pierre Le~Bris, and Pierre Monmarch\'e.
\newblock On systems of particles in singular repulsive interaction in
  dimension one: log and {R}iesz gas.
\newblock {\em J. \'Ec. polytech. Math.}, 10:867--916, 2023.

\bibitem[GLT23]{GLT2023}
J{\'e}r{\'e}mie Guilhot, C{\'e}dric Lecouvey, and Pierre Tarrago.
\newblock Quantum cohomology of the grassmannian and unitary dyson brownian
  motion.
\newblock {\em arXiv preprint arXiv:2211.12836}, 2023.
\newblock Submitted on 23 Nov 2022, last revised 12 May 2023.

\bibitem[GM05]{GuionnetChar}
Alice Guionnet and Myl\`ene Ma\"ida.
\newblock Character expansion method for the first order asymptotics of a
  matrix integral.
\newblock {\em Probab. Theory Related Fields}, 132(4):539--578, 2005.

\bibitem[GS15]{Gorin}
Vadim Gorin and Mykhaylo Shkolnikov.
\newblock Multilevel {D}yson {B}rownian motions via {J}ack polynomials.
\newblock {\em Probab. Theory Related Fields}, 163(3-4):413--463, 2015.

\bibitem[Gui04]{Guionnet2}
Alice Guionnet.
\newblock First order asymptotics of matrix integrals; a rigorous approach
  towards the understanding of matrix models.
\newblock {\em Comm. Math. Phys.}, 244(3):527--569, 2004.

\bibitem[GW09]{Gangbo}
Wilfrid Gangbo and Michael Westdickenberg.
\newblock Optimal transport for the system of isentropic {E}uler equations.
\newblock {\em Comm. Partial Differential Equations}, 34(7-9):1041--1073, 2009.

\bibitem[Ham18]{Hamdi}
Tarek Hamdi.
\newblock Spectral distribution of the free {J}acobi process, revisited.
\newblock {\em Anal. PDE}, 11(8):2137--2148, 2018.

\bibitem[HW96]{Werner}
David~G. Hobson and Wendelin Werner.
\newblock Non-colliding {B}rownian motions on the circle.
\newblock {\em Bull. London Math. Soc.}, 28(6):643--650, 1996.

\bibitem[Ker96]{kerov}
S.~Kerov.
\newblock The boundary of {Y}oung lattice and random {Y}oung tableaux.
\newblock In {\em Formal power series and algebraic combinatorics ({N}ew
  {B}runswick, {NJ}, 1994)}, volume~24 of {\em DIMACS Ser. Discrete Math.
  Theoret. Comput. Sci.}, pages 133--158. Amer. Math. Soc., Providence, RI,
  1996.

\bibitem[KL99]{KipnisLandim}
Claude Kipnis and Claudio Landim.
\newblock {\em Scaling limits of interacting particle systems}, volume 320 of
  {\em Grundlehren der mathematischen Wissenschaften [Fundamental Principles of
  Mathematical Sciences]}.
\newblock Springer-Verlag, Berlin, 1999.

\bibitem[Koi89]{Koike}
Kazuhiko Koike.
\newblock On the decomposition of tensor products of the representations of the
  classical groups: by means of the universal characters.
\newblock {\em Adv. Math.}, 74(1):57--86, 1989.

\bibitem[LW16]{WangNCB}
Karl Liechty and Dong Wang.
\newblock Nonintersecting {B}rownian motions on the unit circle.
\newblock {\em Ann. Probab.}, 44(2):1134--1211, 2016.

\bibitem[Mac15]{Macdonald}
I.~G. Macdonald.
\newblock {\em Symmetric functions and {H}all polynomials}.
\newblock Oxford Classic Texts in the Physical Sciences. The Clarendon Press,
  Oxford University Press, New York, second edition, 2015.
\newblock With contribution by A. V. Zelevinsky and a foreword by Richard
  Stanley.

\bibitem[Mat94]{Matytsin}
A.~Matytsin.
\newblock On the large-{$N$} limit of the {I}tzykson-{Z}uber integral.
\newblock {\em Nuclear Phys. B}, 411(2-3):805--820, 1994.

\bibitem[Men17]{Menon}
Govind Menon.
\newblock The complex burgers equation, the hciz integral and the
  calogero-moser system.
\newblock 2017.

\bibitem[MS25]{Marcolli25}
Matilde Marcolli and David Skigin.
\newblock Merge on workspaces as hopf algebra markov chain.
\newblock {\em arXiv preprint}, 2025.

\bibitem[Nov20]{Novak}
Jonathan Novak.
\newblock On the complex asymptotics of the {HCIZ} and {BGW} integrals.
\newblock {\em arXiv preprint}, 2020.
\newblock Preliminary version.

\bibitem[Pet12]{Petrov}
I.G. Petrovsky.
\newblock {\em Lectures on Partial Differential Equations}.
\newblock Dover Books on Mathematics. Dover Publications, 2012.

\bibitem[Pom92]{Pommerenke}
Ch. Pommerenke.
\newblock {\em Boundary behaviour of conformal maps}, volume 299 of {\em
  Grundlehren der mathematischen Wissenschaften [Fundamental Principles of
  Mathematical Sciences]}.
\newblock Springer-Verlag, Berlin, 1992.

\bibitem[RS15]{Romik}
Dan Romik and Piotr Sniady.
\newblock {Jeu de taquin dynamics on infinite Young tableaux and second class
  particles}.
\newblock {\em The Annals of Probability}, 43(2):682 -- 737, 2015.

\bibitem[Ser19]{Serfaty}
Sylvia Serfaty.
\newblock Microscopic description of {L}og and {C}oulomb gases.
\newblock In {\em Random matrices}, volume~26 of {\em IAS/Park City Math.
  Ser.}, pages 341--387. Amer. Math. Soc., Providence, RI, 2019.

\bibitem[Sta99]{stanley1999enumerative}
Richard~P. Stanley.
\newblock {\em Enumerative Combinatorics, Volume 2}, volume~62 of {\em
  Cambridge Studies in Advanced Mathematics}.
\newblock Cambridge University Press, Cambridge, 1999.

\bibitem[Ste87]{Stembridge}
John~R. Stembridge.
\newblock Rational tableaux and the tensor algebra of {${\rm gl}_n$}.
\newblock {\em J. Combin. Theory Ser. A}, 46(1):79--120, 1987.

\bibitem[VJ67]{VJ}
D.~Vere-Jones.
\newblock Ergodic properties of nonnegative matrices. {I}.
\newblock {\em Pacific J. Math.}, 22:361--386, 1967.

\bibitem[Wan10]{Coulomb}
Tower Wang.
\newblock Coulomb force as an entropic force.
\newblock {\em Phys. Rev. D}, 81:104045, May 2010.

\bibitem[Won01]{Wong}
R.~Wong.
\newblock {\em Asymptotic Approximations of Integrals}.
\newblock Society for Industrial and Applied Mathematics, 2001.

\end{thebibliography}


\end{document}